\documentclass[11pt]{amsart}

\usepackage{amsmath,amssymb,amscd,amsthm,amsxtra}

\usepackage{graphicx, mathrsfs}
\usepackage{amssymb,amsmath,amsthm}
\usepackage{mathrsfs,dsfont}

\makeatletter
\DeclareRobustCommand\widecheck[1]{{\mathpalette\@widecheck{#1}}}
\def\@widecheck#1#2{%
   \setbox\z@\hbox{\m@th$#1#2$}%
   \setbox\tw@\hbox{\m@th$#1%
      \widehat{%
         \vrule\@width\z@\@height\ht\z@
         \vrule\@height\z@\@width\wd\z@}$}%
   \dp\tw@-\ht\z@
   \@tempdima\ht\z@ \advance\@tempdima2\ht\tw@ \divide\@tempdima\thr@@
   \setbox\tw@\hbox{%
      \raise\@tempdima\hbox{\scalebox{1}[-1]{\lower\@tempdima\box\tw@}}}%
   {\ooalign{\box\tw@ \cr \box\z@}}}
\makeatother

\allowdisplaybreaks[2]

\newtheorem{theorem}{Theorem} [section]

\newtheorem{lemma}[theorem]{Lemma}
\newtheorem{proposition}[theorem]{Proposition}
\newtheorem{remark}[theorem]{Remark}

\newtheorem{definition}[theorem]{Definition}

\DeclareMathOperator*{\supp}{supp}

\newcommand{\noi}{\noindent}
\newcommand{\Z}{\mathbb{Z}}

\newcommand{\R}{\mathbb{R}}
\newcommand{\C}{\mathbb{C}}
\newcommand{\T}{\mathbb{T}}

\newcommand{\EE}{\mathcal E}
\newcommand{\al}{\alpha}
\newcommand{\dl}{\delta}

\newcommand{\eps}{\varepsilon}

\newcommand{\ld}{\lambda}

\newcommand{\s}{\sigma}
\newcommand{\FF}{\mathcal F} 
\newcommand{\ft}{\widehat}
\newcommand{\wt}{\widetilde}
\newcommand{\cj}{\overline}
\newcommand{\dx}{\partial_x}
\newcommand{\dt}{\frac{d}{dt}}

\newcommand{\LRA}{\Longrightarrow}

\newcommand{\jb}[1]
{\langle #1 \rangle}

\numberwithin{equation}{section}

\begin{document}

\title[Invariant weighted Wiener measure and a.s. GWP for  DNLS ]{
Invariant weighted Wiener measures and  almost sure global well-posedness  for the periodic derivative NLS.}
\author[Nahmod]{Andrea R. Nahmod$^1$}
\address{$^1$ Radcliffe Institute for Advanced Study\\Harvard University\\Byerly Hall, 8 Garden Street\\
Cambridge, MA 02138 and 
Department of Mathematics \\ University of Massachusetts\\  710 N. Pleasant Street, Amherst MA 01003}
\email{nahmod@math.umass.edu}
\thanks{$^1$ The first author is funded in part by NSF DMS 0803160 and a 2009-2010 Radcliffe Institute for Advanced Study Fellowship.}

\author[Oh]{Tadahiro Oh$^2$}
\address{$^2$ Department of Mathematics\\ 40 St. George Street, Toronto, Ontario, Canada M5S 2E4} \email{oh@math.toronto.edu}
\author[Rey-Bellet]{Luc Rey-Bellet$^3$}
\address{$^3$ Department of Mathematics \\
University of Massachusetts\\ 710 N. Pleasant Street, Amherst MA 01003 }
\email{luc@math.umass.edu}
\thanks{$^3$ The third author is funded in part by NSF DMS 0605058}
\author[Staffilani]{Gigliola Staffilani$^4$}
\address{$^4$  Radcliffe Institute for Advanced Study\\
Harvard University\\
Byerly Hall, 8 Garden Street\\
Cambridge, MA 02138
and 
Department of Mathematics\\
Massachusetts Institute of Technology\\ 77 Massachusetts Avenue,  Cambridge, MA 02139}
\email{gigliola@math.mit.edu}
\thanks{$^4$ The fourth author is funded in part by NSF
DMS 0602678 and a 2009-2010 Radcliffe Institute for Advance Study Fellowship.}
\date{}
\begin{abstract}
In this paper we construct an invariant weighted Wiener measure associated to the periodic derivative nonlinear Schr\"odinger equation in one dimension and establish global well-posedness for data living in its support.  In particular almost surely for data in a Fourier-Lebesgue space ${\mathcal F}L^{s,r}(\T)$  with  $s \ge \frac{1}{2}$, $2 < r < 4$,  $(s-1)r <-1$ and scaling like $H^{\frac{1}{2}-\epsilon}(\T),$ for small  $\epsilon >0$.  We also show the invariance of this measure.
\end{abstract}
\maketitle

\section{Introduction}

In the past few years, methods such as those by J. Bourgain (high-low method; e.g. 
\cite{B3, B4}) on the one hand  and by J. Colliander, M. Keel, G. Staffilani, H. Takaoka and 
T. Tao (I-method or method of {\it almost conservation laws} e.g. 
\cite{CKSTT1, CKSTT2,CKSTT3})  on the other, have been 
applied to study the global in time existence of dispersive equations at regularities 
which are right below or in between those corresponding to conserved quantities. 
As it turns out however, for many dispersive equations and systems there still remains a gap between the local in time results 
and those that could be globally achieved. 
In those cases,  it seems natural to return to  one of Bourgain's early approaches for periodic dispersive equations (NLS, KdV, mKdV, Zakharov system) \cite{B1, B2, B3, B5, B6, B7} 
where global in time existence was studied in the almost sure sense via the existence and
invariance of the associated  Gibbs measure (cf. Lebowitz, Rose and Speer's and  
Zhidkov's works  \cite{LRS} \cite{Z}).  More recently this approach has been used for example
by N. Tzvetkov \cite{Tzv1, Tzv2} for subquintic radial nonlinear wave equation on the disc, 
N. Burq and N. Tzvetkov  \cite{BT1, BT3} for subcubic and subquartic radial 
nonlinear wave equations on 3d ball, N. Burq, L. Thomann, and N. Tzvetkov \cite{BTT} for the 
nonlinear Schr\"odinger equation with harmonic potential,  and by T. Oh 
\cite{Oh1, Oh2, Oh3, Oh4} for the periodic KdV-type coupled systems, 
Schr\"odinger-Benjamin-Ono system and white noise for the KdV equation.

Failure to show global existence by Bourgain's high-low method or the I-method might 
come from certain `exceptional'  initial data set, and the virtue of the Gibbs 
measure is that it does not see that exceptional set.  At the same time, the invariance 
of the Gibbs measure, just like the usual conserved quantities, can be used to control 
the growth in time of those solutions in its support and extend the local in time solutions 
to global ones almost surely. The difficulty in this approach lies in the actual construction of 
the associated Gibbs measure and in showing both its invariance under the flow and 
the almost sure global well-posedness, since, on the one hand, we need invariance to 
show global well-posedness and on the other hand we need globally defined flow to 
discuss invariance.

Our goal in this paper is to construct an invariant weighted Wiener measure associated to 
the periodic derivative nonlinear Schr\"odinger equation DNLS in \eqref{DNLS} in one
dimension and establish global well-posedness for data living in its support.  In particular
almost surely for data in a Fourier-Lebesgue space ${\mathcal F}L^{s,r}$ defined in
\eqref{fourierlebesgue} below (c.f. \cite{Hormander, Gr, Christ, GH}) and scaling like
$H^{\frac{1}{2}-\epsilon}(\T),$ for small  $\epsilon>0$.  The motivation for this paper stems 
from the fact that by scaling DNLS should be well posed for data in $H^\sigma, \, \sigma \ge 0$
but the results so far obtained are much weaker. 

Local well-posedness is known  for $\sigma \ge 1/2$ for the nonperiodic \cite{Tak1} 
and periodic \cite{Herr} cases while 
global well-posednes is known for $\sigma \ge 1/2$ for the nonperiodic case ($\sigma >1/2$
in \cite{CKSTT2} and $\sigma\ge 1/2$ in \cite {MWX}) and for $\sigma > 1/2$ in 
the periodic case \cite{YYW}. 
Furthermore, in the non periodic case the Cauchy initial value problem for DNLS is ill-posed 
for data in $H^{\sigma}(\R),  \sigma< \frac{1}{2}$   \cite{Tak1}  \cite{BiaLin}, a strong indication
that ill-posedness should  also be expected in the periodic case on that range. Gr\"unrock 
and Herr \cite{GH} have
recently established local well posedness for the periodic DNLS in Fourier-Lebesgue spaces 
${\mathcal F}L^{s,r},$ which for appropriate choices of $(s,r)$ scale like $H^{\sigma}(\T)$ for 
any $\sigma>\frac{1}{4}$. Their result is the starting point of this work (cf. Section 2 for a 
more detailed discussion).

The measure we construct is based on the energy functional rather than the Hamiltonian.
Hence we simply refer to it as weighted Wiener measure rather than Gibbs measure since 
the name `Gibbs measure' has traditionally been reserved for those weighted Wiener 
measures constructed using the Hamiltonian.  By invariance of a measure $\mu$ we mean 
that  if  $\Phi(t)$ denote the flow map associated to our nonlinear equation then $\Phi(t)$ is 
defined for all $t \in \R$, $\mu$ almost surely and for all $f \in L^1(\mu)$ and all $t\in \R$, 
$$\int f( \Phi(t) (\phi)) \,  \mu(d\phi) \, =\, \int f(\phi)   \mu(d\phi).$$ 

In general terms our aim is to construct a well defined measure $\mu$ so that local well posedness of the periodic DNLS holds in some space $\mathcal B$ containing the 
support of $\mu$. 
Then we show almost sure global well posedness as well as the invariance of $\mu$ via 
a combination of the methods of  Bourgain and  Zhidkov \cite{Z} in the context of NLS, KdV,
mKdV.  In implementing this scheme however  we need to overcome two main obstacles due 
to the need to gauge the equation to show local well posedness (eg. \cite{Tak1, Herr}) and 
to construct an invariant measure.  The symplectic form associated to the periodic 
gauged derivative nonlinear Schr\"odinger equation GDNLS in \eqref{GDNLS} does not 
commute with Fourier modes truncation and so the truncated finite-dimensional
systems are not necessarily Hamiltonian. 
The first (mild) obstacle is to show the conservation of the Lebesgue measure associated to 
the finite dimensional approximation to the periodic gauged  derivative nonlinear 
Schr\"odinger equation FGDNLS, defined in \eqref{FGDNLS} by hand, rather than 
by using the Hamiltonian structure.  The second obstacle  is much more serious and is 
at the heart of this work. The energy ${\EE}$ defined in \eqref{eu} associated 
to the gauged periodic DNLS{\footnote{We emphasize $\EE$ is not the Hamiltonian of the gauged DNLS.}} which we prove to be conserved in time, ceases to be so when computed 
on solutions of the finite dimensional approximation equation; that is $\frac{d}{dt}\EE(v^N) 
\neq 0, \mbox{ when }  v^N$ is a solution to the finite dimensional gauged DNLS (see
\eqref{EE3}).  In other words the finite dimensional weighted Wiener measure is not invariant
any longer  and unlike Zhidkov's work \cite{Z} on KdV we do not have a priori knowledge of
global well posedness. We  show however that it is {\it almost} invariant in the sense that we
can control the growth in time of $\EE(v^N) (t)$.  This idea is reminiscent of the $I$-method. 
However, while in the $I$-method one needs to estimate the variation of the energy of solutions
to the infinite dimensional equation at time $t$ smoothly projected onto frequencies of size up 
to $N$; here one needs to control the variation of the energy $\EE$ of the solution $v^{N}$ 
to the finite dimensional approximation equation FGDNLS.   We note that the loss in energy 
conservation for solutions to the finite dimensional equation is principally due to the manner
one chooses to approximate the infinite dimensional gauged equation by using Fourier
projections onto the first $N$th frequencies. In \cite{B1} Bourgain describes an alternative
approach that relies on using a discrete system of ODE which seems to preserve the
conservation of energy. This approach however entails a number of other difficulties, for one
needs to replace the circle $\T$ by the cyclic group $\Z_N$ and carry out the analysis on cyclic
groups. We choose not to follow this path here.

We expect the ungauged invariant Wiener measure associated to DNLS \eqref{DNLS} 
we obtain in Section 7 to be absolutely continuous with respect to the  weighted 
Wiener measure constructed by Thomann and Tzvetkov \cite{TTzv}. This question is 
addressed in a forthcoming paper \cite{NRBSS}. 

\medskip

The paper is organized as follows. In Section 2 we present some general background, 
notation and results on the derivative nonlinear Schr\"odinger equation in one dimension.  
In Section 3  we discuss  FGDNLS.  In Section 4 we overcome the first two obstacles mentioned
above. Namely we prove the invariance of the Lebesgue measure associated to FGDNLS and
devote the rest of the section to prove our energy growth estimate Theorem 
\ref{Energy Growth Estimate}. In Section 5 we carry out the construction of the weighted 
Wiener measure associated to the GDNLS. In Section 6 we prove the almost sure global 
well-posedness result for the GDNLS and the invariance of the measure constructed in 
section 5. Finally in Section 7 we  translate back our results to the 
ungauged DNLS equation.

\medskip

{\bf Acknowledgment.}  Andrea R. Nahmod and Gigliola Staffilani would like to warmly thank 
the Radcliffe Institute for Advanced Study at Harvard University for its wonderful hospitality
while part of this work was being carried out. They also thank their fellow Fellows for the 
stimulating environment they created.

\medskip

{\bf Notation.}  Whenever we write $a+$ for $a \in \R$ we mean $a+\varepsilon$ for some $\varepsilon >0$; similarly for $a-$. In addition, we write $A \lesssim B$ to mean there exist some absolute constant $C>0$ such that $ A \leq C B$. 

\medskip

\section{The Derivative NLS Equation in one dimension}

The initial value problem for DNLS takes the form:

\begin{equation} \label{DNLS}
\begin{cases} u_t  \, - i \,  u_{xx}  \, = \, \lambda (|u|^2 u)_x \\
u\big|_{t = 0} = u_0, 
\end{cases}
\end{equation}
where either $(x, t) \in \mathbb{R}\times (-T, T)$ or $ (x, t) \in \mathbb{T}\times (-T, T) $  and  $\lambda$ is  real. In this paper we will take $\lambda=1$ for convenience. 
 DNLS is a Hamiltonian PDE whose flow conserves also mass and energy; i.e. the following are conserved quantities of time\footnote{In fact, DNLS is completely integrable.} (c.f. \cite{KaupNew, HayOz2, Herr}) :
 
 \begin{eqnarray*}
 \mbox{Mass:}  \qquad {M(u)(t)} &=& \int  |u(x,t)|^2 \, dx.   \\
\mbox{ Energy:}  \qquad {E(u)(t)} &=& \int | u_x|^2 \, dx + \frac{3}{2}  {\rm Im} \int  u^2 {\cj {u}\, \cj{u_x}} \, dx + \frac{1}{2} \int |u|^6 \, dx. \\
\mbox{ Hamiltonian:} \qquad {H(u)(t)} &=& {\rm Im} \int u {\cj u}_x \, dx  + \frac{1}{2} \int  |u|^4 \, dx.  
\end{eqnarray*}
DNLS was introduced as a model for the propagation of circularly polarized Alfv\'en
waves in a magnetized plasma with a constant magnetic field 
(cf. Sulem-Sulem \cite{SS}).
The equation is scale invariant for data in $L^2$; {\it i.e.} if $u(x,t)$ is a solution then 
$u_a(x,t) = a^{\alpha} u(a x, a^2 t)$ is also a solution if and only if $\alpha= \frac{1}{2}$. Thus {\it a priori} one expects some form of existence and uniqueness 
results for \eqref{DNLS} for data in $H^{\sigma}, \sigma \geq 0$.  Many results are known for 
the Cauchy problem with smooth data, including data in $H^1$,  such as those by  
M. Tsutsumi and I. Fukuda \cite{TF}, N.Hayashi \cite{Hay}, N. Hayashi and T. Ozawa 
\cite{HayOz1, HayOz2} and T. Ozawa \cite{Oz} and others (cf. references therein). 
   
In looking for solutions to \eqref{DNLS} we face a derivative loss arising from the nonlinear 
term $ (|u|^2 u)_x \, = \,  u^2\,  \overline{u}_x \, +\,  2\, { |u|^2 \, u_x } $  and hence for 
low regularity data the key is to somehow make up for this loss.

For the non-periodic case  ($x \in \R$)  Takaoka \cite{Tak1} proved sharp local 
well-posedness (LWP) in $H^{\frac{1}{2}}(\R)$ relying on the gauge transformation 
used by Hayashi and Ozawa \cite{HayOz1,HayOz2}  and the so-called Fourier restriction 
norm method.  Then,  Colliander-Keel-Staffilani-Takaoka and Tao \cite{CKSTT1, CKSTT2}
established global well-posedness (GWP) in $H^{\sigma}(\R)$, $\sigma>\frac{1}{2}$ of 
small $L^2$ norm using the so-called I-method on the gauge equivalent equation 
(see also \cite{Tak2}).  Here, small in $L^2$ just means less than an appropriate 
constant ~ $\sqrt{ \frac{2 \pi}{\lambda}}$ which forces the associated `energy' to be positive 
via Gagliardo-Nirenberg inequality. This result was recently improved by Miao, Wu and Xu 
to $\sigma \ge 1/2$.  The Cauchy initial value problem for DNLS is ill-posed 
for data in $H^{\sigma}$  and $\sigma < \frac{1}{2}$ (data map fails to be $C^3$ or 
uniformly $C^0$    \cite{Tak1}  \cite{BiaLin}.)  In \cite{Gr} A. Gr\"unrock proved that the 
non-periodic DNLS is locally well posed in the Fourier-Lebesgue spaces  
${\mathcal F}L^{s,r}(\R)$ which for appropriate choices of $(s,r)$ scale like $H^{\sigma}(\R)$ 
for any $\sigma>0$ (c.f. \eqref{fourierlebesgue} below.) 

In the periodic setting,  S. Herr  \cite{Herr}  showed that the Cauchy problem associated 
to periodic  DNLS  is locally well-posed for initial data $u(0) \in H^{\sigma} (\T)$, if 
$\sigma \geq \frac{1}{2}$ in the sense of local existence, uniqueness and continuity of the 
flow map.    Herr's proof is based on an adaptation to the periodic setting of the 
gauge transformation introduced by Hayashi \cite{Hay} Hayashi and Ozawa 
\cite{HayOz1, HayOz2} on $\R$, in conjunction with sharp multilinear estimates for the 
gauged equivalent equation in periodic Fourier restriction norm spaces $X^{s,b}$ that 
yield local well posedness for the gauged equation. Moreover, by use of conservation laws, 
the problem is also shown to be globally well-posed for $\sigma \geq 1$ and data which 
is small in $L^2$-as in \cite{CKSTT1, CKSTT2}- \cite{Herr}.   More recently,  Win 
\cite{YYW} applied the I-method to prove GWP in $H^{\sigma}(\T)$ for $\sigma > \frac{1}{2}$.

 A. Gr\"unrock and S. Herr  \cite{GH} showed that the Cauchy problem associated to (DNLS)  
 is locally well-posed for initial data 
$u_0 \in {\mathcal F} L^{s, r}(\T)$ with  $2 < r < 4$ and  $s \geq \frac{1}{2}$ where
\begin{equation}\label{fourierlebesgue} \Vert u_0 \Vert_{ {\mathcal F} L^{s, r}(\T) } \, :=   \Vert \, \langle n  \rangle^s \, \widehat u_0\, \Vert_{\ell^r_n(\Z)} .  \end{equation}

These spaces scale like the Sobolev $H^{\sigma}(\T)$ ones where $\sigma =  s + \frac{1}{r}- \frac{1}{2} $. Their proof is based on Herr's adapted periodic gauge transformation and new multilinear estimates for the gauged equivalent equation in an appropriate variant  of Fourier restriction norm spaces $X_{r,q}^{s,b}$ introduced by Gr\"unrock-Herr \cite{GH}\footnote{Note that in our notation the indices $(r,q)$ are the dual of the corresponding ones in  Gr\"unrock-Herr \cite{GH}}. 

For $s, b \in \R$, $  r, q \geq 1$ we define the space $X^{s,b}_{r,q}$ as the completion of the Schwartz space $ {\mathcal S}( \T \times \R)$ with respect to the norm

$$\Vert u\Vert_{X^{s,b}_{r,q}} := \Vert \langle n \rangle^s \,  \langle \tau + n^2 \rangle^b {\widehat u}(n, \tau) \Vert_{\ell^r_n L^q_{\tau}} $$ where first we take the $ L^q_{\tau}$ norm and then the ${\ell^r_n}$ one.  We also define the space
$$
\|u \|_{X^{s, b}_{r,q; -}}  := \| \jb{n}^s \jb{\tau-n^2}^b \ft{u}(n, \tau) \|_{l^r_n L^q_\tau},
$$
and note that  $u\in {X^{s, b}_{r,q}}$ if and only if ${\cj u} \in {X^{s, b}_{r, q;-}}$.

For $\delta>0$ fixed, we define the restriction space $X^{s, b}_{r,q}(\delta)$ of all $v=u \left |_{[-\delta,\delta]} \right.$ for some $u \in X^{s, b}_{r,q}$ with norm
\begin{equation}\label{localspace}
\|v\|_{X^{s, b}_{r,q}(\delta)}:=\inf \{\|u\|_{X^{s, b}_{r,q}} \, \, : \, \, u \in X^{s, b}_{r,q} \, \, \mbox{ and }\, \, v=u \left |_{[-\delta,\delta]}\right.\}.
\end{equation}
When we take $q=2$ we simply write $X^{s,b}_{r,2} \, = \, X^{s,b}_{r}$. Note $X^{s,b}_{2,2} \, = \, X^{s,b}$. Later we will also use the space
\begin{equation}\label{zspace}
Z^{s}_r(\delta)\, := \, X^{s,\frac{1}{2}}_{r,2} (\delta)\cap X^{s,0}_{r,1}(\delta).
\end{equation}

Some simple embeddings are as follows.   For $s, b_1, b_2 \in \R$, $r \ge 1$ and $b_1 >  b_2+ \frac{1}{2}$ 
$$ X^{s,b_1}_{r, 2}   \subset X^{s,b_2}_{r, 1} \qquad \text{ and } \qquad  X^{s,0}_{r, 1} \subset C(\R, {\mathcal F} L^{s,r}) $$
which follow by Cauchy-Schwarz with respect to the $L^1_{\tau}$ norm and by ${\mathcal F}^{-1} L^1\subset L^{\infty} $ respectively. In particular
$$Z^{s}_r(\delta)  \subset C([-\delta, \delta], {\mathcal F} L^{s,r}). $$

We finally recall the following estimate\footnote{This is a trilinear refinement of Bourgain's $L^6(\T)$ Strichartz estimate \cite{B8}.} heavily used in the proof of Theorem \ref{Energy Growth Estimate} below.

\begin{lemma}[Lemma 5.1 \cite{GH}] Let $\frac{1}{3}<b<\frac{1}{2}$ and $s> 3\left(\frac{1}{2}-b\right)$. 
Then 
$$
 \Vert u v {\overline w} \Vert_{L^2_{xt}} \lesssim \Vert  u \Vert_{X^{s,b}}  \Vert  v \Vert_{X^{s, b}}   \Vert  w \Vert_{X^{0, b}}. 
$$

In particular if $b=\frac{1}{2}-$, then 
 \begin{equation}\label{iii}
 \Vert u v {\overline w} \Vert_{L^2_{xt}} \lesssim \Vert  u \Vert_{X^{\epsilon, \frac{1}{2}-}}  \Vert  v \Vert_{X^{\epsilon, \frac{1}{2}-}}   \Vert  w \Vert_{X^{0, \frac{1}{2}-}},
 \end{equation} 
 for small $\epsilon>0$;  while when $b=\frac{1}{3}+$
  \begin{equation}\label{iv}
 \Vert u v {\overline w} \Vert_{L^2_{xt}} \lesssim \Vert  u \Vert_{X^{ \frac{1}{2}-, \frac{1}{3}+}}  \Vert  v \Vert_{X^{\frac{1}{2}-, \frac{1}{3}+}}   \Vert  w \Vert_{X^{0, \frac{1}{3}+}}.
 \end{equation} 

\end{lemma}

\bigskip

\subsection{The Periodic Gauged Derivative NLS Equation}

We first  recall S. Herr's gauge transformation. 
For $f\in L^2(\T)$, let $${G}(f)(x):= \exp(-i {J}(f)) \, f(x)$$ where 
\begin{equation}\label{expgauge} 
J(f)(x):=  \frac{1}{2 \pi} \int_0^{2\pi} \int_{\theta}^x |f(y)|^2 - \frac{1}{2 \pi} \Vert f \Vert^2_{L^2(\T)} \, dy \, d\theta. 
\end{equation} 
Note $G(f)$ is $2\pi$-periodic since its integrand has zero mean value. Then for  $u \in C([-T, T]; L^2(\T))$ and  $m(u):=\frac{1}{2 \pi} \int_{\T} |u(x,0)|^2 dx $ the adapted periodic gauge is defined as\footnote{Recall $m(u)(t)$ is conserved under the flow of \eqref{DNLS}.}
 $$ \mathcal G(u)(t,x) : = {G}(u(t))(x - 2 \,t \,m(u)).$$  Note the $L^2$ norm of $ \mathcal G(u)(t,x)$ is still conserved since the torus is invariant under translation.

\noi We have that $$ \mathcal G :  C([-T, T]; H^{\sigma}(\T)) \to   C([-T, T]; H^{\sigma}(\T))$$ is a homeomorphism for any $\sigma \geq 0$ and locally bi-Lipschitz on subsets of   
$C([-T, T]; H^{\sigma}(\T))$ with prescribed $\Vert u(0)\Vert_{L^2}$(\cite{Herr}).  Moreover the same is true if we replace $H^{\sigma}(\T)$ by $\FF L^{s,r}$ with $s > \frac{1}{2} -\frac{1}{r}$ when $2< r <\infty$ and $s \ge 0$ when $r=2$ (\cite{GH}).

Then if $u$ is a solution to DNLS \eqref{DNLS} and $v := \mathcal G(u)$ we have that $v$ solves the gauged DNLS equation (GDNLS):

\begin{equation}\label{GDNLS} v_t - i v_{xx} = - v^2 {\cj v}_x + \frac{i}{2} |v|^4 v  - i  \psi(v) v - i m(v) |v|^2 v \end{equation}
with initial data $v(0)= \mathcal G(u(0))$, where 
\begin{eqnarray} 
 m(v)(t)&:= & \frac{1}{2 \pi} \int_{\T} |v(x,t)|^2 dx  \qquad \qquad \mbox{ and } \label{defm} \\
\psi(v)(t)&:=& -\frac{1}{\pi} \int_{\T}  {\rm Im}  ( v {\cj v}_x ) \, dx \, + \frac{1}{4 \pi} \int_{\T} | v|^4 dx - m(v)^2  \label{defpsi}.
\end{eqnarray}
Note that $m(v)$ is  conserved in time; more precisely $ m(v)(t) = \frac{1}{2 \pi} \int_{\T} |v(x,0)|^2 dx =m(u)$  and that both $m(v)$ and $\psi(v)$ are real. 

 The initial value problem associated to \eqref{GDNLS}  with data in $\FF L^{s,r}(\T)$ is locally well-posed  in $Z^{s}_r(\delta), \, 2<r<4, \, s\geq \frac{1}{2}$, for some $\delta>0$. This was proved in Theorem 7.2 of \cite{GH}.

\begin{remark}\rm  Local  well-posedness for (GDNLS) \eqref{GDNLS}  implies local existence,
uniqueness and continuity of the flow map for DNLS \eqref{DNLS} \cite{Herr, GH}.  One cannot
however carry back  to solutions to DNLS all the auxiliary estimates coming from the local well
posedness result for GDNLS. 
\end{remark}

Now we  show how the energy $E(u)$ and $H(u)$ transform under the gauge. Let $u$ be 
the solution to  (DNLS) \eqref{DNLS} and define 
$$
w = e^{-iJ(u)} u.
$$
Then $w$ solves   (GDNLS) \eqref{GDNLS}  with the extra $ m(w) w_x$ term in the linear 
part of the equation \cite{Herr}.   So  the gauge transform is, properly speaking the 
transformation $w = e^{-iJ(u)} u$ followed by the transformation
$$
v(x,t) = w(t, x- 2m(w) t)
$$
But all the terms involved in the conserved quantities we considered are invariant under 
this second  transformation  $w\to v$ (the torus is invariant under translation).  We also 
notice that $m(u)=m(w)=m(v)$, hence below we will
be simply using $m$ for this quantity.

Since 
$$
u = e^{iJ(w)} w 
$$
we have
$$
u_x = e^{iJ(w)}( w_x +  i J(w)_x w) 
$$
with $J(w)_x =  |w|^2 - m$. 

We have 
\begin{eqnarray}
H(u)  &=& {\rm Im} \int_{\T} u {\cj u}_x \, dx  + \frac{1}{2} \int_{\T} |u|^4 \, dx.  \nonumber \\
 &=& {\rm Im} \int_{\T} w \left( {\cj w}_x - i J(w)_x \cj{w} \right) \, dx  + \frac{1}{2} \int_{\T} |w|^4 \, dx.  \nonumber \\
&=&  {\rm Im} \int_{\T} w  {\cj w}_x   - \frac{1}{2} \int_{\T} |w|^4 \, dx  + 2\pi m^2=:\mathscr H(w) \label{oo4}
\end{eqnarray}

In addition we have 
 \begin{eqnarray}
u_x \cj{u_x}  \,&=&\,   \left( {w_x} + i J(w)_x w \right)    \left( \cj{w_x} - i J(w)_x \cj{w} \right)  \nonumber  \\
                       \,&=&\,  w_x \cj{w_x}   +i J(w)_x (w \cj{w_x} - \cj{w} w_x )   + J(w)_x^2 |w|^2 \nonumber \\   
                       \,&=&\,  w_x \cj{w_x}   -2 {\rm Im}  J(w)_x  w \cj{w_x}    + \left( |w|^6 - 2m |w|^4 +  m^2\,|w|^2\right)   \nonumber \\                        
                        \,&=&\,  w_x \cj{w_x}   -2  {\rm Im} \, w^2  \cj{w} \cj{w_x}   +2m \, {\rm Im} \,w \cj{w_x}   + \left( |w|^6 - 2m |w|^4 +  m^2\,|w|^2\right)                      
                           \label{oo2}
                       \end{eqnarray}
By the same calculations we also have
\begin{equation}\label{oo3}
u^2\cj{u} \cj{u_x} =w^2\cj{w} \cj{w_x} - i|w|^6 +i m\, |w|^4.
\end{equation}                      

We now recall that
\begin{equation}\label{oo1}
E(u)(t) = \int | u_x|^2 \, dx + \frac{3}{2}  {\rm Im} \int  u^2 {\cj {u}\, \cj{u_x}} \, dx + \frac{1}{2} \int |u|^6 \, dx,                  
\end{equation}    
 hence by using  \eqref{oo1}, \eqref{oo2},  \eqref{oo3}  we find  
$$
E(u)=  \int w_x \cj{w_x} \, dx  - \frac{1}{2} {\rm Im} \int w^2 \cj{w} \cj{w_x} \,dx + 2m\, {\rm Im}\, \int w \cj{w_x} \, dx   -\frac{1}{2}m\,  \int |w|^4 \, dx  + 
2\pi m^3.
$$
If we define
\begin{equation}\label{ew}
\mathscr E(w):=  \int_{\T} | w_x|^2 \, dx - \frac{1}{2} {\rm Im} \int_{\T} w^2 {\cj {w} \, \cj{w_x}} \, dx + \frac{1}{4 \pi}  \biggl(\int_{\T} |w(t)|^2 \, dx \biggr)\biggl(\int_{\T} |w(t)|^4 \, dx\biggr),
\end{equation}
then $E(u)$ can be rewritten as
\begin{equation}\label{eu}
E(u) = \mathscr E(w) + 2m\, \mathscr H(w) - 2\pi \, m^3=:\EE(w).
\end{equation}

\begin{remark}\label{conservations}\rm
We observe that $H(u)(t)= \mathscr H(w)(t)$ and $\frac{d}{dt}H(u)(t)=0$ since $H$ is the Hamiltonian for (DNLS) \eqref{DNLS}, hence it follows that 
$\frac{d}{dt}\mathscr H(w)(t)=0$. On the other hand, we also know that $\frac{d}{dt}E(u)(t)=0$, hence $\frac{d}{dt}\mathscr E(w)(t)=0$.
By the translation invariance of integration over $\T$, we have that \eqref{eu} holds with $v$ in place of $w$ and 
$$\frac{d}{dt}\mathscr H(v)(t) = 0 = \frac{d}{dt}\mathscr E(v)(t).$$
\end{remark}

\section{Finite dimensional approximation of (GDNLS)}

We denote by $P_N f =  \sum_{|n|\leq N}\widehat{f}(n)e^{ i n x}$  the finite 
dimensional projection onto the first $2N+1$ modes and $P_N^{\perp}:= I - P_N$. Then 
the  finite dimensional approximation (FGDNLS) is: 
\begin{equation} \label{FGDNLS} v^N_t  = i v^N_{xx}  - P_N ((v^N)^2  \cj{v^N_x})+ \frac{i}{2}\, P_N (|v^N|^4 v^N) - i \psi(v^N)(t) v^N   - i m(v^N) P_N (|v^N|^2 v^N)  \end{equation}
with initial data 
\begin{equation}v_0^N=P_N \,v_0, \label{fid}\end{equation} where $m$ and $\psi$ are as defined in \eqref{defm} and \eqref{defpsi} respectively.

\medskip

\begin{lemma} \label{alsoconserved} We have that  
$$\frac{d}{dt}m(v^N) (t) := \frac{d}{dt}\frac{1}{2\pi} \int_{\T} |v^N(x,t)|^2 dx=0.$$
\end{lemma} 
\begin{proof}
Indeed for simplicity let us momentarily denote by $w:=v^N$ a solution to \eqref{FGDNLS}; note $P_N w =w$. Then
using that for any $F$,  $\int P_N (F(v^N))   \cj{v^N}  dx = \int  F(v^N)  P_N \cj{v^N} dx = \int   F(v^N)  \cj{v^N} dx$ we obtain
 \begin{align*}
&\dt \big(2\pi m(w)\big) = 2 \text{Re} \int w_t \cj{w} \, dx \\
& = 2 \text{Re} \bigg( - i \int |w_{x}|^2 - \int P_N(w^2 \cj{w}_x) \cj{w}
+ \frac{i}{2} \int P_N(|w|^4 w ) \cj{w} \, - \bigg.\\
&\bigg. \qquad \qquad  -\,  i \psi(w) (t) \int |w|^2 - i m(w)(t) \int P_N(|w|^2 w) \cj{w} \bigg)\\
&=  2 \text{Re}    \bigg(     - \int (w^2 \cj{w}_x) \cj{w}
+ \frac{i}{2} \int |w|^6 - i \psi(w) \int| w|^2 - i m(w) \int |w|^4 \bigg)     \\
& = - \int w^2 \cj{w}\cj{w}_x - \int w w_x \cj{w}^2 = -\frac{1}{2} \int \dx (|w|^4) = 0.
\end{align*}
\end{proof}

\begin{theorem}[Local well-posedness] \label{lwp} Let $2<r<4$ and $s\geq \frac{1}{2}$.  
Then for every 
\begin{equation} v^N_0 \in {B_R} :=\{ v^N_0 \in \FF L^{s,r}(\T)\, /\,  \|v^N_0\|_{\FF L^{s,r}(\T)}<R\}
\end{equation} 
and $\delta \lesssim R^{-\gamma}$, for some $ \gamma>0$, there exists a unique solution 
\begin{equation}v^N\in Z^{s}_{r}(\delta)\subset C([-\delta,\delta]; \FF L^{s,r}(\T))\end{equation}
of \eqref{FGDNLS} and \eqref{fid}. Moreover the map 
$$\left(B_R, \|\cdot\|_{\FF L^{s,r}(\T)}\right) \longrightarrow C([-\delta,\delta]; \FF L^{s,r}(\T)) : \, \, \, v^N_0 \rightarrow  v^N$$
is real analytic. 
\end{theorem}
\begin{proof} The proof follows the argument in \cite{GH}, Theorem 7.2 since $P_N$ acts on a multilinear nonlinearity and it is a bounded operator in 
$L^p, 1<p<\infty$ and commutes with $D^s$.  Also, although the proof in \cite{GH} is presented for $s=\frac{1}{2}$, a simple argument of persistence of regularity
gives the result for any $s\geq \frac{1}{2}$.
\end{proof}

The following lemma gives control on how close the finite dimensional approximations
are to the solution of \eqref{GDNLS}.  Our proof is a variation of Bourgain's Lemma 2.27 in \cite{B1} (see also \cite{BT1}).

\begin{lemma}[Approximation lemma]\label{approximation} Let  $v_0\in  \FF L^{s,r}(\T), s > \frac{1}{2},  \, 2< r <4$ be such that  $\|v_0\|_{\FF L^{s,r}(\T)}<A$, for some $A>0$, and let $N$ be a large integer. Assume the solution 
$v^N$ of \eqref{FGDNLS} with initial data $v_0^N(x) := P_N(v_0)$ satisfies the bound
\begin{equation}\label{bound}
\|v^N(t)\|_{\FF L^{s,r}(\T)}\leq A,  \mbox{ for all } t\in [-T,T],
\end{equation}
for some given $T>0$.
Then the IVP (GDNLS)  \eqref{GDNLS} with initial data $v_0$ is well-posed on $[-T,T]$ and  there exists $C_0, C_1>0$, such that its solution $v(t)$ satisfies the following estimate:
\begin{equation}\label{decay}
\|v(t)-v^N(t)\|_{\FF L^{s_1,r}(\T)}\lesssim \exp[C_0(1+A)^{C_1}T]N^{s_1-s},
\end{equation}
for all $t\in[-T,T], \frac{1}{2} \le s_1<s$.
\end{lemma}
\begin{proof}
We first observe that from the local well-posedness theory (\cite{GH} and  Theorem \ref{lwp}),  (GDNLS)  \eqref{GDNLS} with initial data $v_0$ and  (FGDNLS) \eqref{FGDNLS} with initial data $v_0^N$ are both well-posed 
in $[-\delta,\delta], \delta\sim (1+A)^{-\gamma}$.  Let $w:=v-v^N$, then $w$ satisfies the equation 
\begin{equation}\label{diff}
 w_t-iw_{xx}=F(v)-P_NF(v^N)=P_N[F(v)-F(v^N)]+(1-P_N)F(v),
\end{equation}
where $F(\cdot)$ is the nonlinearity of \eqref{GDNLS}.  By the Duhamel principle we have
$$w(t)=S(t)[v_0-v_0^N]+\int_0^t S(t-t')(P_N[F(v)-F(v^N)](t')+(1-P_N)F(v)(t'))\,dt' ,$$
where $S(t)=e^{it \Delta}$, and  from the proof of Theorem 7.2 in \cite{GH} we have the 
bound 
\begin{eqnarray}
\|w\|_{Z^{s_1}_{r}(\delta)}&\lesssim &\|v_0-v_0^N\|_{\FF L^{s_1,r}(\T)}+\delta^\gamma(1+\|v^N\|_{Z^{s_1}_{r}(\delta)}+\|v\|_{Z^{s_1}_{r}(\delta)})^4\|w\|_{Z^{s_1}_{r}(\delta)}\notag\\
&+&\left\| (1-P_N)\int_0^tS(t-t')F(v)(t')\,dt' \right\|_{Z^{s_1}_{r}(\delta)}\notag\\
&\lesssim &AN^{s_1-s}+\delta^\gamma(1+\|v^N\|_{Z^{s_1}_{r}(\delta)}+\|v\|_{Z^{s_1}_{r}(\delta)})^4\|w\|_{Z^{s_1}_{r}(\delta)}+N^{s_1-s}\delta^\gamma(1+\|v\|_{Z^{s}_{r}(\delta)})^5.\label{laststep}
\end{eqnarray}
By choosing a smaller $\delta$ if necessary we obtain from \eqref{laststep} 
$$
\|w\|_{Z^{s_1}_{r}(\delta)}\leq C AN^{s_1-s}+\frac{1}{2}\|w\|_{Z^{s_1}_{r}(\delta)},
$$
for some absolute constant $C>0$, from where 
\begin{equation}\label{loc}
\|v(t)-v^N(t)\|_{\FF L^{s_1,r}(\T)}\leq 2 C AN^{s_1-s}, \, \, \mbox{ for all } t\in [-\delta,\delta]
\end{equation}
and by  iteration \eqref{decay} follows.
\end{proof}

\section{Analysis of the Finite Dimensional Equation (FGDNLS)}

Recall that equation (DNLS)  is Hamiltonian and hence its gauge equivalent formulation 
should stay Hamiltonian  (change of coordinates).  However, the gauge transformation is not 
a `canonical map'  and the symplectic form in the new  coordinates depends on $v$; 
that is we lose the simple expression the symplectic form (namely $\partial_x$) had in the 
original coordinates.  Two problems arise from the lack of commutativity between the gauged
skew-selfadjoint form $J$ and $P_N$:
\smallskip

\noi
(1) The conservation of Lebesgue measure associated to (FGDNLS) is not obvious as before. We must prove that this is indeed the case; see Subsection \ref{lebesgue} below.
\smallskip

and more seriously:

\smallskip
\noi
(2) We lose the conservation of the energy $\EE(v^N)$ for the finite dimensional 
approximations; that is $ \dfrac{d \EE(v^N)}{dt} \ \neq \, 0.$  In particular we lose the invariance
of $\mu_N$,  the associated finite dimensional weighted Wiener  measure. However we have
an  estimate controlling its growth, namely Theorem \ref{Energy Growth Estimate} below.

\subsection{Invariance of the Lebesgue measure}\label{lebesgue}
If we rewrite (FGDNLS) \eqref{FGDNLS} as a system of complex ODE's for the Fourier
coefficients $c_k \equiv \widehat{v^N}(k)$ we obtain a set of $2N+1$ complex equations 
of the form 
$\frac{d}{dt} c_k =F_k( \{ c_j, \bar{c}_j \})$, 
or equivalently $4N+2$ equations $\frac{d}{dt} a_k = {\rm Re F}_k( \{ c_j, \bar{c}_j \})$, 
and $\frac{d}{dt} b_k = {\rm Im F}_k( \{ c_j, \bar{c}_j \})$
for the real functions $a_k={\rm Re}F_k $ and $b_k={\rm Im} F_k $.   

To show that this set of equations preserves volume we need to verify that the divergence of 
the vector field vanishes, i.e.,  
$$ \sum_{k}  \frac{\partial  {\rm Re } F_k } { \partial a_k }   + \frac{\partial  {\rm Im } F_k } { \partial b_k } \,=\,0.$$
This is easily shown to be equivalent to 
$$ \sum_{k}  \frac{\partial  F_k } { \partial c_k }   + \frac{\partial  \bar{F}_k } { \partial \bar{c}_k } \,=\, 0.$$

\noindent
And indeed we have 

\begin{lemma} \label{invlebesgue}The Lebesgue measure $\prod_{|j|\le N} da_j db_j$ is invariant under the flow of the system of ODE's (\ref{ode}). 
\end{lemma}

\begin{proof} 
%
The  (FGDNLS) \eqref{FGDNLS} as a system of complex ODE's for the Fourier coefficients
$c_k$ takes the form 
\begin{eqnarray}
\frac{d}{dt} c_k  \,&=&\,  -i k^2 c_k  +  i \sum_{n_1, n_2, n_3} n_3 c_{n_1} c_{n_2} \bar{c}_{n_3} \delta_{n_1+n_2-n_3 -k}   \nonumber  \\
&& + \frac{i}{2} \sum_{n_1, n_2, n_3,n_4,n_5} c_{n_1} c_{n_2} c_{n_3} \bar{c}_{n_4} \bar{c}_{n_5} \delta_{n_1+n_2+n_3 -n_4-n_5-k}   \nonumber \\
&& - i \psi( \{{c_j}, \bar{c}_j\})  c_k  - i m( \{{c_j}, \bar{c}_j\})   \sum_{n_1, n_2, n_3} c_{n_1} c_{n_2} \bar{c}_{n_3} \delta_{n_1+n_2-n_3 -k}  \label{ode}
\end{eqnarray}
with $m(\{{c_j}, \bar{c}_j\}) \,=\, \sum_{j} |c_j|^2$ and 
\begin{equation}
\psi(\{{c_j}, \bar{c}_j\}) \,=\,  -2 \sum_{k} k |c_k|^2  + \frac{1}{2}  \sum_{n_1, n_2, n_3, n_4} c_{n_1} c_{n_2} \bar{ c}_{n_3} \bar{c}_{n_4} \delta_{n_1+n_2-n_3 -n_4}  - \left( \sum_{j} |c_j|^2 \right)^2. 
\end{equation}
To show that this set equation preserve volume we need to verify 
\begin{equation}
 \sum_{k}  \frac{\partial  F_k } { \partial c_k }   + \frac{\partial  \bar{F}_k } { \partial \bar{c}_k } \,=\, 0.
\end{equation}
The vector field $F_k$ consists of several terms which we analyze separately.

\vspace{2mm}

\noindent  {\bf (1)} $F^{(1)}_k = -i  k^2 c_k$.  Then 
$\frac{\partial  F^{(1)}_k } { \partial c_k }   + \frac{\partial  \bar{F}^{(1)}_k } { \partial \bar{c}_k } = -ik^2 + i k^2 =0$. 

\vspace{2mm}

\noindent {\bf (2)} $F^{(2)}_k =  i \sum_{n_1, n_2, n_3} n_3 c_{n_1} c_{n_2} \bar{c}_{n_3} \delta_{n_1+n_2-n_3 -k} $. To differentiate we consider the terms with $n_1=k$ and $n_2=k$
and obtain 
\begin{equation}
\frac{ \partial F^{(2)}_k}{\partial c_k} =  i 2\pi \sum_{ n_2, n_3} n_3  c_{n_2} \bar{c}_{n_3} \delta_{n_2-n_3} +    i 2\pi \sum_{ n_1, n_3} n_3  c_{n_1} \bar{c}_{n_3} \delta_{n_1-n_3}  = i 4 \pi \sum_{n}n |c_n|^2     
 \end{equation}
and similarly  
\begin{equation} \frac{ \partial \bar{F}^{(2)}_k}{\partial \bar{c}_k} = -i4\pi \sum_{n} n |c_n|^2
\end{equation} and thus all the contributions of this term to the divergence disappear. 

\vspace{2mm}

\noindent {\bf (3)} $F^{(3)}_k = \frac{i}{2} \sum_{n_1, n_2, n_3,n_4,n_5} c_{n_1} c_{n_2} c_{n_3} \bar{c}_{n_4} \bar{c}_{n_5} \delta_{n_1+n_2+n_3 -n_4-n_5-k}$. This term is treated similarly as (2)
and is left to the reader. 

\vspace{2mm}

\noindent{\bf (4)} $F^{(4)}_k   = 2i ( \sum_{j} j |c_j|^2 ) c_k$. We have 
\begin{equation}
\frac{ \partial F^{(4)}_k}{\partial c_k}  \,=\,  2i k |c_k|^2  +  2i \sum_{j} j |c_j|^2
\end{equation}
and
\begin{equation} 
\frac{ \partial \bar{F}^{(4)}_k}{\partial \bar{c}_k}  \,=\,  -2i k |c_k|^2  - 2i \sum_{j} j |c_j|^2
\end{equation}
and so these terms do not contribute to the divergence.  

\vspace{2mm}

\noindent{\bf (5)} $F^{(5)}_k   = i ( \sum_{j}  |c_j|^2 )^2 c_k$.  We have 
\begin{equation}
\frac{ \partial F^{(5)}_k}{\partial c_k}  \,=\,  2i  (\sum_{j} |c_j|^2) |c_k|^2  +  i ( \sum_{j}  |c_j|^2)^2 
\end{equation}
and again we have $\frac{ \partial F^{(5)}_k}{\partial c_k}  + \frac{ \partial \bar{F}^{(5)}_k}{\partial \bar{c}_k}=0$.

\vspace{2mm}

\noindent{\bf (6)}   $F^{(6)}_k  =  -\frac{i}{2} \sum_{n_1, n_2, n_3, n_4} c_{n_1} c_{n_2} \bar{ c}_{n_3} \bar{c}_{n_4} \delta_{n_1+n_2-n_3 -n_4}  c_k $. We have  
\begin{eqnarray}
\frac{ \partial F^{(6)}_k}{\partial c_k} \,&=&\,   -\frac{i}{2} \sum_{n_1, n_2, n_3, n_4} c_{n_1} c_{n_2} \bar{ c}_{n_3} \bar{c}_{n_4} \delta_{n_1+n_2-n_3 -n_4}  \nonumber \\
&& - i  \sum_{n_2, n_3, n_4} c_k c_{n_2}\bar{ c}_{n_3} \bar{c}_{n_4} \delta_{k+n_2-n_3 -n_4} \label{l1}
\end{eqnarray} 
and 
\begin{eqnarray}
\frac{ \partial \bar{F}^{(6)}_k}{\partial \bar{c}_k} \,&=&\,  + \frac{i}{2} \sum_{n_1, n_2, n_3, n_4} \bar{c}_{n_1} \bar{c}_{n_2} c_{n_3} c_{n_4} \delta_{n_1+n_2-n_3 -n_4} \nonumber \\
&& +i  \sum_{n_2, n_3, n_4} \bar{c}_k \bar{c}_{n_2} c_{n_3} c_{n_4} \delta_{k+n_2-n_3 -n_4}.
\label{l2}
\end{eqnarray}
The first terms in  (\ref{l1}) and (\ref{l2}) cancel for each $k$.  By summing the second  terms in  
(\ref{l1}) and (\ref{l2}) over $k$,  we see that they do not contribute to the divergence.

\noindent{\bf (7)}   $F^{(7)}_k  =  -i  \sum_{j} |c_j|^2   \sum_{n_1, n_2, n_3} c_{n_1} c_{n_2} \bar{c}_{n_3} \delta_{n_1+n_2-n_3 -k} $. We have 
\begin{eqnarray}
\frac{ \partial F^{(7)}_k}{\partial c_k}&=&  -i \sum_{n_1, n_2, n_3} c_{n_1} c_{n_2} \bar{c}_{n_3}  \bar{c}_k \delta_{n_1+n_2-n_3 -k} \, - \, 2i (\sum_j|c_j|^2)^2   
\end{eqnarray} 
and 
\begin{eqnarray}
\frac{ \partial \bar{F}^{(7)}_k}{\partial \bar{c}_k}&=&  i \sum_{n_1, n_2, n_3} \bar{c}_{n_1} \bar{c}_{n_2} c_{n_3}  c_k \delta_{n_1+n_2-n_3 -k} \, + \, 2i (\sum_j|c_j|^2)^2.
\end{eqnarray}
The second terms add to $0$ for each $k$ while the first terms cancel if we sum over all $k$. 

\end{proof}

\subsection{Energy growth estimate}\label{energyestimate}

\begin{theorem} \label{Energy Growth Estimate} Let $v^N(t)$ be a solution to (FGDNLS) \eqref{FGDNLS} in $[-\delta,\delta],$ and let $K>0$ be such that 
$\|v^N\|_{X_3^{\frac{2}{3}-, \frac{1}{2}}(\delta)} \leq K.$ Then there exists $\beta>0$ such that 
\begin{equation} \label{growthestimate}
|\EE( v^N(\dl)) - \EE(v^N(0))| = \Big|\int_0^\dl \frac{d}{dt} \EE(v^N)(t)dt\Big|
\lesssim  C(\dl) N^{-\beta}\max(K^6,K^8). \end{equation}
\noi
\end{theorem}

\begin{remark}\rm \label{fixchoice}
 It is possible that  the estimate \eqref{growthestimate}  may still hold for a different choice 
 of  $X_r^{s, \frac{1}{2}}(\delta)$ norm, with $s\geq \frac{1}{2}, \, 2<r<4$ so that  local 
 well-posedness holds.  On the other hand the  pair $(s, r)$ should also be such that 
 $(s-1)\cdot r < -1$ in order for $\mathcal F L^{s,r}$ to  contain the support of the Wiener
 measure  (c.f. Section 5).  Our choice of  $s = \frac{2}{3} -$ and $r=3$  allows us to prove 
 \eqref{growthestimate} while satisfying the conditions for local well-posedness and the 
 support of the measure. Note that  $\mathcal F L^{\frac{2}{3}-, 3}$ scales like $H^{\frac{1}{2}-}$. 
\end{remark}

\subsection{Preparation for the proof of Theorem \ref{Energy Growth Estimate}  }

Let $v^N$ denote the solution of (FGDNLS) \eqref{FGDNLS} which we rewrite as 

\[v^N_t = \mathcal{L} v^N + {P}_N^\perp( (v^N)^2 \cj{v^N_x})
- \frac{i}{2}{P}_N^\perp (|v^N|^4 v^N) + i  m(v^N){P}_N^\perp(|v^N|^2 v^N),\]
where
\begin{equation}\label{L}\mathcal{L} v^N := i v^N_{xx} - (v^N)^2  \cj{v^N_x} + \frac{i}{2} |v^N|^4 v^N
- i \psi(v^N) v^N - i m(v^N) |v^N|^2 v^N.\end{equation}
We first observe that  from \eqref{eu} 
and Lemma \ref{alsoconserved} we have 
\begin{equation}\label{dtE}
\frac{d}{dt}\EE(v^N)= \frac{d}{dt}\mathscr E(v^N)+2m_N \frac{d}{dt}\mathscr H(v^N),
\end{equation}
where $m_N:=m(v^N)$.
\begin{lemma}\label{dtformulas} With the above notations we have
\begin{align}  \label{dt E3}
\dt \mathscr{E}(v^N) (t)= 
& - 2\text{Im} \int v^N\cj{v^Nv^N_x} {P}_N^\perp( (v^N)^2 \cj{v^N_x})\,dx
+ \text{Re} \int v^N\cj{v^Nv^N_x} {P}_N^\perp (|v^N|^4 v^N) \,dx\notag  \\
& - 2 m_N \text{Re} \int v^N\cj{v^Nv^N_x} {P}_N^\perp(|v^N|^2 v^N)\,dx
+ 2 m_N \text{Re} \int v^N \cj{v^N}^2 {P}_N^\perp( (v^N)^2 \cj{v^N_x})\,dx \\
& + m_N  \text{Im} \int v^N \cj{v^N}^2 {P}_N^\perp (|v^N|^4 v^N)\,dx
- 2 m_N ^2 \text{Im} \int v^N \cj{v^N}^2 {P}_N^\perp(|v^N|^2 v^N)\,dx, \notag
\end{align}
\begin{align}\label{dt H3}
\dt \mathscr H(v^N) (t)=&-2\text{Re}\int_\T(\cj{v^N})^2v^N{P}_N^\perp( (v^N)^2 \cj{v^N_x})\, dx+\text{Im} \int v^N (\cj{v^N})^2 {P}_N^\perp (|v^N|^4 v^N)\,dx\notag\\
&- 2 m_N  \text{Im} \int v^N (\cj{v^N})^2 {P}_N^\perp(|v^N|^2 v^N)\,dx, 
\end{align}
and 
\begin{align}\label{EE3}
\dt \EE(v^N) (t)=&- 2\text{Im} \int v^N\cj{v^Nv^N_x} {P}_N^\perp( (v^N)^2 \cj{v^N_x})\,dx
+ \text{Re} \int v^N\cj{v^Nv^N_x} {P}_N^\perp (|v^N|^4 v^N) \,dx \notag  \\
& - 2 m_N \text{Re} \int v^N\cj{v^Nv^N_x} {P}_N^\perp(|v^N|^2 v^N)\,dx
-2 m_N \text{Re}\int_\T(\cj{v^N})^2v^N{P}_N^\perp( (v^N)^2 \cj{v^N_x})\, dx \\
&+3m_N\text{Im} \int v^N (\cj{v^N})^2 {P}_N^\perp (|v^N|^4 v^N)\,dx\
-6 m_N  \text{Im} \int v^N (\cj{v^N})^2 {P}_N^\perp(|v^N|^2 v^N)\,dx. \notag
\end{align}

\end{lemma}
\begin{proof}
From \eqref{ew}  and integration by parts we have that 
\begin{equation}\label{dt E2} 
\dt \mathscr E(v^N)(t) = -2\text{Re} \int  v^N_t \cj{v^N_{xx}}\,dx
- 2\text{Im} \int v^Nv^N_t \cj{v^Nv^N_x}\,dx
+ 2 m_N\text{Re} \int v^N v^N_t \cj{v^N}^2\,dx.
\end{equation}

\noi
Due to the energy conservation for the (GDNLS)  (infinite system), one can see that the contribution  in \eqref{dt E2}
from $\mathcal{L} v^N$ defined in \eqref{L} is zero. On the other hand by orthogonality we also have

\noi
\[-2\text{Re} \int \cj{v^N_{xx}} \Big({P}_N^\perp( (v^N)^2 \cj{v^N_x})
- \frac{i}{2}{P}_N^\perp (|v^N|^4 v^N) + i  m(v^N){P}_N^\perp(|v^N|^2 v^N)\Big) \,dx= 0.\]
Hence \eqref{dt E3} follows.  By a similar argument we obtain \eqref{dt H3} as well. The lemma follows by substituting \eqref{dt E3} and \eqref{dt H3} into \eqref{dtE}.

\end{proof}

 \begin{remark}\rm To establish Theorem \ref{Energy Growth Estimate} we need to estimate the terms in  \eqref{EE3}. In doing so we will ignore absolute constants and weather we are looking at the real or imaginary parts of the terms.
 \end{remark}
 The first term  in \eqref{EE3} gives a contribution to  \eqref{growthestimate} which is 
 essentially:  
 \begin{equation}\label{firstterm}
 I_1=\int_0^\delta \int v^N\cj{v^Nv^N_x} {P}_N^\perp( (v^N)^2 \cj{v^N_x})\,dx \, dt.
 \end{equation}
 This term is the hardest to control since it has two derivatives, so we will treat this one first. We start by discussing 
 how to absorb   the rough time cut-off.  Assume $\phi$ is any function in $X^{\frac{2}{3}-,\frac{1}{2}}_3$ such that 
 \begin{equation}\label{phi}
 \phi|_{[-\delta, \delta]} \, = v^N.
 \end{equation}
Then we write 
$$I_1= \int_{\mathbb{T} \times \mathbb{R}} \, \, \chi_{[0, \delta]}(t)\, \, {P}_N^\perp ((v^N)^2 \dx\cj{v^N})  \,  v^N \cj{v^N v^N_x}  dx dt $$
$$ = \int_{\mathbb{T} \times \mathbb{R}} \, \, \, \, {P}_N^\perp ( (\chi_{[0, \delta]} \phi^N)^2 \, \chi_{[0, \delta]} \cj{\phi^N_x} ) \,  \,  \chi_{[0, \delta]} \phi^N\, \chi_{[0, \delta]} \cj{\phi^N}\,  \chi_{[0, \delta]} \cj{ \phi^N_x} dx dt  $$ and by denoting 
\begin{equation}\label{w}
w := \chi_{[0, \delta]} \phi, \qquad  \qquad   w=P_N(w),
\end{equation}
 we will in fact   show that 
 
\begin{align} \label{wtermfirst} |I_1|=&\left | \int_{\mathbb{T} \times \mathbb{R}} \, \, \, \, {P}_N^\perp ( (w)^2 \dx\cj{w} )  \,  w
 \cj{w w_x} dx dt\right | \leq  \, C(\delta) N^{-\beta}  \Vert w\Vert_{X^{\frac{2}{3}-, \frac{1}{2}-}_{3} }^6.   
\end{align}
To go back to $v^N$ we use the following lemma:
\begin{lemma}[Time-Cutoff] \label{time-cutoff} Let  $b <  b_1 <1/2$. Then the exists $C'(\delta) >0$ such that  
$$\| w \|_{X^{\frac{2}{3}-,b}_3}  \leq C'(\delta)\, \| \phi  \|_{X^{\frac{2}{3}-,b_1}_3} \leq C'(\delta)\, \| v^N  \|_{X^{\frac{2}{3}-,\frac{1}{2}}_3(\delta)} $$ 
where $w, \, \phi$  and $v^N$ are as above.
\end{lemma}

\begin{proof}  
Since the  regularity in $x$ does not play any role, without any loss of generality we ignore the power $s=\frac{2}{3}-$. Then, 
\begin{eqnarray}\label{conv}
\| w\|_{X^{0,b}_3} &=& \biggl( \sum_{n} \biggl( \int  \, | \widehat{ \chi_{[0, \delta]} \phi} (n, \tau)|^2 \langle \tau+n^2 \rangle^{2b} \, d\tau \biggr)^{\frac{3}{2}}\biggr)^{\frac{1}{3}} \notag \\
&=& \biggl( \sum_{n} \biggl( \int  \, |  \int_{\tau_1}\, \widehat{\chi_{[0, \delta]}}(\tau- \tau_1) \, \widehat{\phi} (n, \tau_1)\, d \tau_1 |^2 \, \langle \tau+n^2 \rangle^{2b} \, d\tau \biggr)^{\frac{3}{2}}\biggr)^{\frac{1}{3}}.  
\end{eqnarray}
Writing $ \tau + n^2 =  (\tau -\tau_1) + (\tau_1 + n^2)$ we bound \eqref{conv} by 
\begin{eqnarray}
&& \lesssim  \biggl( \sum_{n} \biggl( \int  \, |  \int_{\tau_1}\, \widehat{\chi_{[0, \delta]}}(\tau- \tau_1)  \,  \langle \tau - \tau_1 \rangle^{b}\, \widehat{\phi} (n, \tau_1)\, d \tau_1 |^2 \, \, d\tau \biggr)^{\frac{3}{2}}\biggr)^{\frac{1}{3}}  \label{conv2}\\
&&+ \,  \biggl( \sum_{n} \biggl( \int  \, |  \int_{\tau_1}\, \widehat{\chi_{[0, \delta]}}(\tau- \tau_1) \, \widehat{\phi} (n, \tau_1)\,  \langle \tau_1 +n^2 \rangle^{b} \, d \tau_1 |^2 \, d\tau \biggr)^{\frac{3}{2}} \biggr)^{\frac{1}{3}}. \label{conv3} 
\end{eqnarray}
We treat the first sum \eqref{conv2},  the second one \eqref{conv3} being similar. If $\langle \tau - \tau_1\rangle < \langle \tau_1 + n^2 \rangle$ then by Young's inequality \eqref{conv2} can be bounded by
\begin{equation*} 
\lesssim  \| \frac{ \widehat{\chi_{[0, \delta]}}(\tau)}{\langle \tau\rangle^{\varepsilon}} \|_{L^1} \left\|  \| \widehat{\phi}(\tau, n) \langle \tau +n^2 \rangle^{b+ \varepsilon} \|_{L^2} \right\|_{\ell^3} \, 
\lesssim \, \|\chi \|_{H^{\beta}}\,  \| \phi \|_{X^{0, b_1}_3} 
 \end{equation*}  
by Cauchy-Schwarz on the $\widehat{\chi}$ term provided $ \beta+\varepsilon > \frac{1}{2},  \beta < \frac{1}{2}$  and where  $b_1:=\, b+\varepsilon < \frac{1}{2}$. 

On the other hand if $\langle \tau - \tau_1\rangle \geq \langle \tau_1 +n^2 \rangle$, then again by Young's inequality \eqref{conv2} can be bounded by  
\begin{equation*} 
\lesssim  \|  \widehat{\chi_{[0, \delta]}(\tau)} \langle \tau\rangle^{b+ \varepsilon} \|_{L^2} \left\|  \| \widehat{\phi}(\tau, n) \langle \tau + n^2 \rangle^{-\varepsilon} \|_{L^1} \right\|_{\ell^3} \, 
\lesssim \, \|\chi \|_{H^{b+\varepsilon}}\,  \| \phi \|_{X^{0, b_1}_3} 
 \end{equation*}  
by Cauchy-Schwarz on the $\widehat \phi$ term provided $ b_1+\varepsilon > \frac{1}{2},  b_1 < \frac{1}{2}$. 
Finally by taking infimum and using the definition of $X^{0,\frac{1}{2}}_3(\delta)$ a bound in terms of $ \Vert v^N\Vert_{X^{0,\frac{1}{2}}_3(\delta)}$ follows.

\end{proof}

\subsection{Proof of Theorem \ref{Energy Growth Estimate}  }
 Returning to \eqref{wtermfirst}  we write 
\begin{align} 
 I_1&=  \int_{\mathbb{T} \times \R} 
  {P}_N^\perp (w^2 \dx\cj{w})  \,  w\,\cj{w w_x}  dx dt \notag \\
&= \int_\tau \sum_{|n| >N} \widehat{(w^2 \cj{w_x})}(n, \tau) \cj{\widehat{(\cj{w} w w_x)}}(n, \tau) d\tau  \notag \\
&= \int \sum_{|n| >N} 
\bigg(\int_{\tau = \tau_1 + \tau_2 - \tau_3} \sum_{n = n_1 + n_2 - n_3, \,|n_j| \leq N } 
\ft{w}(n_1, \tau_1)\ft{w}(n_2, \tau_2)(-i n_3)\cj{\ft{w}}(n_3, \tau_3) d\tau_1 d\tau_2 \bigg) \notag\\
& \times \bigg(\int_{ -\tau = \tau_4 - \tau_5 - \tau_6} \sum_{-n = n_4 - n_5 - n_6, \,|n_j| \leq N } 
\ft{w}(n_4, \tau_4)\cj{\ft{w}}(n_5, \tau_5)(-i n_6)\cj{\ft{w}}(n_6, \tau_6) d\tau_4 d\tau_5 \bigg) d\tau
\notag
\end{align}

\noi

\begin{align*}
= &\int \sum_{N< |n| \leq 3 N} 
\bigg(\int_{\tau = \tau_1 + \tau_2 + \tau_3} \sum_{n = n_1 + n_2 + n_3, \,|n_j| \leq N } 
\ft{w}(n_1, \tau_1)\ft{w}(n_2, \tau_2)(i n_3)\ft{\cj{w}}(n_3, \tau_3) d\tau_1 d\tau_2  \bigg) \\
& \times \bigg(\int_{ -\tau = \tau_4 + \tau_5 + \tau_6} \sum_{-n = n_4 + n_5 + n_6, \,|n_j| \leq N } 
\ft{w}(n_4, \tau_4)\ft{\cj{w}}(n_5, \tau_5)(i n_6)\ft{\cj{w}}(n_6, \tau_6) d\tau_4 d\tau_5 \bigg) d\tau \\
= &\int \sum_{N< |n| \leq 3 N} 
\bigg(\int_{\tau = \tau_1 + \tau_2 + \tau_3} \sum_{n = n_1 + n_2 + n_3, \,|n_j| \leq N } 
\ft{w_1}(n_1, \tau_1)\ft{w_2}(n_2, \tau_2)(i n_3)\ft{\cj{w_3}}(n_3, \tau_3) d\tau_1 d\tau_2\bigg) \\
& \times \bigg(\int_{ -\tau = \tau_4 + \tau_5 + \tau_6} \sum_{-n = n_4 + n_5 + n_6, \,|n_j| \leq N } 
\ft{w_4}(n_4, \tau_4)\ft{\cj{w_5}}(n_5, \tau_5)(i n_6)\ft{\cj{w_6}}(n_6, \tau_6) d\tau_4 d\tau_5 \bigg) d\tau, 
\end{align*}

\noi
where $w_1 = w_2 = w_4= w$ and $\cj{w_3} = \cj{w_5} = \cj{w_6} = \cj{w}$.

\begin{remark} \rm In what follows we always think of $N_j, N$ as dyadic; more precisely $N_j:= 2^{K_j}, \, N:= 2^{K}$ where $ K_j < K$ since 
$n_j \in \mathbb{Z}$.  By a slight abuse of notation we then denote by 
$N_j$ both $|n_j|$ and the dyadic interval $[2^{K_j}, \, 2^{K_j+1})$ $|n_j|$ belongs to when $n_j \neq 0$.  Moreover we denote by 
$w_{N_j}$ the function such that $\widehat{w_{N_j}}(n_j) = \chi_{\{|n_j| \sim N_j \}}  \widehat{w_{j}}(n_j)$

\end{remark}

From the expression  above we then  have,
\begin{equation} \label{Size} |n_j|  \leq  N, \qquad   N \le |n| \leq  3N, \qquad n= n_1 +n_2 +n_3, \quad \text{and} \quad -n = n_4 + n_5 + n_6,
\end{equation}
\begin{equation}  \label{RelativeSize}N \sim \max(N_1, N_2, N_3)  \sim \max(N_4, N_5, N_6),\end{equation}
\begin{equation} \label{Algebra1}
\tau + n^2 - (\tau_1 +n_1^2) - (\tau_2 +n_2^2) -(\tau_3 -n_3^2) = 2(n-n_1)(n-n_2) \end{equation}
and
\begin{equation}
\label{Algebra2}
\tau + n^2 + (\tau_4 +n_4^2) + (\tau_5 - n_5^2) +(\tau_6 - n_6^2) = 2(n+n_5)(n+n_6). 
\end{equation}

\medskip

So if we let  ${\tilde \s_j} := {\tau_j \pm n^2_j}$ and \, $\s_j:=  \jb{\tau_j \pm n^2_j}$  
we have by subtracting  \eqref{Algebra1}  from   \eqref{Algebra2}  

\begin{equation} \label{AAlgebra1}
 \sum_{j = 1}^6 {\tilde \s_j}  = - 2 \,( \, n\, (n_1 + n_2 + n_5 + n_6) - n_1 n_2 + n_5 n_6 \, ). 
\end{equation}

This in turn can also be rewritten using $n_1+n_2+n_3+n_4+n_5+n_6=0$ or $n= n_1 +n_2 +n_3$ and $-n = n_4 + n_5 + n_6$ as:
\medskip 
\begin{equation} \label{AAlgebra2}
 \sum_{j = 1}^6  {\tilde\s_j}  =  2 (\,  n\, ( n_3 + n_4)  + n_1 n_2 - n_5 n_6 \,). 
\end{equation}
In addition, since $\tau_1+\tau_2+\tau_3+\tau_4+\tau_5+\tau_6=0$, adding and subtracting $n_j^2, \, j=1,\dots,6$ in the appropriate fashion, we obtain:
\begin{equation} \label{AAlgebra3}
 \sum_{j = 1}^6 {\tilde \s_j} =  (n_3^2 + n_5^2 + n_6^2) - ( n_1^2 + n_2^2 + n_4^2).  
\end{equation}

\medskip 

Hence we need to estimate 

\begin{align}\label{MainExp} 
|I_1|=&\biggl| \sum_{N_i \leq N; \, i=1, \dots 6} \int_{\mathbb{R}} \int_{\mathbb{T}} \, \, \,  P_N^{\perp} \,\biggl(  w_{N_1} \, w_{N_2} \, \dx \cj{w_{N_3}}\, \biggr)\,  w_{N_4} \, \cj{w_{N_5}} \, \dx \cj{w_{N_6}} \, dx dt  \biggr| \quad \\ 
&=\, \biggl|\sum_{N_i \leq N; \, i=1, \dots 6}  \sum_{|n| \geq N} \int_{\tau}  \biggl( \int_{\tau= \tau_1+\tau_2+\tau_3} \, \sum_{n= n_1+n_2+n_3} \, \ft{w_{N_1}}  \ft{w_{N_2}} \,(i n_3)\,  \ft{\cj{w_{N_3}}} \, \, d \tau_1 d\tau_2\biggr) \quad \times \notag \\
&\qquad \quad \qquad \qquad \biggl( \int_{-\tau= \tau_4+\tau_5+\tau_6} \, \sum_{-n= n_4+n_5+n_6} \, \ft{w_{N_4}}  \ft{\cj{w_{N_5}}} \,(i n_6)\,  \ft{\cj{w_{N_6}}} \, \, d\tau_4 d\tau_5 \biggr)   \, d\tau  \biggr|  \quad \notag \\
&\leq \, \sum_{N\leq |n| \leq 3N} \, \sum_{N_i \leq N; \, i=1, \dots 6} \, \int_{\tau}  \biggl( \int_{\tau= \tau_1+\tau_2+\tau_3} \, \sum_{n= n_1+n_2+n_3} \, |\ft{w_{N_1}}| |\ft{w_{N_2}}| \, |n_3| \,  |\ft{\cj{w_{N_3}}}| \, \, d \tau_1 d\tau_2 \biggr) \quad \times \label{start} \\
&\qquad \quad \qquad \qquad \biggl( \int_{-\tau= \tau_4+\tau_5+\tau_6} \, \sum_{-n= n_4+n_5+n_6} \, |\ft{w_{N_4}}|  \,|\ft{\cj{w_{N_5}}}| \, |n_6|\,  |\ft{\cj{w_{N_6}}}|  \, \, d\tau_4 d\tau_5 \biggr)   \, d\tau.  \notag 
\end{align}

\bigskip 

\begin{remark} \rm This expression \eqref{start}  will be our point of departure in beginning our estimate.  In what follows we will  abuse notation and write $w_{N_j}$ for $ \widecheck {|\ft{w_{N_j}}|} $ and   $\cj{w_{N_k}}$ for  $\widecheck{|\ft{\cj{w_{N_k}}}|} $ since at the end we will estimate all functions in the $X^{s,b}_r$ norms which depend solely on the absolute value of the Fourier transform.

\end{remark}

We start by laying out all possible cases and organizing them according to the sizes of the two derivative terms. 

\medskip

{\bf Types:} 

\begin{enumerate}

\item [I.]   \quad $N_3\sim N, \, N_6 \sim N$   
\medskip

 \item[II.] \quad $N_3 \sim N$ and $N_6 \ll N$ 
\medskip
\item[III.]  \quad $N_6 \sim N$ and $N_3 \ll N$  

 \medskip

\item[IV.]  \quad  $N_3\ll N; \,  N_6 \ll  N$

\end{enumerate}

\medskip

Now we subdivide into all subcases in each situation and group them according to how many low frequencies ( ie. $N_j \ll N$) we have overall taking into account \eqref{RelativeSize}. 

\medskip 

{\bf All Cases for each type:} 

\medskip

\begin{enumerate}

\item [IA.]   \, $N_3\sim N, \, N_6 \sim N$   and  $4$ lows:   $N_1, N_2, N_4, N_5 \ll N$

\medskip

\item [IB.]  \, $N_3\sim N, \, N_6 \sim N$   and  $3$ lows

\quad (i) \, \, $N_1, N_2, N_4 \ll N$ and $N_5 \sim N$

\smallskip

\quad (ii)  \, \, $N_1, N_2, N_5 \ll N$ and $N_4 \sim N$

\smallskip

\quad  (iii)  \, \, $N_1, N_4, N_5 \ll N$ and $N_2 \sim N$

\smallskip

\quad  (iv)  \, \, $N_2, N_4, N_5 \ll N$ and $N_1 \sim N$

\medskip

\item [IC.]  \quad $N_3\sim N, \, N_6 \sim N$   and  $2$ lows

\quad (i) \, \, $N_1, N_2 \ll N$ and $N_4, N_5 \sim N$

\smallskip

\quad (ii)  \, \, $N_1, N_4 \ll N$ and $N_2, N_5 \sim N$

\smallskip 

\quad (iii) \, \, $N_1, N_5 \ll N$ and $N_2, N_4 \sim N$

\smallskip

\quad (iv)  \, \, $N_2, N_4 \ll N$ and $N_1, N_5 \sim N$

\smallskip

\quad (v)  \, \, $N_2, N_5 \ll N$ and $N_1, N_4 \sim N$

\smallskip

\quad (vi)  \, \, $N_4, N_5\ll N$ and $ N_1, N_2 \sim N$

\medskip 

\item [ID.]  \quad $N_3\sim N, \, N_6 \sim N$   and  $1$ low

\quad (i) \, \, $N_1\ll N$ and $N_2, N_4, N_5 \sim N$

\smallskip

\quad (ii) \, \, $N_2 \ll N$ and $N_1, N_4, N_5 \sim N$

\smallskip

\quad (iii) \, \, $N_4 \ll N$ and $N_1, N_2, N_5 \sim N$

\smallskip

\quad (iv) \, \, $N_5 \ll N$ and $N_1, N_2, N_4 \sim N$

\medskip

\item [IE.]  \quad $N_3\sim N, \, N_6 \sim N$   and   $N_1, N_2, N_4, N_5 \sim N$

\bigskip

\item[IIA.]   $N_3 \sim N$ and $N_6 \ll N$ and $3$ lows

\smallskip

\quad (i) \, \, $N_1, N_2, N_4 \ll N$ and $N_5 \sim N$

\smallskip

\quad (ii) \, \, $N_1, N_2, N_5 \ll N$ and $N_4 \sim N$

\medskip

\item[IIB.]   $N_3 \sim N$ and $N_6 \ll N$ and $2$ lows

\smallskip

\quad (i) \, \,  $N_1, N_2 \ll N$  and $N_4, N_5 \sim N$

\smallskip

\quad (ii) \, \,   $N_1, N_4 \ll N$  and $N_2, N_5 \sim N$

\smallskip

\quad (iii) \, \,  $N_1, N_5 \ll N$  and $N_2, N_4 \sim N$
\smallskip

\quad (iv) \, \, $N_2, N_4 \ll N$  and $N_1, N_5 \sim N$
\smallskip

\quad (v) \, \,  $N_2, N_5 \ll N$  and $N_1, N_4 \sim N$

\medskip

\item[IIC.]   $N_3 \sim N$ and $N_6 \ll N$ and $1$ low

\smallskip

\quad (i) \, \, $N_1 \ll N$  and $N_2, N_4, N_5 \sim N$

\smallskip

\quad (ii) \, \,  $N_2 \ll N$  and $N_1, N_4, N_5 \sim N$
\smallskip

\quad (iii) \, \, $N_4 \ll N$  and $N_1, N_2, N_5 \sim N$
\smallskip

\quad (iv) \, \, $N_5 \ll N$  and $N_1, N_2, N_4 \sim N$

\medskip
\item[IID.]   $N_3 \sim N$ and $N_6 \ll N$ and $N_1, N_2, N_4, N_5 \sim N$

\bigskip

\item[IIIA.]   $N_6 \sim N$ and $N_3 \ll  N$ and $3$ lows

\smallskip

\quad (i) \, \, $N_2, N_4, N_5 \ll N$ and $N_1 \sim N$

\smallskip

\quad (ii) \, \, $N_1, N_4, N_5 \ll N$ and $N_2 \sim N$

\medskip

\item[IIIB.]   $N_6 \sim N$ and $N_3 \ll N$ and $2$ lows

\smallskip

\quad (i) \, \,  $N_4, N_5 \ll N$  and $N_1, N_2 \sim N$

\smallskip

\quad (ii) \, \,   $N_1, N_4 \ll N$  and $N_2, N_5 \sim N$

\smallskip

\quad (iii) \, \,  $N_1, N_5 \ll N$  and $N_2, N_4 \sim N$
\smallskip

\quad (iv) \, \, $N_2, N_4 \ll N$  and $N_1, N_5 \sim N$
\smallskip

\quad (v) \, \,  $N_2, N_5 \ll N$  and $N_1, N_4 \sim N$

\medskip

\item[IIIC.]   $N_6 \sim N$ and $N_3 \ll N$ and $1$ low

\smallskip

\quad (i) \, \, $N_1 \ll N$  and $N_2, N_4, N_5 \sim N$

\smallskip

\quad (ii) \, \,  $N_2 \ll N$  and $N_1, N_4, N_5 \sim N$
\smallskip

\quad (iii) \, \, $N_4 \ll N$  and $N_1, N_2, N_5 \sim N$
\smallskip

\quad (iv) \, \, $N_5 \ll N$  and $N_1, N_2, N_4 \sim N$

\medskip
\item[IIID.]   $N_6 \sim N$ and $N_3 \ll N$ and  $N_1, N_2, N_4, N_5 \sim N$

\bigskip

\item[IVA.]   $N_3\ll N, N_6 \ll N$ and $2$ lows

\smallskip

\quad (i) \, \, $N_1, N_4 \ll N$ and $N_2, N_5 \sim N$

\smallskip

\quad (ii) \, \,  $N_1, N_5 \ll N$ and $N_2, N_4 \sim N$

\smallskip

\quad (iii) \, \,  $N_2, N_4 \ll N$ and $N_1, N_5 \sim N$

\smallskip

\quad (iv) \, \,  $N_2, N_5 \ll N$ and $N_1, N_4 \sim N$

\medskip

\item[IVB.]   $N_3\ll N, N_6 \ll N$ and $1$ low

\quad (i) \, \, $N_1\ll N$ and $N_2, N_4,  N_5 \sim N$

\smallskip

\quad (ii) \, \,   $N_2\ll N$ and $N_1, N_4,  N_5 \sim N$

\smallskip

\quad (iii) \, \,  $N_4\ll N$ and $N_1, N_4,  N_5 \sim N$
\smallskip

\quad (iv) \, \,  $N_5 \ll N$ and $N_1,N_2,  N_4 \sim N$

\medskip 

\item[IVC.]   $N_3\ll N, N_6 \ll N$ and $N_1, N_2, N_4, N_5 \sim N$

\end{enumerate}

\bigskip

 In what follows we will use the following estimates repeatedly:

\begin{lemma}\label{lemma 0}
Let $w_{N_i}$ be as above. Then 
\begin{eqnarray}
\label{i} \| w_{N_i} \|_{X^{0+,\frac{1}{2}-}}\, &\leq& \, N_i^{- \frac{1}{2}+} \,  \| w_{N_i} \|_{X^{\frac{2}{3}-, \frac{1}{2}-}_3}  \\
\label{v} \| w_{N_i} \|_{X^{\frac{1}{2}-,\frac{1}{3}+}}\, &\leq& \,  \,  \| w_{N_i} \|_{X^{\frac{2}{3}-, \frac{1}{2}-}_3}.  
\end{eqnarray}
We also have that
\begin{equation}\label{vi}\| w_{N_i} \|_{L^{8}_{xt}} \, \leq \,    \| w_{N_i} \|_{X^{\frac{13}{24}+, \frac{3}{8}+}_3}. \end{equation}
If we assume that $\sigma_i\lesssim N^{\gamma}$,  for any $\gamma>0$, then
\begin{equation}\label{ii}\| w_{N_i} \|_{L^{\infty}_{xt}} \, \leq \, N^{0+} \,  \| w_{N_i} \|_{X^{\frac{2}{3}-, \frac{1}{2}-}_3}. \end{equation}
\end{lemma}
\begin{proof}
The estimates \eqref{i} and  \eqref{v} are  a consequence of frequency localization and H\"older's inequality. The estimate \eqref{ii} is a consequence of Sobolev embedding together with the assumption  that 
$\sigma_i\leq N^\gamma$.
\end{proof}

\begin{lemma}\label{lemmacloseparab}
Let  $0<\beta<2$, $\rho \geq 0$ and $\delta >0$. Let $M>0$ and $w_M$ be such that $\supp w_M(\cdot, x)\subset [-\delta,\delta], \, x \in \T$. Then if we define 
$$ \widehat{J_\beta w_M}(\tau,n):=\chi_{\{|n|\sim M\}}\chi_{\{|\tau+n^2|\leq M^\beta\}}|\widehat {w_M}(\tau,n)|,$$ we have 
\begin{equation}\label{closeparab}
\|J_\beta w_M\|_{X^{0,\rho}}\lesssim C_{\delta} \,A(\beta,M)^{\frac{1}{6}}M^{\rho\beta+}\|w_M\|_{X^{0,\frac{1}{6}}_3},
\end{equation}
where $A(M, \beta)$ defined below is bounded by  $1 + M^{\beta -1}$. 
\end{lemma}

\begin{proof} We write 

\begin{eqnarray} \label{meas}
\| J_{\beta} w_M \|_{X^{0, \rho}}^2 \, &=&\, \sum_{|n| \sim M} \, \int_{|\tau +n^2| \leq M^{\beta}}  \, |\widehat {w_M}(\tau,n)|^2 \langle \tau + n^2 \rangle^{2 \rho} \, d\tau \notag \\
&\leq& M^{2 \rho \beta} \, \int_{\tau} \biggl( \sum_{ |n| \sim M,  \, |\tau +n^2| \leq M^{\beta} }  |\widehat {w_M}(\tau,n)|^2 \biggr) \, d \tau \notag \\
&\leq & M^{2 \rho \beta}  \int_{\tau}  \biggr[ \sum_{ |n| \sim M,  \, |\tau + n^2| \leq M^{\beta} }  |\widehat {w_M}(\tau,n)|^3 \biggr]^{\frac{2}{3}} \, | S(\tau, M, \beta)|^{\frac{1}{3}} \, d\tau, 
\end{eqnarray}
where 

\begin{equation} S(\tau, M, \beta): = \{ n \in \Z \, :\,  |n| \sim M \, \text{ and }\, |\tau +n^2| \leq M^{\beta} \}.  \end{equation} and $|S|$ represents the counting measure of the set. 

We will show below that 

\begin{equation} \label{claim} \quad  A(M, \beta) := \sup_{\tau} \, |  S(\tau, M, \beta) | \le 1 + M^{\beta -1} \end{equation} 

Hence  \eqref{meas} is less than or equal to 
\begin{eqnarray*}
& &A(M,\beta)^{\frac{1}{3}}M^{2\rho\beta}\int_\tau\biggl[    \sum_{n}   \chi_{\{|n| \sim M\}}(n),  \, \chi_{\{|\tau + n^2| \leq M^{\beta}\}}(\tau,n)|\widehat {w_M}(\tau,n)|^3        \biggr]^{\frac{2}{3}}\,d\tau\\
&=&A(M,\beta)^{\frac{1}{3}}M^{2\rho\beta}\int_\tau\left\|  \biggl\{  \chi_{\{|\tau + n^2| \leq M^{\beta}\}}(\tau,n)\widehat {w_M}(\tau,n) \biggr\}_n   \right\|^2_{\ell^3(|n|\sim M)}\, d\tau\\
&\sim& 
A(M,\beta)^{\frac{1}{3}}M^{2\rho\beta}\int_t\left\|  \mathcal F^{-1}_\tau\biggl(\biggl\{  \chi_{\{|\tau+n^2| \leq M^{\beta}\}}(\tau,n)\widehat {w_M}(\tau,n) \biggr\}_n \biggr)(t) \right\|^2_{\ell^3(|n|\sim M)}\, dt\\
&=&A(M,\beta)^{\frac{1}{3}}M^{2\rho\beta}\int_t\left\| \biggl\{ \mathcal F^{-1}_\tau\biggl( \chi_{\{|\tau+ n^2| \leq M^{\beta}\}}(\tau,n)\biggr) *    \mathcal F^{-1}_\tau\biggl(\widehat {w_M}(\tau,n)  \biggr)\biggr\}_n(t) \right\|^2_{\ell^3(|n|\sim M)}\, dt.
\end{eqnarray*}
Note that $\mathcal F^{-1}_\tau\biggl(\widehat {w_M}(\cdot,n)\biggr)(t)$ is still supported on $[-\delta,\delta]$ for all $n$ and 
\begin{equation}\label{calc1}
\mathcal F^{-1}_\tau\biggl( \chi_{\{|\tau+ n^2| \leq M^{\beta}\}}(\cdot,n)\biggr)(t)= 2e^{-itn^2}\, \frac{\sin(M^\beta t)}{t}.
\end{equation}
We then continue the chain of inequalities from the expression above with 
\begin{eqnarray*} 
&=&A(M,\beta)^{\frac{1}{3}}M^{2\rho\beta}\int_t \left\| \int_\R \chi_{[-\delta,\delta]}(t')   \mathcal F_n (w_M(t',\cdot))(n)  e^{-i(t-t')n^2}\, \frac{\sin(M^\beta (t-t'))}{t-t'}
dt' \right\|^2_{\ell^3(|n|\sim M)}\, dt\nonumber\\
&\leq&A(M,\beta)^{\frac{1}{3}}M^{2\rho\beta}\int_\R \biggl[ \int_\R \chi_{[-\delta,\delta]}(t') \| \mathcal F_n (w_M(t',\cdot))(n) \|_{\ell^3(|n|\sim M)} \left| \frac{\sin(M^\beta (t-t'))}{t-t'}\right|
dt'  \biggr]^2\, dt.
  \end{eqnarray*}
Let $p=2-$ and $q=1+$, then we compute
\begin{eqnarray} 
\left\| \frac{\sin(M^\beta (t-t'))}{t-t'} \right\|_{L^q_t} \, &=&\, M^{\beta} \biggl( \int_{\R} \, \left|\frac{\sin(M^\beta t)}{t M^{\beta}} \right |^q\, dt \biggr)^{\frac{1}{q}}\nonumber \\
&=& M^{\beta} M^{-\frac{\beta}{q}} \,  \biggl( \int_{\R} \, \left | \frac{\sin(r)}{r}\right |^q \, dr \biggr)^{\frac{1}{q}} \lesssim M^{0+}.\label{calc2}
\end{eqnarray}
On the other hand for $\frac{1}{\gamma}=\frac{1}{p}-\frac{1}{3}$
\begin{eqnarray}
\left\|\chi_{[-\delta,\delta]}(\cdot) \| \mathcal F_n (w_M(t,\cdot))(n) \|_{\ell^3(|n|\sim M)} \right\|^2_{L^p_t}
&\lesssim &\delta^\frac{2}{\gamma}\left\|\left\| \mathcal F_n (w_M(t,\cdot))(n)\right\|_{\ell^3(|n|\sim M)} \right\|^2_{L^3_t}\nonumber\\
&\lesssim& \delta^\frac{2}{\gamma}\left\|\left\| e^{itn^2}\mathcal F_n (w_M(t,\cdot))(n)\right\|_{\ell^3(|n|\sim M)} \right\|^2_{L^3_t}\nonumber\\
&=&\delta^\frac{2}{\gamma}\left\|\left\| e^{itn^2}\mathcal F_n (w_M(t,\cdot))(n)\right\|_{L^3_t}\right\|^2_{\ell^3(|n|\sim M)}\nonumber\\
&\lesssim& \delta^\frac{2}{\gamma}\left\|\left\| e^{itn^2}\mathcal F_n (w_M(t,\cdot))(n)\right\|_{H^\frac{1}{6}_t}\right\|^2_{\ell^3(|n|\sim M)}\nonumber\\
&=&\delta^\frac{2}{\gamma}\left\|w_M\right\|^2_{X^{0,\frac{1}{6}}_3},\label{calc3}
\end{eqnarray}
where we used the Sobolev theorem and the definition of  $X^{s,b}_r$. Finally 
by Young's inequality, \eqref{calc2} and  \eqref{calc3} we have the desired estimate. 

It remains to show \eqref{claim}. We use an argument similar to \cite{DPST}. For fixed $\tau$ let $S:=S(\tau, M, \beta)\ne \emptyset$, then there exists $n_0\in S$ and hence
\begin{equation}|S|\leq 1+|\{l \in \Z /  |\, n_0+ l \,  |\sim M, \, |\tau+(n_0+l)^2|\leq M^\beta\}|\leq 1+|\{   l\in \Z \, /\, |l|\leq M, \, |2n_0l+l^2|\lesssim M^\beta    \}|.\label{set}
\end{equation}

\begin{equation} 
 |2n_0l+l^2| = | (l + n_0)^2 - n_0^2 | \lesssim M^\beta \quad   \text{ if and only if }   \quad -CM^{\beta} + n_0^2 \leq  (l + n_0)^2 \leq n_0^2 + C M^{\beta}
\end{equation}
Hence we need  $| \, l \, | \leq M$ to satisfy  
\begin{eqnarray*} 
&&- \sqrt{ n_0^2 + C M^{\beta} } \, \leq \, (l + n_0 ) \, \leq  \, \sqrt{ n_0^2 + C M^{\beta}},  \\
&& (l + n_0)\,  \geq \, \sqrt{ n_0^2 - C M^{\beta} } \quad  \text{ or } \quad 
 ( l + n_0 ) \, \leq \, - \sqrt{n_0^2 - C M^{\beta}}.
 \end{eqnarray*} In other words we need to know the size of  
 $$[ - \sqrt{n_0^2 + C M^{\beta}}, \, - \sqrt{ n_0^2 - CM^{\beta}}] \, \cup \, [ \sqrt{n_0^2 - C M^{\beta}}, \,  \sqrt{ n_0^2 + CM^{\beta}}]$$ which is of the order of $\frac{M^{\beta}}{|n_0|}$. Hence since $|n_0| \sim M$ we have that 
\begin{equation} 
|S| \leq 1 + M^{\beta-1} 
\end{equation}
which implies \eqref{claim} by taking $\sup_{\tau}$.

\end{proof}

In what follows we are under the assumption that $\sigma_j\lesssim N^7$ for all $j=1, \dots, 6$. Towards the end of the proof  we remove this assumption.
We begin by treating all cases with at least two high frequencies in the non derivative terms.  All cases in [IC],  [ID] [IE] [IIB] [IIC] [IID] [IIIB] [IIIC] [IIID] [IVA] [IVB] [IVC]
follow from the following lemma applied with the exponent $\sigma$ appearing below set equal to $0$.
\vskip .1in

\begin{lemma}\label{2nonderivatives} 

 Assume there are ${i, j} \in \{1,2,4,5 \}$ such that $N_i \ge N^{1- \sigma}$ for $ 0 \leq  \sigma < \frac{1}{6}$  and $N_j \sim N$ then 
\eqref{start} can be estimated by $N^{-\frac{1}{12}+ \frac{\sigma}{2} }\, \prod_{i=1}^6 \, \Vert w_i\Vert_{X^{s,b}_3}$.
\end{lemma}

\begin{proof}  By Plancherel we have that \eqref{start} is less than or equal to
\begin{equation} \label{space} \sum_{N_j \sim N;  N_i \ge  N^{1-\sigma}; \, N_k \leq N,\, 1\le k \leq 6 }  \int_{\mathbb{R}} \, \int_{\mathbb{T}} \,\,   N_3\, N_6 \, \, \, w_{N_1} \, w_{N_2}  \, \cj{w_{N_3}} \, w_{N_4}\,   \cj{ w_{N_5}}\,  \cj{w_{N_6}}\, \, \, dx \, dt. \end{equation} 

Let $0< \beta < 1$ to be determined below. Assume 

\begin{equation} \label{sigma3} \sigma_3 \leq N_3^{\beta}.\end{equation} By Cauchy-Schwarz's inequality,  grouping 
the first three functions in \eqref{space} in $L^2_{xt}$  and the last three in $L^2_{xt}$ and using \eqref{iii}  we have that  \eqref{space} is less than or equal to
\begin{equation} \label{space2} \sum_{ N_j \sim N;  N_i \ge  N^{1-\sigma}; \, N_k \leq N  }\, \, \,  N_3\, N_6   \prod_{i=1}^6 \Vert w_{N_i} \Vert_{X^{\epsilon, \frac{1}{2}-} }. \end{equation} 

Note now that by  \eqref{sigma3}   $ w_{N_3}$ is equal to  $J_{\beta} w_{N_3} $  as defined in Lemma \ref{lemmacloseparab} above. Then we have

\begin{equation} \label{lemma1}  \Vert  w_{N_3} \Vert_{X^{\epsilon, \frac{1}{2}-} } \leq  C_{\delta}\,  N_3^{\frac{1}{2} \beta+} \, \Vert  w_{N_3} \Vert_{X_3^{0, \frac{1}{6}+} }. \end{equation}

Hence by \eqref{i}, \eqref{lemma1} we have that \eqref{space2} is less than or equal to

\begin{equation}   \sum_{   N_j \sim N;  N_i \ge  N^{1-\sigma}; \, N_k \leq N  }\, \, \,  N_3\, N_6    N_1^{-\frac{1}{2} +} N_2^{-\frac{1}{2} +}N_3^{\frac{1}{2} \beta+ } N_3^{-\frac{2}{3}} N_4^{-\frac{1}{2} +} N_5^{-\frac{1}{2} +} N_6^{-\frac{1}{2}+} \, \biggl( \prod_{i=1}^6\,  \Vert w_{N_i} \Vert_{X^{\frac{2}{3}-, \frac{1}{2}-}_3}\biggr).  \end{equation}

\begin{equation} \lesssim  \sum_{ N_j \sim N;  N_i \ge  N^{1-\sigma}; \, N_k \leq N}\, \, \,  N_3^{\frac{1}{3}+ \frac{\beta}{2}+} \, N^{\frac{1}{2}+}    N^{-1 + \frac{\sigma}{2}} \, \biggl( \prod_{i=1}^6\,  \Vert w_{N_i} \Vert_{X^{\frac{2}{3}-, \frac{1}{2}-}_3} \biggr).  \end{equation}

From here we apply H\"older's inequality  with $r=3, r'=\frac{3}{2}$  to sum in $N_j, N_i, N_k$  (multiply and divide by $N_j^{-\epsilon}$ with a loss of $N^{\epsilon}$ for each term). For example, 
\begin{eqnarray}\label{holder}
\sum_{N_j \leq N}   \Vert w_{N_j} \Vert_{X^{s, b}_3} &=&  \sum_{N_j \leq N} \left\| \Vert \langle n_j\rangle^s \langle \tau+n_j^2 \rangle^b \, {\widehat w_{N_j}}(\tau, n_j)  \Vert_{L^2_{\tau}}  \right\|_{\ell^3}. 
\end{eqnarray}
Set $Y_{N_j}(n_j): = \Vert \langle n_j\rangle^s \langle \tau- n_j^2 \rangle^b \, {\widehat w_{N_j}}(\tau, n_j)  \Vert_{L^2_{\tau}},$ then the expression in  \eqref{holder} equals
\begin{eqnarray}\label{holder1}
 \sum_{N_j \leq N}  N_j^{\varepsilon} N_j^{-\varepsilon}  \Vert Y_{N_j}  \Vert_{\ell^3} &\leq&
N^{\varepsilon}  \biggl( \sum_{N_j \leq N} N_j^{-\frac{3}{2} \varepsilon} \biggr)^{\frac{2}{3}} \biggl( \sum_{N_j \leq N}  \Vert Y_{N_j}  \Vert^3_{\ell^3}  \biggr)^{\frac{1}{3}} \notag \\
&\lesssim& N^{\varepsilon} \biggl( \sum_{N_j \leq N}  \sum_{|n_j| \sim N_j} \Vert  \langle n_j \rangle^s \langle \tau+ n_j^2 \rangle^b \, {\widehat w_{j}}(\tau, n_j)   \Vert^3_{L^2_{\tau}}  \biggr)^{\frac{1}{3}} \\
&\sim&\, N^{\varepsilon} \Vert w_j \Vert_{X^{s, b}_3}.  \notag
\end{eqnarray}
Note then that all in all we get at worst a factor of $N^{-\frac{1}{6}+ \frac{\beta}{2}+ \frac{\sigma}{2}+}$.

Now assume that 
\begin{equation} \label{sigma3>}  \sigma_3 \geq  N_3^{\beta}.\end{equation}
Then rewrite  \eqref{space} as 
\begin{equation} \label{space3} \sum_{N_j \sim N;  N_i \ge  N^{1-\sigma}; \, N_k \leq N }  \int_{\mathbb{R}} \, \int_{\mathbb{T}} \,\,   N_3\, N_6\,| \sigma_3|^{-\frac{1}{2}+}\, \, \, w_{N_1} \, w_{N_2}  \, | \sigma_3|^{\frac{1}{2}-} \cj{w_{N_3}} \, w_{N_4}\,   \cj{ w_{N_5}}\,  \cj{w_{N_6}}\, \, \, dx \, dt. \end{equation} 
We do H\"older by placing $ | \sigma_3|^{\frac{1}{2}-} \cj{w_{N_3}} $ in $L^2_{x,t}$, the product of $\cj{w_{N_6}}$ with the two largest among $w_{N_1}, \, w_{N_2}, \, w_{N_4},\,   \cj{ w_{N_5}}$ in $L^2_{x,t}$ while the remaining ones in $L^\infty_{x,t}$. Then by \eqref{i} and \eqref{ii}, we bound \eqref{space3} by

\begin{eqnarray*} &\lesssim&  \sum_{ N_j \sim N;  N_i \ge  N^{1-\sigma}; \, N_k \leq N}\, \, \,  N_3N_3^{-\frac{1}{2}- \frac{\beta}{2}+} \, N_6N_6^{-\frac{1}{2}+}N^{-\frac{1}{2}+}    N^{-\frac{1}{2} + \frac{\sigma}{2}+} \, \biggl( \prod_{i=1}^6\,  \Vert w_{N_i} \Vert_{X^{\frac{2}{3}-, \frac{1}{2}-}_3} \biggr) \\
&\lesssim&  \sum_{ N_j \sim N;  N_i \ge  N^{1-\sigma}; \, N_k \leq N}\, \, \,  N_3^{\frac{1}{2}- \frac{\beta}{2}+} \,    N^{-\frac{1}{2} + \frac{\sigma}{2}+} \, \biggl( \prod_{i=1}^6\,  \Vert w_{N_i} \Vert_{X^{\frac{2}{3}-, \frac{1}{2}-}_3} \biggr).
 \end{eqnarray*}
We want that $\beta>\sigma$ to conclude by H\"older the desired inequality with a decay in $N$. We now impose that 
$$-\frac{1}{6}+\frac{\beta}{2}+\frac{\sigma}{2}=-\frac{\beta}{2}+\frac{\sigma}{2}, $$
whence $\beta=\frac{1}{6}$ and provided $0<\sigma<\frac{1}{6}$ the lemma follows.
\end{proof}

\bigskip

It remains then to treat cases [IA], [IB], [IIA] and [IIIA]. Before starting we note the following support condition that will be used throughout in what follows.

\medskip

\subsection*{Support Condition}

By \eqref{Size} and \eqref{RelativeSize} the triplet $(w_{N_1}, w_{N_2},  \cj{w_{N_3}} )$ satisfies $n= n_1 + n_2 + n_3$, $|n_j| \leq N,\,  N \leq |n| \leq 3N$ 
and $N \sim \max(N_1, N_2, N_3)$.  

Suppose that -say- $\max(N_1, N_2) \leq N^{\theta}$ for some $ 0 < \theta <1$.  Without any loss of generality assume $n >0$.  Then, we have that $ N\leq n \leq  (n_1 +n_2) + n_3 \leq 2 N^{\theta} + N $ and hence  $n = N + k$ where $0 \leq k \leq 2 N^{\theta}$. Next observe that  $n_3 = n - (n_1 +n_2) = N +k  - (n_1 +n_2)$ with $|k -(n_1 +n_2) | \leq  4 N^{\theta}$, whence $n_3 = N + O(N^{\theta})$. In other words, we have that whenever $\max(N_1, N_2) \le N^{\theta}$ the support of $\widehat{\cj{w_{N_3}}}$ is  of size $O(N^{\theta})$.  Note that we could have just as well said that the support of $\widehat{\cj{w_{N_3}}}$ is of size $O(\max(N_1,N_2))$ in lieu of $O(N^{\theta})$. 

When we are in this situation we say we have the {\it support condition} on $\cj{w_{N_3}}$.  This argument is symmetric with respect to $w_{N_1}, w_{N_2}$ or $\cj{w_{N_3}}$.  The exact same analysis holds for $(w_{N_4}, \cj{w_{N_5}}, \cj{w_{N_6}})$.   By abuse of notation we still write for example, $\widehat{\cj{w_{N_3}}}(n_3)$ for $\widehat{\cj{w_{N_3}}}(n_3) \chi_{I_3}(n_3)$, where $I_3(n_3)$ is the support of $\widehat{\cj{w_{N_3}}}$ when the support condition holds.

\begin{remark} \label{support} \rm As a consequence of the support condition,  estimate \eqref{i} can be improved. For example if we have the support condition on $\widehat{\cj{w_{N_3}}}$ then 
$$\| w_{N_3} \|_{X^{0+, \frac{1}{2}-}} \lesssim | I_3 |^{\frac{1}{6}} \| w_{N_3} \|_{X_3^{0+, \frac{1}{2}-}} $$

\end{remark}

\subsection*{ \bf Case  [IIIA] }  Note that  (i) and (ii) are symmetric with respect to $j=1$ and $j=2$.  So we only consider (i).  Observe also that a priori there is no help from a large $\sigma_j$.  Let $ \sigma, \delta$ be two positive constants to be determined later but such that $1 - \sigma > \delta$.
\medskip 

\underbar{Subcase 1:}  Assume $N_2, N_4, N_5 < N^{1- \sigma}, \, N_3 \lesssim N^{\delta}$ and $N_1 \sim N \sim N_6$ in \eqref{start}.  Then we have the support condition on $w_{N_1}$ and $\cj{w_{N_6}}$.  Let us denote by $\sum_{\ast}$ the sum over the set of $N_j \le N , 1 \le j \le 6$ such that  $N_1, N_6 \sim N$,   $N_j <  N^{1-\sigma}$ for $j=2, 4, 5$ and  $N_3 \lesssim N^{\delta }$.  
By Cauchy-Schwarz, \eqref{iii}, Lemma \ref{lemma 0} and Remark \ref{support} we then have that \eqref{start} is less than or equal to
\begin{eqnarray*}  &&\sum_{\ast} N_3\, N_6 \, \, \,   \max(N_2, N_3)^{\frac{1}{6}} N_1^{-\frac{2}{3}+} N_2^{-\frac{1}{2}+} N_3^{-\frac{1}{2}+}  N_4^{-\frac{1}{2}+} N_5^{-\frac{1}{2}+} \times \\
&&\hspace{4cm}  \times  \max(N_4, N_5)^{\frac{1}{6}} N_6^{-\frac{2}{3}+} \, \biggl( \prod_{i=1}^6\,  \Vert w_{N_i} \Vert_{X^{\frac{2}{3}-, \frac{1}{2}-}_3} \biggr)\\
&\lesssim & \sum_{\ast} N_3^{\frac{1}{2}+} N_6^{\frac{1}{3}+} \max(N_2, N_3)^{\frac{1}{6}} N_2^{-\frac{1}{2}+}  N^{-\frac{2}{3}+}  \, \biggl( \prod_{i=1}^6\,  \Vert w_{N_i} \Vert_{X^{\frac{2}{3}-, \frac{1}{2}-}_3} \biggr) 
 \end{eqnarray*}  since $N_4^{-\frac{1}{2}+} N_5^{-\frac{1}{2}+} \max(N_4, N_5)^{\frac{1}{6}}$ is bounded. On the other hand the latter expression is worst possible when $\max(N_2, N_3) \sim N_3$; hence if $\delta < \frac{1}{2}$ we conclude by H\"older as before with a decay of $ N^{-\frac{1}{3}} N^{\frac{2}{3} \, \delta}$.  
\medskip 

\underbar{Subcase 2:}  Assume $N_2, N_4, N_5 < N^{1- \sigma}, \, N_3 \gtrsim N^{\delta}$ and $N_1 \sim N \sim N_6$ in \eqref{start}. We further subdivide as follows:
\medskip 

{\it Subcase 2a)}  Assume $N_2, N_4, N_5 \ll N^{\delta}, \, N_3 \gtrsim N^{\delta}$ and $N_1 \sim N \sim N_6$ in \eqref{start}. Then from \eqref{AAlgebra2}  there exists $\sigma_j \gtrsim  N^{1+\delta}$. Denote by $\sum_{\ast}$ the sum over the set of $N_j \le N , 1 \le j \le 6$ such that  $N_1, N_6 \sim N$,   $N_j <  N^{\delta}$ for $j=2, 4, 5$ and  $N_3 \geq N^{\delta }$.  

\medskip 

$\bullet$ Suppose $j=2,4$ or $5$; $j=2$ or $4$ are symmetric. So we treat first $j=2$ and then $j=5$. By Plancherel we have that \eqref{start} is less than or equal to 
\begin{eqnarray*}
&&\sum_{\ast}    \int_\R\int_\T N_3  N_6\, \sigma_2^{-\frac{1}{2}+} \,  w_{N_1} \, \sigma_2^{\frac{1}{2}-} w_{N_2}  \cj{ w_{N_3}}  w_{N_4}   \cj{ w_{N_5}} \cj{ w_{N_6}}\, dxdt\\
&\lesssim& \sum_{\ast} N_3^{\frac{1}{2}+}   N_6^{\frac{1}{2}+}  N^{-\frac{1}{2}-\frac{\delta}{2}}  N_1^{-\frac{1}{2}+}  N_2^{-\frac{1}{2}+} N_4^{0+}\, N_5^{0+}  \, \biggl( \prod_{i=1}^6\,  \Vert w_{N_i} \Vert_{X^{\frac{2}{3}-, \frac{1}{2}-}_3} \biggr)\
\end{eqnarray*} by Cauchy Schwarz placing $w_{N_1}  \cj{ w_{N_3}} \cj{ w_{N_6}}$ in $L^2$, $\sigma_2^{\frac{1}{2}} w_{N_2}$ in $L^2$ and $ w_{N_4}   \cj{ w_{N_5}} $ in $L^{\infty}$.  From \eqref{iii} and Lemma \ref{lemma 0} we obtain the desired estimate with decay $N^{-\frac{\delta}{2}}$ so long as $\delta >0$. 

If $j=5$ we proceed as above with same grouping in $L^2$ but exchanging the roles of $w_{N_2}$ and $\cj{w_{N_5}}$ for the other $L^2$ and one of the $L^{\infty}$ bounds. 

\medskip 

$\bullet$ Suppose $j=3,6$ or $1$; $j=3$ or $6$ are symmetric. So we treat first $j=3$ and then $j=1$.  Proceeding as above from \eqref{start} we now have
\begin{eqnarray*}
&&\sum_{\ast} \int_\R\int_\T N_3  N_6\, \sigma_3^{-\frac{1}{2}+} \,  w_{N_1} \,  w_{N_2} \sigma_3^{\frac{1}{2}-} \cj{ w_{N_3}}  w_{N_4}   \cj{ w_{N_5}} \cj{ w_{N_6}}\, dxdt\\
&\lesssim&\sum_{\ast} N_3^{\frac{1}{2}+}   N_6^{\frac{1}{2}+}  N^{-\frac{1}{2}-\frac{\delta}{2}}   N_1^{-\frac{1}{2}+} N_2^{0+}\, N_4^{-\frac{1}{2}+} N_5^{0+}  \, \biggl( \prod_{i=1}^6\,  \Vert w_{N_i} \Vert_{X^{\frac{2}{3}-, \frac{1}{2}-}_3} \biggr)\
\end{eqnarray*} 
by Cauchy Schwarz placing $w_{N_1}  w_{N_4} \cj{ w_{N_6}}$ in $L^2$, $\sigma_3^{\frac{1}{2}-} \cj{w_{N_3}}$ in $L^2$ and $ w_{N_2}   \cj{ w_{N_5}} $ in $L^{\infty}$.  We thus obtain the desired estimate as before with decay $N^{-\frac{\delta}{2}}$ so long as $\delta >0$. 

If $j=1$ then we group  $\cj{w_{N_3}}  w_{N_4} \cj{ w_{N_6}}$ in $L^2$,  $\sigma_1^{\frac{1}{2}-} w_{N_1}$ in $L^2$ and the other two on $L^{\infty}$ to reach the same estimate.
 
\bigskip

{\it Subcase 2b)}  Suppose there exists $i \in \{2, 4, 5 \}$ such that $N_i \gtrsim N^{\delta}$ and $N_j \ll N^{\delta}$ for $j \neq i$, and $i,  j \in \{ 2,4, 5 \}$ while still  $N_3 \gtrsim N^{\delta}$ and $N_1 \sim N \sim N_6$ in \eqref{start}.  

\medskip

$\bullet$ \, Suppose $i=2$ first. Then we further split the sum over this set into three sums, $S_1, S_2$ and $S_3$ according to whether $N^{\delta} \lesssim N_2 \ll N_3; \,  \, N_2 \sim N_3$  \, or  $ \, N_2 \gg N_3$ respectively.  When considering the sums over $S_1$ or over $S_3$ we have that from \eqref{AAlgebra2} there exists $\sigma_j \gtrsim  N^{1+\delta}$ and hence the estimates for $S_1$ and $S_3$ follow exactly as those in {\it Subcase 2a)}. 

We treat then $S_2$. Since $N_2 \sim N_3$ and $N_2 < N^{1 - \sigma}$, we also have $N_3 < N^{1-\sigma}$; while $N_4, N_5 \lesssim N^{\delta}$.  Thus we have the support condition in $w_{N_1}$ and $\cj{w_{N_6}}$. Then from \eqref{start} by Cauchy-Schwarz, \eqref{iii}, Lemma \ref{lemma 0} and Remark \ref{support},  grouping $w_{N_1} \,  w_{N_2} \cj{ w_{N_3}}$ in $L^2$ and  $w_{N_4}   \cj{ w_{N_5}} \cj{ w_{N_6}}$ and \eqref{i} we have 
\begin{eqnarray*} 
&&\sum_{S_2}  N_3 N_6 \max(N_2, N_3)^{\frac{1}{6}} N_1^{-\frac{2}{3}+}  N_2^{-\frac{1}{2}+}  N_3^{-\frac{1}{2}+}  N_4^{-\frac{1}{2}+}  N_5^{-\frac{1}{2}+} \,\times\\
&&\hspace{4cm}\times \max(N_4, N_5)^{\frac{1}{6}} N_6^{-\frac{2}{3}+} \, 
\biggl( \prod_{i=1}^6\,  \Vert w_{N_i} \Vert_{X^{\frac{2}{3}-, \frac{1}{2}-}_3} \biggr) \\
&\lesssim&  \sum_{S_2}     N_6^{\frac{1}{3}+}\max(N_2, N_3)^{\frac{1}{6}}  N_1^{-\frac{2}{3}+}   \,    \biggl( \prod_{i=1}^6\,  \Vert w_{N_i} \Vert_{X^{\frac{2}{3}-, \frac{1}{2}-}_3} \biggr) 
\end{eqnarray*} since $  N_4^{-\frac{1}{2}+}  N_5^{-\frac{1}{2}+}  \max(N_4, N_5)^{\frac{1}{6}}$ is bounded and $N_2 \sim N_3$.  Summing as usual, we get the desired estimate with 
decay $N^{-\frac{1}{6}+}$ regardless of $\sigma>0$. 

\medskip 

$\bullet$ \, Suppose $i=4$.  Again, we further split the sum over this set into three sums, $S_1, S_2$ and $S_3$ according now to whether $N^{\delta} \lesssim N_4 \ll  N_3; \, N_4 \sim N_3$  \, or  $ \, N_4 \gg  N_3$ respectively.  For the sums over $S_1$ or over $S_3$, from \eqref{AAlgebra2}  we have a $\sigma_j \gtrsim  N^{1+\delta}$ and hence the estimates for $S_1$ and $S_3$ follow exactly as those in {\it Subcase 2a)}. 

We treat then $S_2$. Since $N_4 \sim N_3$, $N_3 < N^{1-\sigma}$ while $N_2, N_5 \lesssim N^{\delta}$; so once again we have a support condition in $w_{N_1}$ and $\cj{w_{N_6}}$. Proceeding as before we have 
\begin{eqnarray*} 
&&\sum_{S_2}  N_3 N_6 \max(N_2, N_3)^{\frac{1}{6}} N_1^{-\frac{2}{3}+}  N_2^{-\frac{1}{2}+}  N_3^{-\frac{1}{2}+}  N_4^{-\frac{1}{2}+}  N_5^{-\frac{1}{2}+}  \times\\
&& \hspace{4cm} \times \max(N_4, N_5)^{\frac{1}{6}} N_6^{-\frac{2}{3}+} \,  \biggl( \prod_{i=1}^6\,  \Vert w_{N_i} \Vert_{X^{\frac{2}{3}-, \frac{1}{2}-}_3} \biggr) \\
&\lesssim& \sum_{S_2}   N_3^{\frac{1}{2}+} N_6 N_3^{\frac{1}{6}} N^{-\frac{2}{3}+}  N_2^{-\frac{1}{2}+}   N_4^{-\frac{1}{2} + \frac{1}{6}+} N_5^{-\frac{1}{2}+} N^{-\frac{2}{3}+} \,  \biggl( \prod_{i=1}^6\,  \Vert w_{N_i} \Vert_{X^{\frac{2}{3}-, \frac{1}{2}-}_3} \biggr) \\
&\lesssim& \sum_{S_2}   N_3^{\frac{1}{3}+} N \, N^{-\frac{4}{3}+}  \biggl( \prod_{i=1}^6\,  \Vert w_{N_i} \Vert_{X^{\frac{2}{3}-, \frac{1}{2}-}_3} \biggr). 
\end{eqnarray*}  Since $N_4 \sim N_3$ and $N_3 < N^{1 -\sigma}$ summing as before we have the desired estimate with decay $N^{-\frac{\sigma}{3}}$ so long as $\sigma>0$.

\medskip

$\bullet$ \, Suppose $i=5$.  We now split the sum over this set into three sums, $S_1, S_2$ and $S_3$ according to whether $N^{\delta} \lesssim N_5 \ll  N_3; \, \, N_5 \sim N_3$  \, or  $ \, N_5 \gg N_3$ respectively.  Again for the sums over $S_1$ or over $S_3$, from \eqref{AAlgebra2}  we have a $\sigma_j \gtrsim  N^{1+\delta}$ and hence the estimates for $S_1$ and $S_3$ follow exactly as those in {\it Subcase 2a)}. 

We treat then $S_2$. Since $N_5 \sim N_3$, $N_3 < N^{1-\sigma}$ while $N_2, N_4 \lesssim N^{\delta}$;  we have a support condition in $w_{N_1}$ and $\cj{w_{N_6}}$. Proceeding as before we have 

\begin{eqnarray*} 
&&\sum_{S_2}  N_3 N_6 \max(N_2, N_3)^{\frac{1}{6}} N_1^{-\frac{2}{3}+}  N_2^{-\frac{1}{2}+}  N_3^{-\frac{1}{2}+}  N_4^{-\frac{1}{2}+}  N_5^{-\frac{1}{2}+} \times\\
&&\hspace{4cm}\times  \max(N_4, N_5)^{\frac{1}{6}} N_6^{-\frac{2}{3}+} \,  \biggl( \prod_{i=1}^6\,  \Vert w_{N_i} \Vert_{X^{\frac{2}{3}-, \frac{1}{2}-}_3} \biggr) \\
&\lesssim& \sum_{S_2}   N_3^{\frac{1}{2}+} N_6^{\frac{1}{3}+}  N_3^{\frac{1}{6}} N^{-\frac{2}{3}+}  N_2^{-\frac{1}{2}+}   N_4^{-\frac{1}{2}+} N_5^{-\frac{1}{2}+ \frac{1}{6}+}   \biggl( \prod_{i=1}^6\,  \Vert w_{N_i} \Vert_{X^{\frac{2}{3}-, \frac{1}{2}-}_3} \biggr)\\
&\lesssim& \sum_{S_2}   N_3^{\frac{1}{3}+} N^{-\frac{1}{3}+}  \biggl( \prod_{i=1}^6\,  \Vert w_{N_i} \Vert_{X^{\frac{2}{3}-, \frac{1}{2}-}_3} \biggr)
\end{eqnarray*} which summing  over $S_2$ gives the desired estimate with the same $N^{-\frac{\sigma}{3}}$ decay as in the previous case so long as $\sigma>0$.

\bigskip

{\it Subcase 2c)} Suppose that there exist at least $i, j \in \{ 2, 4, 5\}$  \, ($i \neq j$) such that $N_i, N_j \gtrsim   N^{\delta}$ while  $N_3 \gtrsim N^{\delta}$ and $N_1 \sim N \sim N_6$ in \eqref{start}.  Note that $N_4, N_5 < N^{1-\sigma}$ which ensures a support condition on $\cj{w_{N_6}}$.

\medskip

$\bullet$ \, Suppose $(i,j)=(4,5)$. Proceeding as above and using similar arguments we have
\begin{eqnarray*}
&&\sum_{\ast}  N_3\, N_6 \, N_1^{-\frac{1}{2}+} N_2^{-\frac{1}{2}+} N_3^{-\frac{1}{2}+} N_4^{-\frac{1}{2}+} N_5^{-\frac{1}{2}+} \max(N_4, N_5)^{\frac{1}{6}} N_6^{-\frac{2}{3}+}   \,  \biggl( \prod_{i=1}^6\,  \Vert w_{N_i} \Vert_{X^{\frac{2}{3}-, \frac{1}{2}-}_3} \biggr) \\
&=&  \sum_{\ast}   \, N_3^{\frac{1}{2}+}  N_6^{\frac{1}{3}+}\,  N^{-\frac{1}{2}+}   \,  N_4^{-\frac{1}{2}+} N_5^{-\frac{1}{2}+} \max(N_4, N_5)^{\frac{1}{6}} \biggl( \prod_{i=1}^6\,  \Vert w_{N_i} \Vert_{X^{\frac{2}{3}-, \frac{1}{2}-}_3} \biggr) 
\end{eqnarray*} from where using that  $N_4, N_5 \gtrsim N^{\delta}$ and $N_3 \gtrsim  N^{\delta}$  we get the desired bound with decay $N^{\frac{1}{3} - \frac{5}{6} \delta}$ so long as 
$\delta > \frac{2}{5}$.
\medskip 

$\bullet$ \, Suppose $(i,j)=(2,5)$.  Once again proceeding as before and using similar arguments we have
\begin{eqnarray*}
&&\sum_{\ast} N_3\, N_6 \, N_1^{-\frac{1}{2}+} N_2^{-\frac{1}{2}+} N_3^{-\frac{1}{2}+} N_4^{-\frac{1}{2}+} N_5^{-\frac{1}{2}+} \max(N_4, N_5)^{\frac{1}{6}} N_6^{-\frac{2}{3}+}   \,  \biggl( \prod_{i=1}^6\,  \Vert w_{N_i} \Vert_{X^{\frac{2}{3}-, \frac{1}{2}-}_3} \biggr) \\
&\lesssim&  \sum_{\ast}  N_3^{\frac{1}{2}+}  N_6^{\frac{1}{3}+}\,  N^{-\frac{1}{2}+} N^{-\frac{\delta}{2}} N^{-\frac{\delta}{2} + \frac{\delta}{6}} \, \,  \biggl( \prod_{i=1}^6\,  \Vert w_{N_i} \Vert_{X^{\frac{2}{3}-, \frac{1}{2}-}_3} \biggr) 
\end{eqnarray*} using that $N_2 \gtrsim N^{\delta}$ and that $N_4^{-\frac{1}{2}+} N_5^{-\frac{1}{2}+} \max(N_4, N_5)^{\frac{1}{6}}$ is worse possible when $N_4 \ll N_5$  but $N_5 \gtrsim N^{\delta}$. Hence we once again obtain the desired estimate with decay $N^{\frac{1}{3} - \frac{5}{6} \delta}$ so long as 
$\delta > \frac{2}{5}$.
\medskip 

$\bullet$ \, Suppose $(i,j)=(2, 4)$. This is exactly as in the previous case by exchanging the roles of $4$ and $5$.

\bigskip

\underbar{Subcase 3:}  Assume there exists at least one $i \in \{2, 4, 5\}$ such that $N_i \gtrsim N^{1-\sigma}$,  $N_2, N_4, N_5 \ll  N$ while $N_3 \ll N$ and $N_1 \sim N\sim N_6$ in \eqref{start}.   This case follows from Lemma \ref{2nonderivatives} with $0 < \sigma < \frac{1}{6}$ as in its statement.  
\medskip

All in all,  for Case [IIIA] we need $\frac{2}{5} < \delta < \frac{1}{2}$ and $ 0 < \sigma < \frac{1}{6}$. 

\begin{remark}\rm In the proof of the remaining cases, in order to keep the notation lighter, we will ignore the $+\epsilon$ appearing in the exponent of the $N_i$'s  in \eqref{i}. For example we simply write 
$N_i^{-\frac{1}{2}}$ instead of $N_i^{-\frac{1}{2}+}$. 
\end{remark}

\subsection*{ \bf Case  [IA] }  Assume  $N_3 \sim N \sim N_6$  while $N_1, N_2, N_4, N_5 \ll  N$ in \eqref{start} and denote as before by $\sum_{\ast}$ the sum over this set. 
Observe that from \eqref{Algebra1}-\eqref{AAlgebra3} there exists $\sigma_j \gtrsim N^2$. 

\medskip 

\underbar{Subcase 1:}  Assume  in addition $N_1, N_2 < N^{\delta}$ for some $\delta >0$.  We then have the support condition on $\cj{w_{N_3}}$.  

\medskip 

$\bullet$ Suppose $j=3$ or $6$; say $j=3$ ($j=6$ is similar). Then we rewrite \eqref{start} as follows:
\begin{eqnarray*}
&&\sum_{\ast}  \int_\R\int_\T N^2 \,  \sigma_3^{-\frac{1}{2}+} \,  w_{N_1} \, 
w_{N_2}  \sigma_3^{\frac{1}{2}-}\cj{ w_{N_3}}  w_{N_4}   \cj{ w_{N_5}} \cj{ w_{N_6}}\, dxdt\\
&\lesssim& \sum_{\ast} \, N^2 N^{-1} N_1^{0+} N_2^{0+}  \max(N_1, N_2)^{\frac{1}{6}} N_3^{-\frac{2}{3}} N_4^{-\frac{1}{2}} N_5^{-\frac{1}{2}} N_6^{-\frac{1}{2}} \,  \biggl( \prod_{i=1}^6\,  \Vert w_{N_i} \Vert_{X^{\frac{2}{3}-, \frac{1}{2}-}_3} \biggr) 
\end{eqnarray*} by placing  $ \sigma_3^{\frac{1}{2}-}\cj{ w_{N_3}} $ in $L^2_{xt}$, $w_{N_4}   \cj{ w_{N_5}} \cj{ w_{N_6}}$ in $L^2_{xt}$, $w_{N_1} w_{N_2} $ in $L^{\infty}_{xt}$ and using the support condition on $\cj{ w_{N_3}}$.  By H\"older's inequality, summing as above,  we get the desired estimate with decay $N^{\frac{\delta}{6}-\frac{1}{6}}$ so long as  
$\delta <1$. 

\medskip

$\bullet$ Suppose $j=1, 2, 4 $ or $5$. By symmetry (relative to conjugates) $j=1, 2, 4$ are similar; so suppose $j=1$. We rewrite \eqref{start} as 
\begin{eqnarray*}
&&\sum_{\ast}   \int_\R\int_\T N^2 \,  \sigma_1^{-\frac{1}{2}+} \,  w_{N_1}   \sigma_1^{\frac{1}{2}-} \, 
w_{N_2}  \cj{ w_{N_3}}  w_{N_4}   \cj{ w_{N_5}} \cj{ w_{N_6}}\, dxdt\\
&\lesssim&  \sum_{\ast} \, N^2 N^{-1} N_1^{-\frac{1}{2}} N_2^{0+}  \max(N_1, N_2)^{\frac{1}{6}} N_3^{-\frac{2}{3}} N_4^{-\frac{1}{2}} N_5^{0+} N_6^{-\frac{1}{2}} \,  \biggl( \prod_{i=1}^6\,  \Vert w_{N_i} \Vert_{X^{\frac{2}{3}-, \frac{1}{2}-}_3} \biggr) 
\end{eqnarray*} by placing  $ \sigma_1^{\frac{1}{2}-} w_{N_1}  $ in $L^2_{xt}$, \,  $ \cj{ w_{N_3}}  w_{N_4}  \cj{ w_{N_6}}$ in $L^2_{xt}$, $w_{N_2}  \cj{w_{N_5}} $ in $L^{\infty}_{xt}$ and using the support condition on $\cj{ w_{N_3}}$.  Once again,  by H\"older's inequality,  summing as before we get the desired estimate with decay $N^{\frac{\delta}{6}-\frac{1}{6}}$ so long as  $\delta <1$. 

\medskip 
If $j=5$ \begin{eqnarray*}
&&\sum_{\ast}   \int_\R\int_\T N^2 \,  \sigma_5^{-\frac{1}{2}+} \,  w_{N_1}   \, 
w_{N_2}  \cj{ w_{N_3}}  w_{N_4}   \sigma_5^{\frac{1}{2}-}  \cj{ w_{N_5}} \cj{ w_{N_6}}\,dxdt\\
&\lesssim& \sum_{\ast} \, N^2 N^{-1} N_1^{0+} N_2^{0+}  \max(N_1, N_2)^{\frac{1}{6}} N_3^{-\frac{2}{3}} N_4^{-\frac{1}{2}} N_5^{-\frac{1}{2}} N_6^{-\frac{1}{2}} \,  \biggl( \prod_{i=1}^6\,  \Vert w_{N_i} \Vert_{X^{\frac{2}{3}-, \frac{1}{2}-}_3} \biggr) 
\end{eqnarray*} by placing  $ \sigma_5^{\frac{1}{2}-} \cj{w_{N_5}}  $ in $L^2_{xt}$, \,  $ \cj{ w_{N_3}}  w_{N_4}  \cj{ w_{N_6}}$ in $L^2_{xt}$, $w_{N_1}  w_{N_2} $ in $L^{\infty}_{xt}$ and using the support condition on $\cj{ w_{N_3}}$.  Once again, H\"older's inequality,  summing as before we get the desired estimate with decay $N^{\frac{\delta}{6}-\frac{1}{6}}$ so long as  $0< \delta <1$.

\medskip 

\underbar{Subcase 2:}  Assume either  $N_1$ or  $N_2 > N^{\delta}$.  Suppose $N_1 > N^{\delta}$; otherwise exchange the roles of $w_{N_1}$ and $w_{N_2}$ below. We no longer rely on the support condition but on the lower bound on $N_1$ as follows.  
\medskip 

$\bullet$ Suppose $j=3$ or $6$; say $j=3$ ($j=6$ is similar). Then proceeding as before we rewrite \eqref{start} as 
\begin{eqnarray*}
&&\sum_{\ast}  \int_\R\int_\T N^2 \,  \sigma_3^{-\frac{1}{2}+} \,  w_{N_1} \, 
w_{N_2}  \sigma_3^{\frac{1}{2}-}\cj{ w_{N_3}}  w_{N_4}   \cj{ w_{N_5}} \cj{ w_{N_6}}\, dxdt\\
&\lesssim&  \sum_{\ast} \, N^2 N^{-1} N_1^{-\frac{1}{2}} N_2^{0+} N_3^{-\frac{1}{2}} N_4^{0+} N_5^{-\frac{1}{2}} N_6^{-\frac{1}{2}} \,  \biggl( \prod_{i=1}^6\,  \Vert w_{N_i} \Vert_{X^{\frac{2}{3}-, \frac{1}{2}-}_3} \biggr) 
\end{eqnarray*} by placing  $ \sigma_3^{\frac{1}{2}-}\cj{ w_{N_3}} $ in $L^2_{xt}$, \,  $w_{N_1}   \cj{ w_{N_5}} \cj{ w_{N_6}}$ in $L^2_{xt}$, $w_{N_2} w_{N_4} $ in $L^{\infty}_{xt}$ .  By H\"older's inequality, summing as above, we get the desired estimate with decay $N^{-\frac{\delta}{2}}$ so long as  $\delta >0$. 

$\bullet$ Suppose $j=1$ or $2$; say $j=1$ ($j=2$ is similar). We now write 

\begin{eqnarray*}
&&\sum_{\ast} \int_\R\int_\T N^2 \,  \sigma_1^{-\frac{1}{2}+} \,  w_{N_1} \sigma_1^{\frac{1}{2}-}\, 
w_{N_2} \cj{ w_{N_3}}  w_{N_4}   \cj{ w_{N_5}} \cj{ w_{N_6}}\, dxdt\\
&\lesssim& \sum_{\ast} \, N^2 N^{-1} N_1^{-\frac{1}{2}}N_2^{-\frac{1}{2}} N_3^{-\frac{1}{2}}  N_4^{0+} N_5^{0+} N_6^{-\frac{1}{2}} \,  \biggl( \prod_{i=1}^6\,  \Vert w_{N_i} \Vert_{X^{\frac{2}{3}-, \frac{1}{2}-}_3} \biggr) 
\end{eqnarray*} by placing  $ \sigma_1^{\frac{1}{2}-} { w_{N_1}} $ in $L^2_{xt}$, \,  $w_{N_2} \cj{ w_{N_3}}  \cj{ w_{N_6}}$ in $L^2_{xt}$, $w_{N_4} \cj{w_{N_5}} $ in $L^{\infty}_{xt}$ .  Once again,  
by H\"older's inequality and summing as above, we get the desired estimate with decay $N^{-\frac{\delta}{2}}$ so long as  $\delta >0$.

\medskip
$\bullet$ Suppose $j=4$ then proceed as above but place $ \sigma_4^{\frac{1}{2}-}{ w_{N_4}} $ in $L^2_{xt}$,  $w_{N_1} \cj{ w_{N_3}}  \cj{ w_{N_6}}$ in $L^2_{xt}$,  and  $w_{N_2} \cj{w_{N_5}} $ in $L^{\infty}_{xt}$.

\medskip 
 
$\bullet$ Suppose $j=5$ then once again we proceed as above but now place $ \sigma_5^{\frac{1}{2}-}\cj{ w_{N_5}} $ in $L^2_{xt}$,  $w_{N_1} \cj{ w_{N_3}}  \cj{ w_{N_6}}$ in $L^2_{xt}$,  and  $w_{N_2} {w_{N_4}} $ in $L^{\infty}_{xt}$.

\medskip

\begin{remark} \rm Matching Subcases 1 and 2 above means $-\frac{\delta}{2} = \frac{\delta}{6}- \frac{1}{6}$ which requires $\delta=\frac{1}{4}$ and yields a decay of $N^{-\frac{1}{8}+}$. 

\end{remark}

\bigskip

\subsection*{ \bf Case  [IIA] } Part (i) will follow similarly to Case [IA] while part (ii) to Case [IIIA].  

\smallskip

\underbar{Part (i)}  We are under the assumptions $N_3 \sim N \sim N_5$ while $N_1, N_2, N_4, N_6 \ll N$. It follows from  \eqref{AAlgebra3}, there exists $\sigma_j \gtrsim  N^2$. We proceed exactly as in [IA] exchanging in each instance the roles
of $\cj{ w_{N_6}}$ and $\cj{ w_{N_5}}$

\medskip

\underbar{Part (ii)}  We are under the assumptions   $N_3 \sim N \sim N_4$ while $N_1, N_2, N_5, N_6 \ll N$. We have {\it a priori} no help from a large $\sigma_j$ at our disposal.  We proceed then as in [IIIA]  above with the roles of 
$(N_3; \cj{w_{N_3}})$ switched with  that of $(N_6; \cj{w_{N_6}})$ and $(N_1; {w_{N_1}})$  with  $(N_4; {w_{N_4}})$. Hence for $\sigma, \delta >0$ -to be determined- in \underbar{Subcase 1} we are under the assumption $N_1, N_2, N_5 < N^{1-\sigma}$, $N_6 \lesssim N^{\delta}$ and $ N_3 \sim N\sim N_4$. While in  \underbar{Subcase 2} we assume $N_1, N_2, N_5 < N^{1-\sigma}$ while $ N_6 \gtrsim N^{\delta}$ and $N_3\sim N\sim N_4$, and further subdivide just as before into {\it Subcase 2a)}: $N_1, N_2, N_5 \ll  N^{\delta}$ while $N_6 \gtrsim N^{\delta}$ which implies from \eqref{AAlgebra1} the existence of a $\sigma_j \gtrsim N^{1 +\delta}$;  {\it Subcase 2b)}:  there exists $i \in \{1, 2, 5 \}$ such that $N_i \gg  N^{\delta}$ and $N_j \lesssim N^{\delta}$ for $j \neq i$ and $i,  j \in \{ 1,2, 5 \}$ while still  $N_6 \gtrsim  N^{\delta}$ and $N_3 \sim N \sim N_4$ in \eqref{start} and {\it Subcase 2c)}:  that there exist at least $i, j \in \{ 1, 2, 5\}$  \, ($i \neq j$) such that $N_i, N_j \gg  N^{\delta}$ while  $N_6 \gtrsim N^{\delta}$ and $N_3 \sim N \sim N_4$ in \eqref{start}.  Note that $N_1, N_2 < N^{1-\sigma}$ which ensures a support condition on $\cj{w_{N_3}}$.
\underbar{Subcase 3:}  Assume there exists at least one $i \in \{1, 2, 5\}$ such that $N_i \gtrsim N^{1-\sigma}$  $N_2, N_1, N_5 \ll  N$ while $N_6 \ll N$ and $N_3 \sim N\sim N_4$ in \eqref{start}.   
This case follows from Lemma \ref{2nonderivatives} with $0 < \sigma < \frac{1}{6}$ as in its statement.  

Proceeding then just as in [IIIA] we conclude the desired estimate with the same decay in $N$ as in [IIIA] as long as $\frac{2}{5} < \delta < \frac{1}{2}$ and $ 0 < \sigma < \frac{1}{6}$ as before. 

\bigskip

\subsection*{ \bf Case  [IB] }  We first note that parts  (ii), (iii) and (iv) are all symmetric relative to conjugation;  so we only consider (i) and (ii).

\medskip

\underbar{Part (i)}  We are under the assumptions $N_3 \sim N_5 \sim N_6 \sim N$ while $N_1, N_2, N_4 \ll N$. It follows from \eqref{AAlgebra3}, there exists $\sigma_j \gtrsim N^2$.  

\smallskip

$\bullet$ Suppose $j=1, 2$ or $4$. By symmetry is enough to consider $j=1$ and $j=4$.  To obtain decay we need to use the support condition. This we further subdivide into two cases.

\smallskip

{\it Subcase 1:}  Assume  in addition $N_1, N_2 < N^{\delta}$ for some $\delta >0$.  We then have the support condition on $\cj{w_{N_3}}$. For $j=1$ we have:

\begin{eqnarray*}
&&\sum_{\ast}  \int_\R\int_\T N^2 \,  \sigma_1^{-\frac{1}{2}+} \,  w_{N_1} \sigma_1^{\frac{1}{2}-}\, 
w_{N_2} \cj{ w_{N_3}}  w_{N_4}   \cj{ w_{N_5}} \cj{ w_{N_6}}\, dxdt\\
&\lesssim&  \sum_{\ast} \, N^2 N^{-1} N_1^{-\frac{1}{2}}N_2^{-\frac{1}{2}} N_3^{-\frac{2}{3}} N^{\frac{\delta}{6}}  N_4^{0+} N_5^{0+} N_6^{-\frac{1}{2}} \,  \biggl( \prod_{i=1}^6\,  \Vert w_{N_i} \Vert_{X^{\frac{2}{3}-, \frac{1}{2}-}_3} \biggr) 
\end{eqnarray*} by placing  $ \sigma_1^{\frac{1}{2}-} { w_{N_1}} $ in $L^2_{xt}$, \,  $w_{N_2} \cj{ w_{N_3}}  \cj{ w_{N_6}}$ in $L^2_{xt}$, $w_{N_4} \cj{w_{N_5}} $ in $L^{\infty}_{xt}$.  By H\"older's inequality, summing as above, we get the desired 
estimate with decay $N^{-\frac{1}{6} + \frac{\delta}{6}}$ so long as  $0< \delta <1$. 

\smallskip

For $j=4$,  we place  $ \sigma_4^{\frac{1}{2}} { w_{N_4}} $ in $L^2_{xt}$, \, $w_{N_1} \cj{ w_{N_3}}  \cj{ w_{N_6}}$ in $L^2_{xt}$ and $w_{N_2} \cj{w_{N_5}} $ in $L^{\infty}_{xt}$ and proceed similarly.

\smallskip

{\it Subcase 2:}  Assume either $N_1$ or  $N_2 > N^{\delta}$. By symmetry suppose $N_1 > N^{\delta}$; otherwise exchange the roles of $w_{N_1}$ and $w_{N_2}$ below.  We use then the lower bound on $N_1$ as follows. For $j=1$:

\begin{eqnarray*}
&&\sum_{\ast} \int_\R\int_\T N^2 \,  \sigma_1^{-\frac{1}{2}+} \,  w_{N_1} \sigma_1^{\frac{1}{2}-}\, 
w_{N_2} \cj{ w_{N_3}}  w_{N_4}   \cj{ w_{N_5}} \cj{ w_{N_6}}\, dxdt\\
&\lesssim& \sum_{\ast} \, N^2 N^{-1} N_1^{-\frac{1}{2}}N_2^{-\frac{1}{2}} N_3^{-\frac{1}{2}}  N_4^{0+} N_5^{0+} N_6^{-\frac{1}{2}} \,  \biggl( \prod_{i=1}^6\,  \Vert w_{N_i} \Vert_{X^{\frac{2}{3}-, \frac{1}{2}-}_3} \biggr) 
\end{eqnarray*} by placing  $ \sigma_1^{\frac{1}{2}-} { w_{N_1}} $ in $L^2_{xt}$, \,  $w_{N_2} \cj{ w_{N_3}}  \cj{ w_{N_6}}$ in $L^2_{xt}$, $w_{N_4} \cj{w_{N_5}} $ in $L^{\infty}_{xt}$.  Hence, by H\"older's inequality and summing as usual we get the desired 
estimate with decay $N^{-\frac{\delta}{2}}$ so long as  $\delta >0$. 

\smallskip

For $j=4$,  we place  $ \sigma_4^{\frac{1}{2}-} { w_{N_4}} $ in $L^2_{xt}$, \, $w_{N_1} \cj{ w_{N_3}}  \cj{ w_{N_6}}$ in $L^2_{xt}$ and $w_{N_2} \cj{w_{N_5}} $ in $L^{\infty}_{xt}$ and proceed similarly.

\smallskip

\begin{remark}\rm Note that once again, matching Subcases 1 and 2 above means $-\frac{\delta}{2} = \frac{\delta}{6}- \frac{1}{6}$ which requires $\delta=\frac{1}{4}$ and yields a decay of $N^{-\frac{1}{8}+}$.
 \end{remark}
 \medskip
 
 $\bullet$ Suppose $j=3, 6$ or $5$. By symmetry relative to conjugation it's enough to consider -say- $j=3$. We have 
 
 \begin{eqnarray*}
&&\sum_{\ast}  \int_\R\int_\T N^2 \,  \sigma_3^{-\frac{1}{2}+} \,  w_{N_1} w_{N_2} \,  \sigma_3^{\frac{1}{2}-} \cj{ w_{N_3}}  w_{N_4}   \cj{ w_{N_5}} \cj{ w_{N_6}}\, dxdt\\
&\lesssim& \sum_{\ast} \, N^2 N^{-1} N_1^{0+} N_2^{0+} N_3^{-\frac{1}{2}}  N_4^{-\frac{1}{2}} N_5^{-\frac{1}{2}} N_6^{-\frac{1}{2}} \,  \biggl( \prod_{i=1}^6\,  \Vert w_{N_i} \Vert_{X^{\frac{2}{3}-, \frac{1}{2}-}_3} \biggr) 
\end{eqnarray*} by placing  $ \sigma_3^{\frac{1}{2}-} \cj{w_{N_3}} $ in $L^2_{xt}$, \,  $w_{N_4} \cj{ w_{N_5}}  \cj{ w_{N_6}}$ in $L^2_{xt}$, $w_{N_1} {w_{N_2}} $ in $L^{\infty}_{xt}$.  Hence, by H\"older's inequality, summing as usual we get the desired estimate with decay $N^{-\frac{1}{2}+}$.

 \bigskip

 \underbar{Part (ii)}  We are under the assumptions $N_3 \sim N_4 \sim N_6 \sim N$ while $N_1, N_2, N_5\ll N$. It follows from \eqref{AAlgebra1}, there exists $\sigma_j \gtrsim N^2$.
 
 \medskip
 $\bullet$ Suppose $j=1, 2$ or $5$. Suppose $j=1$ then 
  \begin{eqnarray*}
&&\sum_{\ast}  \int_\R\int_\T N^2 \,  \sigma_1^{-\frac{1}{2}+} \,   \sigma_1^{\frac{1}{2}-} w_{N_1} w_{N_2} \,  \cj{ w_{N_3}}  w_{N_4}   \cj{ w_{N_5}} \cj{ w_{N_6}}\, dxdt\\
&\lesssim&  \sum_{\ast} \, N^2 N^{-1} N_1^{-\frac{1}{2}} N_2^{0+} N_3^{-\frac{1}{2}}  N_4^{-\frac{1}{2}} N_5^{0+} N_6^{-\frac{1}{2}} \,  \biggl( \prod_{i=1}^6\,  \Vert w_{N_i} \Vert_{X^{\frac{2}{3}-, \frac{1}{2}-}_3} \biggr) 
\end{eqnarray*} by placing  $ \sigma_1^{\frac{1}{2}-} {w_{N_1}} $ in $L^2_{xt}$, \,  $w_{N_4} \cj{ w_{N_3}}  \cj{ w_{N_6}}$ in $L^2_{xt}$, $w_{N_2} \cj{w_{N_5}} $ in $L^{\infty}_{xt}$.  Hence, by H\"older's inequality and summing as usual we get the desired estimate with decay $N^{-\frac{1}{2}+}$. 
 
If $j=2, 5$ we proceed similarly; keeping  $w_{N_4} \cj{ w_{N_3}}  \cj{ w_{N_6}}$ in $L^2_{xt}$ and exchanging the roles of either  $w_{N_2}$ or  $\cj{w_{N_5}}$ with that of $w_{N_1}$ above. 
 
\medskip
 
$\bullet$ Suppose $j=3, 6$ or $4$. Suppose $j=3$ then 
 
\begin{eqnarray*}
&&\sum_{\ast} \int_\R\int_\T N^2 \,  \sigma_3^{-\frac{1}{2}+} \,  w_{N_1} w_{N_2} \,  \sigma_3^{\frac{1}{2}-} \cj{ w_{N_3}}  w_{N_4}   \cj{ w_{N_5}} \cj{ w_{N_6}} \, dxdt\\
&\lesssim& \sum_{\ast} \, N^2 N^{-1} N_1^{0+} N_2^{-\frac{1}{2} } N_3^{-\frac{1}{2}}  N_4^{-\frac{1}{2}} N_5^{0+} N_6^{-\frac{1}{2}} \,  \biggl( \prod_{i=1}^6\,  \Vert w_{N_i} \Vert_{X^{\frac{2}{3}-, \frac{1}{2}-}_3} \biggr) 
\end{eqnarray*} by placing  $ \sigma_3^{\frac{1}{2}-} \cj{w_{N_3}} $ in $L^2_{xt}$, \,  $ w_{N_2}  w_{N_4} \cj{ w_{N_6}}$ in $L^2_{xt}$, $w_{N_1} \cj{w_{N_5}} $ in $L^{\infty}_{xt}$.  Hence, by H\"older's and summing as usual we get the desired estimate with decay $N^{-\frac{1}{2}+}$. 

If $j=6$ we proceed similarly exchanging the roles of  $\cj{w_{N_3}}$ and  $\cj{w_{N_6}}$ above.  

If $j=4$ we  place  $ \sigma_4^{\frac{1}{2}-} {w_{N_4}} $ in $L^2_{xt}$ and group $w_{N_2}  \cj{w_{N_3}} \cj{ w_{N_6}}$ in $L^2_{xt}$ to derive the same conclusion.
 
\bigskip

We now remove the assumption we made at the beginning of the proof. Suppose  that there is at least a $\sigma_j>N^7$. It follows from \eqref{AAlgebra1} and  \eqref{AAlgebra2} 
that there are two indices $1\leq i_1\ne i_2\leq 6 $ such that 
$\s_{i_1}, \s_{i_2} \gtrsim N^{7}$. 
Then, by \eqref{iv} and \eqref{v}, we have
\begin{eqnarray}\label{Extra-Terms-I-new} 
|I_1|&&\lesssim \sum_{N\leq |n| \leq 3N} \, \sum_{N_i \leq N; \, i=1, \dots 6} \, \int_{\tau}  \biggl( \int_{\tau= \tau_1+\tau_2+\tau_3} \, \sum_{n= n_1+n_2+n_3} \, |\ft{w_{N_1}}| |\ft{w_{N_2}}| \, |n_3| \,  |\ft{\cj{w_{N_3}}}| \, \, d \tau_1 d\tau_2 \biggr) \, \times\notag\\
&& \qquad \qquad\qquad \biggl( \int_{-\tau= \tau_4+\tau_5+\tau_6} \, \sum_{-n= n_4+n_5+n_6} \, |\ft{w_{N_4}}|  \,|\ft{\cj{w_{N_5}}}| \, |n_6|\,  |\ft{\cj{w_{N_6}}}|  \, \, d\tau_4 d\tau_5 \biggr)   \, d\tau.  \\
&&\lesssim \sum_{N\leq |n| \leq 3N} \, \sum_{N_i \leq N; \, i=1, \dots 6} \,N^2
 \|w_{N_1} w_{N_2 } \cj{w_{N_3}}\|_{L^2_{xt}}
\|w_{N_4} \cj{w_{N_5} w_{N_6}}\|_{L^2_{xt}}\notag\\
&& \lesssim \sum_{N\leq |n| \leq 3N} \, \sum_{N_i \leq N; \, i=1, \dots 6} \,N^{-\frac{1}{3}+} \|w_{N_{i_1}}\|_{X^{\frac{1}{2}-, \frac{1}{2}-}}
  \|w_{N_{i_2}}\|_{X^{\frac{1}{2}-, \frac{1}{2}-}}
 \prod_{j \ne i_1,i_2} \|w_j\|_{X^{\frac{1}{2}-, \frac{1}{3}+}}\notag \\
&& \lesssim  N^{-\frac{1}{3}+} \prod_{j = 1}^6 \|w_j\|_{X^{\frac{2}{3}-, \frac{1}{2}-}_3}.\notag
\end{eqnarray}

To treat the  remaining terms in \eqref{EE3} we first note that these are either higher order with no derivatives or same order as the first but with only one derivative term. 
We again start by assuming that $\sigma_j\lesssim N^9$ for all $j$. Under this assumption their estimate follow from the following lemma.

\begin{lemma}[Remaining Terms]\label{rest}
There exists $\beta >0$ such that following estimates hold:
\begin{eqnarray}\label{Extra-Terms-I} 
&&\sum_{N\leq |n| \leq 3N} \, \sum_{N_i \leq N; \, i=1, \dots 6} \, \int_{\tau}  \biggl( \int_{\tau= \tau_1+\tau_2+\tau_3} \, \sum_{n= n_1+n_2+n_3} \, |\ft{w_{N_1}}| |\ft{w_{N_2}}|  \,  |\ft{\cj{w_{N_3}}}| \biggr) \quad \times  \\
 &&\biggl(\int_{-\tau= \tau_4+\tau_5+\tau_6} \, \sum_{-n= n_4+n_5+n_6} \, |\ft{w_{N_4}}|  \,|\ft{\cj{w_{N_5}}}| \, |m(n_6)|\,  |\ft{\cj{w_{N_6}}}| \biggr)   \, d\tau \lesssim  N^{-\beta}  \prod_{i=1}^6\,  \Vert w_{i} \Vert_{X^{\frac{2}{3}-, \frac{1}{2}-}_3}  \notag
\end{eqnarray}

\begin{eqnarray}\label{Extra-Terms-II} 
&&\sum_{N\leq |n| \leq 3N} \, \sum_{N_i \leq N; \, i=1, \dots 8} \, \int_{\tau}  \biggl( \int_{\tau= \sum_{i=1}^5 \tau_i} \, \sum_{n= \sum_{i=1}^5 n_i} \, |\ft{w_{N_1}}| |\ft{w_{N_2}}| \,  |\ft{\cj{w_{N_3}}}|  \ft{{w_{N_4}}}| \ft{\cj{w_{N_5}}}|\biggr) \quad \times  \\
 &&\biggl(\int_{-\tau= \tau_6+\tau_7+\tau_8} \, \sum_{-n= n_6+n_7+n_8} \, |\ft{w_{N_6}}|  \,|\ft{\cj{w_{N_7}}}| \, |m(n_8)|\,  |\ft{\cj{w_{N_8}}}| \biggr)   \, d\tau  \lesssim  N^{-\beta} \prod_{i=1}^8\,  \Vert w_{i} \Vert_{X^{\frac{2}{3}-, \frac{1}{2}-}_3}  \notag 
\end{eqnarray} where the multiplier $m$ satisfies: $|m(\xi)| \leq \langle{\xi}\rangle$. 
 \end{lemma}

\begin{proof}
Here we will only prove \eqref{Extra-Terms-II} since  \eqref{Extra-Terms-I} is similar but simpler. Without loss of generality we can assume that $N_1\sim N\sim N_8$.  Fix any $0<\sigma<1$ and consider the following cases. 
\smallskip

\underbar{Case 1}: Assume that $N_i\lesssim N^\sigma, \, i\ne 1, 8$. Then  we have the support condition on  $w_{N_1}$ and $\cj{w_{N_8}}$. By Plancherel \eqref{Extra-Terms-II}  is less than or equal to
\begin{eqnarray}&&\sum_{N_1,N_8 \sim N;  N_i \leq  N^{\sigma}, i\ne 1,8 }  \int_{\mathbb{R}} \, \int_{\mathbb{T}} \,\,   \, N_8 \, \, \, w_{N_1} \, w_{N_2}  \, \cj{w_{N_3}} \, w_{N_4}\,   \cj{ w_{N_5}}\,  w_{N_6}\cj{w_{N_7}}\cj{w_{N_8}}\, \, \, dx \, dt  \label{8-terms} \\
&\lesssim&\sum_{N_1,N_8 \sim N;  N_i \leq  N^{\sigma}, i\ne 1,8 }N\|w_{N_1} \, w_{N_2}  \, \cj{w_{N_3}}\|_{L^2_{x,t}}\| w_{N_4}\, \cj{ w_{N_5}}\|_{L^\infty_{x,t}}\,\|  w_{N_6}\cj{w_{N_7}}\cj{w_{N_8}}\|_{L^2_{x,t}} \notag\\
&\lesssim&\sum_{N_1,N_8 \sim N;  N_i \leq  N^{\sigma}, i\ne 1,8 }NN_1^{-\frac{2}{3}}\max(N_2,N_3,N_4,N_5)^{\frac{1}{6}}N_4^{0+}N_5^{0+}N_6^{-\frac{1}{2}}N_7^{-\frac{1}{2}}
N_8^{-\frac{2}{3}} \notag\\
&\times&\max(N_6,N_7)^{\frac{1}{6}}\biggl( \prod_{i=1}^8\,  \Vert w_{N_i} \Vert_{X^{\frac{2}{3}-, \frac{1}{2}-}_3} \biggr) 
\lesssim N^{-\frac{1}{3}+\frac{\sigma}{6}+}\biggl( \prod_{i=1}^8\,  \Vert w_{i} \Vert_{X^{\frac{2}{3}-, \frac{1}{2}-}_3} \biggr). \notag
\end{eqnarray} 

\smallskip

\underbar{Case 2}:  Assume there exists $k \ne 1, 8$ such that $N_k > N^{\sigma}$. Without loss of generality -say- $k=4$. Then we bound \eqref{8-terms} as follows: 
\begin{eqnarray*}&&\sum_{N_1,N_8 \sim N;  N_4 > N^{\sigma}, N_i \leq N; i \neq 1,4,8 }  N  \|w_{N_1} \, w_{N_4}  \, \cj{w_{N_3}}\|_{L^2_{x,t}}\| w_{N_2}\, \cj{ w_{N_5}}\|_{L^\infty_{x,t}}\,\|  w_{N_6}\cj{w_{N_7}}\cj{w_{N_8}}\|_{L^2_{x,t}} \\
&\lesssim&\sum_{N_1,N_8 \sim N;  N_4 > N^{\sigma}, N_i \leq N; i \neq 1,4,8 }  N N_1^{-\frac{1}{2}}  N_3^{-\frac{1}{2}}   N_4^{-\frac{1}{2}}  N_2^{0+}N_5^{0+} N_6^{-\frac{1}{2}}N_7^{-\frac{1}{2}}
N_8^{-\frac{1}{2}} \, \biggl( \prod_{i=1}^8\,  \Vert w_{N_i} \Vert_{X^{\frac{2}{3}-, \frac{1}{2}-}_3} \biggr)  \\
&\lesssim&  N^{-\frac{\sigma}{2}+}\biggl( \prod_{i=1}^8\,  \Vert w_{i} \Vert_{X^{\frac{2}{3}-, \frac{1}{2}-}_3} \biggr).  
\end{eqnarray*} 
\end{proof} 

We now remove the assumption we made before the lemma above. Suppose  that there is at least a $\sigma_j>N^9$. The term with six factors follows just in \eqref{Extra-Terms-I-new}. To estimate the term with eight factors we first observe that as before there are at least two  indices $1\leq i_1\ne i_2\leq 8$  such that 
$\s_i, \s_j \gtrsim N^{9}$. Next we use H\" older inequality to bound the left hand side of \eqref{Extra-Terms-II} by 
\begin{equation}\label{Extra-Terms-II-new} 
\sum_{N\leq |n| \leq 3N} \, \sum_{N_i \leq N; \, i=1, \dots 8} N^{} \prod_{i=1}^8\,  \Vert w_{N_{i}} \Vert_{L^8_{tx}}  \lesssim \sum_{N\leq |n| \leq 3N} \, \sum_{N_i \leq N; \, i=1, \dots 8} N^{} \prod_{i=1}^8\,  
 \Vert w_{N_i} \Vert_{X^{\frac{13}{24}+, \frac{3}{8}+}_3}
\end{equation} 
by \eqref{vi}. Using $\sigma_{i_1}, \sigma_{i_2}>N^9$ we conclude that 
\begin{eqnarray*}\eqref{Extra-Terms-II-new}& \lesssim &\sum_{N\leq |n| \leq 3N} \, \sum_{N_i \leq N; \, i=1, \dots 8} N^{-\frac{1}{8}+}  \Vert w_{N{i_1}} \Vert_{X^{\frac{13}{24}+, \frac{1}{2}-}_3} \Vert w_{N_{i_2}} \Vert_{X^{\frac{13}{24}+, \frac{1}{2}-}_3} 
\prod_{i\ne i_1,i_2}\,  \Vert w_{N_{i} }\Vert_{X^{\frac{13}{24}+, \frac{3}{8}+}_3}\\
& \lesssim &N^{-\frac{1}{8}+}  \prod_{i=1}^8\,   \Vert w_{i} \Vert_{X^{\frac{2}{3}-, \frac{1}{2}-}_3}
\end{eqnarray*}

\section{Construction of Weighted Wiener Measures}

In this section we construct weighted Wiener measures and associated probability spaces 
on which we establish well-posedness.   To construct these measures we make use of the
conserved quantities $\EE(v)$  (given in \eqref{eu})  and the $L^2$-norm.  As a motivation we
recall a well known fact in finite dimensional spaces.  Suppose we have a well-posed ODE 
$y_t = F(y)$,  where $F$ is a  divergence-free vector field. Assume $G(y)$ is a constant 
of motion such that for reasonable $f$,  $f(G(y)) \in L^1(dy)$. Then by Liouville's Theorem, 
$d\mu(y) = Z^{-1} f(G(y)) dy$ is, for suitable normalization constant $Z$,  an invariant 
probability measure for flow map for the ODE.

To construct the measures on infinite dimensional spaces we will consider conserved 
quantities of the  form $\exp(-\frac{\beta}{2} \EE(v))$.   But there is a priori little hope of constructing a finite measure using this quantity since (a) the nonlinear part of $\EE(v)$ is not bounded below and (b) the linear part is only 
non-negative but not  positive definite.  To resolve this we use the conservation of $L^2$-norm and 
consider  instead the conserved quantity
\begin{equation}
\chi_{\{\|v\|_{L^2} \le B\}}  e^{-\frac{\beta}{2}\mathcal{N}(v)}
e^{-\frac{\beta}{2} \int (|v|^2 + |v_x|^2) dx} 
\end{equation}
where $\mathcal{N}(v)$ is the nonlinear part of the energy, i.e. 
\begin{eqnarray}\label{nonlinear energy}
 \mathcal{N}(v) &= &- \frac{1}{2} 
{\rm Im} \int_{\T} v^2 \cj{v} \cj{v}_x \, dx - \frac{1}{4 \pi}  
\biggl(\int_{\T} |v|^2 \, dx \biggr)\biggl(\int_{\T} |v|^4 \, dx\biggr) \, + \\&+&\frac{1}{ \pi} 
\biggl(\int_{\T} |v|^2 \, dx \biggr) \biggl({\rm Im} \int_{\T} v  \cj{v}_x \, dx\biggr)\, +\, \frac{1}{4\pi^2}\biggl(\int_{\T} |v|^2 \, dx \biggr)^3.\notag
\end{eqnarray}
and $B$ is a (suitably small) constant.

By analogy with the finite dimensional case we would like to construct the 
measure (with $v(x)= u(x)+ i w(x)$) 
\begin{align} \label{HGibbs2.0}
 ``\  d \mu_\beta & = Z^{-1} \chi_{\{\|v\|_{L^2} \le B\}}  e^{-\frac{\beta}{2}\mathcal{N}(v)}
e^{-\frac{\beta}{2} \int (|v|^2 + |v_x|^2) dx} \prod_{x \in\T}d u(x) d w(x)   \ ".
\end{align}
This is a purely formal, although suggestive, expression  since it is impossible to define the 
Lebesgue measure on  an infinite-dimensional space as a countably additive measure.  Moreover,  it will turn out  that $\int |u_x|^2=\infty$,  $\mu$ almost surely.

One uses instead a Gaussian measure as reference measure and the measure $\mu$ is 
constructed in two steps.  First one constructs a Gaussian measure 
$\rho$ as  the limit of the finite-dimensional measures on $\mathbb{R}^{4N+2}$ given by 
\begin{align}\label{finitedim01} 
 d \rho_N & = Z_{0, N}^{-1} \exp \Big( -\frac{\beta}{2} \sum_{|n| \leq N} (1+|n|^2)|\ft{v}_n|^2 \Big) 
 \prod_{|n| \leq N} d a_n d b_n
\end{align}
where $\ft{v}_n = a_n + i b_n$.   The construction of such Gaussian measures is a classical subject,
see e.g. Gross \cite{GROSS} and Kuo \cite{KUO}.  For our purpose we will need to realize this measure 
as a measure supported on a suitable Banach space.  Once this measure $\rho$ has been constructed 
one constructs the measure $\mu$ as a  measure which is absolutely continuous with respect to 
$\rho$ and whose Radon-Nikodym derivative is
$$
\frac{d \mu}{d \rho}  \,=\,  \tilde{Z}^{-1} \chi_{\{\|v\|_L^2 \le B\}}  e^{ - \frac{\beta}{2} \mathcal{N}(v)}. 
$$

For this measure to be normalizable it turns out that one needs $B$ to be sufficiently 
small. Also the constant $\beta$ in the measure does not play any role in the analysis (
although the  cutoff $B$ depends on $\beta$) and thus in the  sequel we will set $\beta=1$.  
But note that the measures for different $\beta$ are all invariant and they are all mutually singular \cite{GROSS,KUO}.

First let us recall some facts on  Gaussian measures in Hilbert spaces and Banach spaces. 
For details see Zhidkov \cite{Z}, Gross \cite{GROSS} and Kuo \cite{KUO}.
Let $n \in \mathbb{N}$ and $\mathcal T$ be a symmetric positive $n \times n$ matrix
with real entries.
The  Borel measure $\rho$ in $\mathbb{R}^n$ given by 
\[ d \rho(x) = \frac{1}{\sqrt{(2\pi)^n \det (\mathcal T )}} \exp \big( -\tfrac{1}{2} \langle {\mathcal T}^{-1} x, x \rangle_{\mathbb{R}^n} \big)\, dx\] 
is called a (nondegenerate centered) Gaussian measure in $\mathbb{R}^n$. Note that $\rho(\mathbb{R}^n) = 1$.

Now, we consider the analogous definition of the infinite dimensional (centered) Gaussian measures.
Let $H$ be a real separable Hilbert space and $\mathcal T: H \to H$ be a linear positive self-adjoint operator 
(generally unbounded) with eigenvalues $\{\ld_n\}_{n\in \mathbb{N}}$
and the corresponding eigenvectors $\{e_n\}_{n\in\mathbb{N}}$  forming an orthonormal basis of $H$.
We call a set $M \subset H$  cylindrical if there exists an integer $n\geq 1$ and a Borel set $F \subset \mathbb{R}^n$
such that
\begin{equation} \label{Hcylinder}
 M = \big\{ x \in H : ( \jb{ x, e_1}_H, \cdots, \jb{ x, e_n}_H ) \in F \big\}. 
\end{equation}

\noindent
Given the operator $\mathcal T$, we denote by $\mathcal{A}$ the set of all cylindrical subsets of $H$  and one 
can easily verify that  $\mathcal{A}$ is a field.  The centered Gaussian measure in $H$ with correlation 
operator $\mathcal T$  is  defined as  the additive (but not countably additive in general) measure $\rho$ defined on 
the field $\mathcal{A}$ via
\begin{equation} \label{HGaussian}
 \rho(M) = (2\pi)^{-\frac{n}{2}} \prod_{j = 1}^n \ld_j^{-\frac{1}{2}} \int_F e^{-\frac{1}{2}\sum_{j = 1}^n \ld_j^{-1} x_j^2 }d x_1 \cdots dx_n,
\text{ for }M \in \mathcal{A} \text{ as in \eqref{Hcylinder}. }
\end{equation}

The following proposition tells us when  this Gaussian measure $\rho$ is countably additive.

\begin{proposition} \label{PROP:Hadd}
The  Gaussian measure $\rho$ defined in \eqref{HGaussian} is countably additive
on the field $\mathcal{A}$ if and only if $\mathcal T$ is an operator of trace class, 
i.e., $\sum_{n = 1}^\infty \ld_n < \infty$.
If the latter holds, then
the minimal $\s$-field $\mathcal{M}$ containing the field $\mathcal{A}$ of all cylindrical sets is the Borel $\s$-field on $H$.
\end{proposition}

Consider a sequence of the finite dimensional Gaussian measures $\{\rho_n\}_{n\in\mathbb{N}}$ as follows.
For fixed $n \in \mathbb{N}$, let $\mathcal{M}_n$ be the set of all cylindrical sets in $H$ of the form 
\eqref{Hcylinder} with this fixed $n$
and arbitrary Borel sets $F\subset \mathbb{R}^n$.
Clearly, $\mathcal{M}_n$ is a $\s$-field, and setting 
\[ \rho_n(M) = (2\pi)^{-\frac{n}{2}} \prod_{j = 1}^n \ld_j^{-\frac{1}{2}} \int_F e^{-\frac{1}{2}\sum_{j = 1}^n \ld_j^{-1} x_j^2 }d x_1 \cdots dx_n\]

\noindent
for $M \in \mathcal{M}_n$, we obtain a countably additive measure $\rho_n$ defined on $\mathcal{M}_n$.
Then, one can extend the measure  $\rho_n$  onto the whole Borel $\s$-field $\mathcal{M}$ of $H$
by setting $ \rho_n(A) := \rho_n(A \cap \text{span}\{e_1, \cdots, e_n\})$
for $A \in \mathcal{M}$.\footnote{Note a slight abuse of notation.
We use $\rho_n$ to denote a Gaussian measure on $\text{span}\{e_1, \cdots, e_n\}$ 
as well as its extension on $H$.
A similar comment applies in the following.}
Then, we have

\begin{proposition} \label{PROP:Zhidkov2}
Let $\rho$ in  \eqref{HGaussian} be countably additive.
Then,  $\{\rho_n\}_{n\in \mathbb{N}}$ constructed above converges weakly to $\rho$ as $n \to \infty$.
\end{proposition}

For our problem then we consider the Gaussian measure $\rho$ which is the weak limit of the finite dimensional Gaussian measures
\begin{equation}\label{finite N rho}
 d\rho_N  = Z_{0, N}^{-1} 
 \exp \Big( -\frac{1}{2} \sum_{ |n| \leq N}(1+ |n|^2) |\ft{v}_n|^2 \Big) \prod_{|n| \leq N} d a_n d b_n \,.
\end{equation}
Let $J_s :=  (1- \Delta)^{s-1}$, then we have 
$$
\sum_{n} (1 + |n|^2)  \left| \ft{v}_n \right|^2 \,=\, \langle v , v \rangle_{H^1} \,=\,  \langle  J_s^{-1} v , v \rangle_{H_s} \,.
$$
The operator $J_s : H_s \to H_s$  has the set of eigenvalues $\{(1+|n|^2)^{(s-1)}\}_{n \in \mathbb{Z}}$ and 
the corresponding  eigenvectors   $\{ (1+|n|^2)^{-s/2} e^{inx}\}_{n \in \mathbb{Z}}$ form an orthonormal 
basis of $H^s$. Since $J_s$ is of trace class if and only if  $s < \frac{1}{2}$, by Proposition \ref{PROP:Hadd}, 
$\rho$  is  a countably additive measure on $H^s$  for any $s < 1/2$ (but not for $s\ge 1/2$.)

Unfortunately, \eqref{GDNLS} is locally well-posed in $H^s(\T)$ only for $ s\geq \frac{1}{2}$ \cite{Herr}.
Instead, we propose to work in the Fourier-Lebesgue space $\mathcal{F} L^{s, r} (\mathbb{T})$ 
defined in \eqref{fourierlebesgue}  in view of the local well-posedness result by 
Gr\"unrock-Herr \cite{GH}.   Since $\mathcal{F} L^{s, r}$ is not a Hilbert space, we need to 
construct $\rho$ as a measure supported on a Banach space.

\subsection{ General Banach space setting}
 
Let us recall the basic theory of abstract Wiener spaces \cite{KUO}.  Given  a real separable Hilbert 
space $H$ with norm $\|\cdot \|$,  let $\mathcal{F} $ denote the set of finite dimensional orthogonal projections $\mathbb{P}$ of $H$.
Then, define a cylinder set $E$ by  $E = \{ x \in H: \mathbb{P}x \in F\}$ where $\mathbb{P} \in \mathcal{F}$ 
and $F$ is a Borel subset of $\mathbb{P}H$,
and let $\mathcal{R} $ denote the collection of such cylinder sets.
Note that $\mathcal{R}$ is a field but not a $\s$-field.
The Gaussian measure $\rho$ on $H$ is defined 
by 
\[ \rho(E) = (2\pi)^{-\frac{n}{2}} \int_F e^{-\frac{\|x\|^2}{2}} dx  \]

\noindent
for $E \in \mathcal{R}$, where
$n = \text{dim} \mathbb{P} H$ and  
$dx$ is the Lebesgue measure on $\mathbb{P}H$.
It is known that $\rho$ is finitely additive but not countably additive in $\mathcal{R}$.

\begin{definition}[Gross \cite{GROSS}] \label{H measurable} A seminorm $|||\cdot|||$ in $H$ is called measurable if for every $\eps>0$, 
there exists $\mathbb{P}_\eps \in \mathcal{F}$ such that 
\begin{equation*}
 \rho( ||| \mathbb{P} x ||| > \eps  )< \eps 
\end{equation*}
for $\mathbb{P} \in \mathcal{F}$ orthogonal to $\mathbb{P}_\eps$.
\end{definition}
Any measurable seminorm  is weaker  than the norm of $H$,
and $H$ is not complete with respect to $|||\cdot|||$ unless $H$ is finite dimensional.
Let $\mathcal B$ be the completion of $H$ with respect to $|||\cdot|||$
and denote by $i$ the inclusion map of $H$ into $\mathcal B$.
The triple $(i, H, \mathcal B)$ is called an abstract Wiener space.

Now, regarding $y \in {\mathcal B}^\ast$ as an element of $H^\ast \equiv H$ by restriction,
we embed ${\mathcal B}^\ast $ in $H$.
Define the extension of $\rho$ onto $\mathcal B$ (which we still denote by $\rho$)
as follows.
For a Borel set $F \subset \R^n$, set
\[ \rho ( \{x \in \mathcal B: ((x, y_1), \cdots, (x, y_n) )\in F\})
:= \rho( \{x \in H: (\jb{x, y_1}_H, \cdots, \jb{x, y_n}_H )\in F\}),\]

\noindent
where $y_j$'s are in ${\mathcal B}^\ast$ and $(\cdot , \cdot )$ denote the natural pairing between $\mathcal B$ and ${\mathcal B}^\ast$.
Let $\mathcal{R}_{\mathcal B}$ denote the collection of cylinder sets
$ \{x \in \mathcal B: ((x, y_1), \cdots, (x, y_n) )\in F \}$
in $\mathcal B$.

\begin{proposition}[Gross \cite{GROSS}]
$\rho$ is countably additive in the $\s$-field generated by $\mathcal{R}_{\mathcal B}$.
\end{proposition}

\subsection{ Back to our setting}

In the present context, we will let $H= H^1(\mathbb{T})$ and
$\mathcal B=\mathcal{F} L^{s, r}(\T)$ with $2 \leq r <\infty$  and $(s-1)r < -1$.  First we prove the following result.

\begin{proposition} \label{PROP:meas} Let $2 \leq r <\infty$ and assume $(s-1)r < -1$. Then  the seminorm  
$\|\cdot\|_{\mathcal{F} L^{s, r}}$ is measurable. Moreover, we have the following 
exponential tail estimate: there exists $C>0$ and $ c > 0$ (which both depends  on $(s,r)$)
such that,  for  $K>0$,
\begin{equation} \label{Hld}
\rho ( \|v\|_{\mathcal{F} L^{s, r}} > K) \leq C e^{-cK^2}.
\end{equation}
\end{proposition}

\noindent

This shows that $(i, H, \mathcal B) = (i, H^1, \mathcal{F} L^{s, r})$ \, ($2 \leq r <\infty$) 
is an abstract Wiener space if  $(s-1)r < -1$ and thus the 
Wiener measure $\rho$  can be realized as a countably additive 
measure supported  on  $\mathcal{F} L^{s, r}$ for $(s-1)r < -1$.  This is
hardly surprising since  this is equivalent to $\sigma \equiv s + \frac{1}{r} - \frac{1}{2} < \frac{1}{2}$
and  $\mathcal{F} L^{s, r}$ scale as $H^\sigma$.   

The second part of Proposition \ref{PROP:meas} is a consequence of Fernique's 
theorem \cite{FER} (c.f. Theorem 3.1 of Chapter III in \cite{KUO}).

\begin{remark} \rm Proposition \ref{PROP:meas} was essentially proved in \cite{Oh3} in the context of white noise for the KdV equation.
  We include here a proof in our DNLS context for completeness\footnote{ Proposition  \ref{PROP:meas}  also holds for $r<2$ and $(s-1)r < -1$, albeit with a different proof (see \cite{Benyi} for details).  
For our purposes $2 \leq r <\infty$ suffices and so we restrict ourselves to that case.}.  
\end{remark}

It is useful to note that the measure $\rho_N$ given in \eqref{finite N rho} can be regarded 
as  the induced probability measure on  $\mathbb{C}^{2N+1}\cong \mathbb{R}^{4N+2}$  
under the map 
\begin{equation} \label{H induced}
 \omega \longmapsto \bigg\{ \frac{ g_n}{\sqrt{ 1+ |n|^2} }\bigg\}_{|n| \leq N}, 
\end{equation}
where $g_n(\omega)$, $|n| \leq N$,  are independent standard complex 
Gaussian random variables on a probability space  $(\Omega, \mathcal F, P)$ (i.e. $\widehat{v}_n =  \frac{ g_n}{\sqrt{ 1+ |n|^2}}$).  
In a similar manner, we can view $\rho$ as the induced probability measure under the map
$ \omega \mapsto \{  \frac{ g_n}{\sqrt{ 1+ |n|^2} }  \}_{n\in \mathbb{Z}}$, 
where $g_n(\omega)$ are independent standard complex Gaussian random variables.

For the proof of Proposition, \ref{PROP:meas} we first recall the following result.

\begin{lemma}[Lemma 4.7  \cite{Oh4}] \label{CL:decay}
Let $\{g_n\}$ be a sequence of independent standard complex-valued
Gaussian random variables. Then, for $M$ dyadic and $\dl <
\frac{1}{2}$, we have
\[ \lim_{M\to \infty} M^{2\dl} \frac{\max_{|n|\sim M } |g_n|^2}{ \sum_{|n|\sim M} |g_n|^2} = 0\, \text{ a.s.} \]
\end{lemma}

\begin{proof}[Proof of Proposition \ref{PROP:meas}]  

Let $2 \leq r < \infty$ and $(s-1)r < -1$.  In view of  Definition \ref{H measurable},  it
suffices to show that for given $\eps> 0$, there exists a large $M_0$ such that 
\begin{equation} \label{msble2}
 \rho \big(\|P^{\perp}_{M_0} v\|_{\mathcal{F}L^{s, r}} > \eps) < \eps,
\end{equation}
\noi where $P^{\perp}_{M_0}$is the projection onto the frequencies
$|n| > M_0$.  Note that  if $ \mathbb{P} $ is a finite dimensional projection such that $ \mathbb{P} \perp P_{M_0}$ then $\|  \mathbb{P} v \|_{\mathcal{F}L^{s, r}} \leq  \|  P^{\perp}_{M_0} v \|_{\mathcal{F}L^{s, r}}$.

In view of \eqref{H induced},
we assume that $v$ is of the form 
\begin{equation}\label{FW1}
v (x) = \sum_n \frac{ g_n}{\sqrt{ 1+ |n|^2} }e^{inx}, 
\end{equation}
where $\{g_n\}$ is as in \eqref{H induced}.

\noi

Let $\dl < \frac{1}{2}$ to be chosen later.  Then, by Lemma \ref{CL:decay} and Egoroff's theorem, 
there exists a set $E$ such that $\rho (E^c) < \frac{1}{2}\eps$ 
and the convergence in Lemma \ref{CL:decay} is uniform on $E$, 
i.e. we can choose dyadic $M_0$ large enough such that
\begin{equation} \label{A:decay}
\frac{\| \{g_n (\omega) \}_{|n| \sim M} \|_{L^{\infty}_n} } {\|
\{g_n (\omega) \}_{|n| \sim M} \|_{L^{2}_n} }
 \leq M^{-\dl},
\end{equation}

\noindent for  all $\omega \in E$  and dyadic $M > M_0$. In the
following, we will work only on $E$ and  drop `$\cap E$' for
notational simplicity. However, it should be understood that all the
events are under the intersection with $E$ so that  \eqref{A:decay}
holds.

Let $\{\s_j \}_{j \geq 1}$ be a sequence of positive numbers such
that $\sum \s_j = 1$, and let $ M_j = M_0 2^j$ dyadic. Note that
$\s_j = C 2^{-\ld j} =C M_0^\ld M_j^{-\ld}$ for some small $\ld > 0$
(to be determined later.) Then,  from \eqref{FW1}, we have
\begin{align} \label{A:subadd}
\rho \big(\|P^{\perp}_{M_0} v(\omega) \|_{\mathcal{F}L^{s, r}} > \eps)
& \leq \sum_{j = 1}^\infty \rho \big( \| \{\jb{n}^{s-1}
g_n(\omega)  \}_{|n| \sim M_j} \|_{L_n^{r}}  > \s_j \eps \big).
\end{align}

\noindent By interpolation and \eqref{A:decay}, we have
\begin{align*}
\| \{ & \jb{n}^{s-1}  g_n \}_{|n| \sim M_j} \|_{L_n^{r}} \sim
M_j^{s-1} \| \{ g_n \}_{|n| \sim M_j} \|_{L_n^{r}} \leq  M_j^{s-1}
\| \{ g_n \}_{|n| \sim M_j} \|_{L_n^{2}}^\frac{2}{r}
  \| \{ g_n \}_{|n| \sim M_j} \|_{L_n^{\infty}}^\frac{r-2}{r} \\
& \leq  M_j^{s-1} \| \{ g_n \}_{|n| \sim M} \|_{L_n^{2}}
\Bigg(\frac{ \| \{ g_n \}_{|n| \sim M_j} \|_{L_n^{\infty}} } {\| \{
g_n \}_{|n| \sim M_j} \|_{L_n^{2}}} \Bigg)^\frac{r-2}{r} \leq
M_j^{s-1 -\dl \frac{r-2}{r}} \| \{ g_n \}_{|n| \sim M_j}
\|_{L_n^{2}}.
\end{align*}

\noindent 
Thus, if we have $\|\{\jb{n}^{s-1}g_n \}_{|n| \sim M_j}
\|_{L_n^{r}}  > \s_j \eps$,
 then we have
 $\| \{ g_n \}_{|n| \sim M_j} \|_{L_n^{2}}
 \gtrsim R_j $
 where $R_j := \s_j \eps M_j^\al$ with $\al := -s+1+\dl \frac{r-2}{r} $.
With $r = 2 + \theta$, we have $\al =  \frac{-(s-1)r +  \dl
\theta}{2 +  \theta} > \frac{1}{2}$ by taking $\dl$ sufficiently
close to $\frac{1}{2}$ since $-(s-1)r > 1$. Then, by taking $\ld >
0$ sufficiently small, $R_j = \s_j \eps M_j^\al = C \eps M_0^\ld
M_j^{\al-\ld} \gtrsim C \eps M_0^{\ld} M_j^{\frac{1}{2}+} $. By a
direct computation in polar coordinates, we have
\begin{align*}
\rho\big(  \| \{ g_n \}_{|n| \sim M_j} \|_{L_n^{2}}  \gtrsim R_j
\big)
 \sim \int_{B^c(0, R_j)} e^{-\frac{1}{2}|g_n|^2} \prod_{|n| \sim M_j} dg_n
 \lesssim \int_{R_j}^\infty e^{-\frac{1}{2}s^2}  s^{2  \# \{|n| \sim M_j\} -1} ds.
\end{align*}

\noindent Note that, in the inequality, we have dropped the implicit
constant $\s(S^{2 \# \{|n| \sim M_j\} -1})$, a surface measure
of the $2 \# \{|n| \sim M_j\} -1$ dimensional unit sphere,
since $\s(S^n) = 2 \pi^\frac{n}{2} / \Gamma (\frac{n}{2}) \lesssim
1$. By the change of variable $t = M_j^{-\frac{1}{2}}s$, we have
$s^{2 \# \{|n| \sim M_j\} -2} \lesssim s^{4M_j} \sim
M_j^{2M_j} t^{4M_j}.$ Since $t > M_j^{-\frac{1}{2}} R_j = C \eps
M_0^{\ld} M_j^{0+}$, we have $M_j^{2M_j} = e^{2M_j \ln M_j} <
e^{\frac{1}{8}M_jt^2}$ and $ t^{4M_j} < e^{\frac{1}{8}M_jt^2}$ for
$M_0$ sufficiently large. Thus, we have $s^{2 \# \{|n| \sim
M_j\} -2} < e^{\frac{1}{4}M_jt^2} = e^{\frac{1}{4}s^2}$ for $ s >
R_j.$ Hence,  we have
\begin{align} \label{A:highfreq1}
\rho\big(   \| \{ g_n \}_{|n| \sim M_j} \|_{L_n^{2}}  \gtrsim R_j
\big) \leq C \int_{R_j}^\infty e^{-\frac{1}{4}s^2} s ds \leq
e^{-cR_j^2} = e^{-cC^2 M_0^{2\ld} M_j^{1+} \eps^2}.
\end{align}

\noindent From \eqref{A:subadd} and \eqref{A:highfreq1}, we have
\begin{align*}
\rho \big(\|P^{\perp}_{M_0}v\|_{\mathcal{F}L^{s, r}} > \eps) \leq
\sum_{j =1}^\infty e^{-cC^2 M_0^{1+2\ld+} (2^j)^{1+} \eps^2} \leq
\tfrac{1}{2} \eps
\end{align*}

\noi by choosing $M_0$ sufficiently large as long as $(s-1)r<-1$.
Hence, the seminorm $\|\cdot\|_{\mathcal{F} L^{s, r}}$ is measurable for $(s-1)r < -1$.

The tail estimate \eqref{Hld}
is a direct consequence of
Fernique's theorem \cite[Theorem 3.1]{KUO}.
\end{proof}

To construct the weighted Wiener measure $\mu$ let us define
\begin{equation} \label{Hweight}
R(v) : = \chi_{\{ \| v \|_{L^2}\leq B\}}  e^{-\frac{1}{2} \mathcal{N}(v)}\,,  \quad \quad R_N(v)\,:=\, R(v^N)
\end{equation}
where   $\mathcal{N}(v)$ is the nonlinear part of the energy defined in \eqref{nonlinear energy} and at this stage and for the remainder of this section $v^N = P_N(v)$ for some generic function $v$. In the next section $v^N$ will denote the solution to the (FGDNLS) \eqref{FGDNLS} as in Section 3.    We write  
\begin{equation}\label{nonlinear decomp}\mathcal{N}_N(v): = \mathcal{N}(v^N) = F_N(v) + G_N(v)+K_N(v),\end{equation}
where 
\begin{eqnarray} \label{F}
 F_N(v) &=& - \frac{1}{2} 
{\rm Im} \int_{\T} (v^N)^2 \cj{v^N} \cj{v^N_x}  dx,\\
G_N(v) &=&
-\frac{1}{4 \pi}  \biggl(\int_{\T} |v^N|^2 \, dx \biggr)\biggl(\int_{\T} |v^N|^4 \, dx\biggr),\label{G}\\
K_N(v)& = &\frac{1}{ \pi} 
\biggl(\int_{\T} |v^N|^2 \, dx \biggr) \biggl({\rm Im} \int_{\T} v^N  (\cj v_x^N)dx \biggr) \,+ \, \frac{1}{4\pi^2}\biggl(\int_{\T} |v^N|^2 \, dx \biggr)^3.\label{Z}
\end{eqnarray} 
We will construct the measure 
\begin{equation} \label{HGibbs2}
d\mu = Z^{-1} R(v) d\rho \,,
\end{equation} 
for sufficiently small $B$, as the weak limit of  the finite dimensional weighted  Wiener measures 
$\mu_N$ on $\mathbb{R}^{4N+2}$  given by 
\begin{eqnarray}
d\mu_N  \,& =& \,   Z_N^{-1} R_N(v) d \rho_N  \nonumber  \\
&=&{Z}_N^{-1} \chi_{\{ \| {v}^N \|_{L^2} \leq B\}} 
e^{- \frac{1}{2} \mathcal{N}(v^N)} d \rho_N \label{HGibbs3}
\end{eqnarray}
for suitable normalization $Z_N$.

\begin{lemma} \label{LEM:Hbound}

\textup{(a)}  The sequence $F_N$ converges in $L^2(d \rho)$ to  
\[F(v) = - \frac{1}{2} {\rm Im} \int_{\T} v^2 \cj{v} \cj{v}_x  dx.\]

\noi
Moreover, for $\al < \frac{3}{4}$, there exist $C, \dl >0$
such that for all $M \geq N\geq 1$ and $\ld > 0$, we have
\begin{equation} 
\rho( |F_M(v) - F_N(v)| > \ld) \leq C  e^{-\dl N^\al \ld^\frac{1}{2}}
\end{equation}

\noi
\textup{(b)}
Let $p \in [2, \infty)$. 
Then, 
there exist $\al, C$
such that for all $M \geq N\geq 1$ and $\ld > 0$, we have
\begin{align} \label{Hld1}
& \rho( \| P_N v\|_{L^p(\T)} > \ld) <  C e^{-c \ld^2}\\
& \label{Hld2}
\rho( \|P_M v - P_N v\|_{L^p(\T)} > \ld) < C e^{-c N^{2\al} \ld^2}
\end{align}
\end{lemma}

\begin{proof}  Part (a) was proved by Thomann and Tzvetkov in \cite{TTzv} Proposition 3.1. 
using Proposition \ref{chaos} below. Note that their proof only uses the fact that $v$ is in the support of the measure and is independent of the function space $v$ is in.
\smallskip

\noi To prove part (b) we first note that for any $ 2 \leq p < \infty$, and $N \leq M$, 
\begin{eqnarray}\label{lpbound1}
\| P_N v \|_{L^p(\T)} &\leq C \| P_N v \|_{{\FF L}^{\frac{2}{3}- , 3}(\T)} \\
\| P_N v - P_M v \|_{L^p(\T)} &\leq C    \frac{1}{N^{\alpha}}   \| P_M v \|_{{\FF L}^{\frac{2}{3}- , 3}(\T)}, \label{lpbound2}
\end{eqnarray}
where $ \alpha = \dfrac{1}{p}-\, $.  Then use \eqref{lpbound1} and \eqref{lpbound2} in conjunction with \eqref{Hld} to conclude the proof.

\end{proof}

\begin{lemma}\label{hard1} $K_N(v)$ is Cauchy in measure; i.e. for every $\gamma>0$ and $N \leq M$ 
$$ \lim_{N, M \to \infty} \, \rho( | K_M(v) - K_N(v) | >  2 \gamma ) = 0,$$
and hence  $K_N$ converges in measure  to  
\[K(v) =\frac{1}{ \pi} 
\biggl(\int_{\T} |v|^2 \, dx \biggr) \biggl({\rm Im} \int_{\T} v  \cj v_x\biggr)\,+ \, \frac{1}{4\pi^2}\biggl(\int_{\T} |v|^2 \, dx \biggr)^3.\]

\end{lemma}

Before the proof we need the following Proposition \ref{chaos} (see Thomann and Tzevtkov
 \cite{TTzv} for a proof) and Lemma \ref{hard2} which we prove below. 

\begin{proposition}\label{chaos} Let $d \ge 1$ and $c(n_1, \dots, n_k) \in \C$. Let $\{g_n\}_{1\leq n\leq d} \in \mathcal N_{\C}(0,1)$ be complex $L^2$ normalized independent Gaussians.  For $k\ge 1$ denote by $A(k, d):= \{ (n_1, \dots, n_k) \in \{1, \dots, d\}^k, \, n_1 \leq \dots \leq n_k \}$ and 
\begin{equation}\label{form} 
S_k(\omega) = \sum_{A(k, d)}  c(n_1, \dots, n_k) g_{n_1}(\omega) \dots g_{n_k}(\omega). 
\end{equation}
Then for all $d \ge 1$ and $p \ge 2$ 
$$ \Vert S_k \Vert_{L^p(\Omega)} \leq \sqrt{k+1} (p -1)^{\frac{k}{2}}  \Vert S_k \Vert_{L^2(\Omega)}. $$
\end{proposition}

Let $X_N(v) =    \int_{\T} v^N  {\cj v_x^N}. $

\begin{lemma}\label{hard2} For any  $N\leq M$ and $\eps>0$ we have
\begin{eqnarray}
|X_N(v)|&\lesssim& N^{2\eps}\|v^N\|^2_{\FF L^{\frac{2}{3}-\eps,3}}\label{xn}\\
\|X_M(v)-X_N(v)\|_{L^4}&\lesssim&\frac{1}{N^{\frac{1}{2}}}.\label{l4xn} \\
\mbox{Moreover, }\qquad \|X_M(v)-X_N(v)\|_{L^q}&\lesssim& c (q-1) \frac{1}{N^{\frac{1}{2}}} \qquad \mbox{ for any } q \geq 2. \label{lqxn}
\end{eqnarray}
\end{lemma}
\begin{proof}
To prove \eqref{xn} we use Plancherel  and H\"older inequality to  obtain
\begin{eqnarray*}|X_N(v)|&\leq& \sum_{|n|\leq N}|n|\, |\widehat{v^N}(n)|^2\\
&\leq&  \left(\sum_{|n|\leq N}|n|^{-1+6\eps}\right)^{\frac{1}{3}} \left(\sum_{|n|\leq N}(|n|^{\frac{2}{3}-\eps} |\widehat{v^N}(n)|)^3\right)^{\frac{2}{3}}\leq N^{2\eps}\|v^N\|^2_{\FF L^{\frac{2}{3}-\eps,3}}.
\end{eqnarray*}

To prove \eqref{l4xn} we start by recalling that $v^N(\omega,x):=\sum_{|n|\leq N}\frac{g_n(\omega)}{\langle n\rangle}e^{inx}$.  Then by Plancherel
$$
X_N(v)=-i\sum_{|n|\leq N}n\frac{|g_n(\omega)|^2}{\langle n\rangle^2} \quad \mbox{ and } \quad X_M(v)-X_N(v)=-i\sum_{N\leq |n| < M}n\frac{|g_n(\omega)|^2}{\langle n\rangle^2}, 
$$
and 
\begin{equation}\label{gg}
|X_M(v)-X_N(v)|^2=\sum_{N\leq |n_1|, |n_2| < M}n_1n_2\frac{|g_{n_1}(\omega)|^2|g_{n_2}(\omega)|^2}{\langle n_1\rangle^2\langle n_2\rangle^2}=:Y^1_{N,M}+Y^2_{N,M}+Y^3_{N,M},
\end{equation}
where 
\begin{eqnarray}
Y^1_{N,M}&:=&\sum_{N\leq |n_2|, |n_1| < M}n_1n_2\frac{(|g_{n_1}(\omega)|^2-1)(|g_{n_2}(\omega)|^2-1)}{\langle n_1\rangle^2\langle n_2\rangle^2}\label{y1}\\
Y^2_{N,M}&:=&\sum_{N\leq |n_2|, |n_1| < M}n_1n_2\frac{(|g_{n_1}(\omega)|^2-1)+(|g_{n_2}(\omega)|^2-1)}{\langle n_1\rangle^2\langle n_2\rangle^2}\label{y2}\\
Y^3_{N,M}&:=&\sum_{N\leq |n_2|, |n_1| <  M}\frac{n_1n_2}{\langle n_1\rangle^2\langle n_2\rangle^2}\notag.
\end{eqnarray}
By symmetry $Y^3_{N,M}=0$. We now observe that 
\begin{equation}\label{l4xn1}
\|X_M(v)-X_N(v)\|_{L^4}^4\lesssim \|Y^1_ {N,M}\|_{L^2}^2+\|Y^2_{N,M}\|_{L^2}^2.
\end{equation} 
We now proceed as in \cite{TTzv}, denote by $G_n(\omega):=|g_{n}(\omega)|^2-1$ and note that by the independence of $g_n(\omega)$ (c.f. \eqref{H induced}), 
\begin{equation}\label{expectation}
\mathbb E[G_n(\omega)G_m(\omega)]=0 \quad \mbox{ for $n\neq m$}.
\end{equation}
Since 
$$|Y^1_{N,M}|^2=\sum_{N\leq |n_1|, |n_2|, |n_3|, |n_4| < M}n_1n_2 n_3n_4\frac{G_{n_1}G_{n_2}G_{n_3}G_{n_4}}{\langle n_1\rangle^2\langle n_2\rangle^2\langle n_3\rangle^2\langle n_4\rangle^2}.$$
We compute $\mathbb E[|Y^1_{N,M}|^2]$ and by \eqref{expectation} the only contributions come from ($n_1=n_3$ and $n_2=n_4$), ($n_1=n_2$ and $n_3=n_4$) or ($n_2=n_3$ and $n_1=n_4$) . Hence by symmetry and using that the fourth moment of the Gaussian $g_n(\omega)$ are bounded we have 
\begin{equation}\label{y12}
\|Y^1_{N,M}\|_{L^2}^2=E[|Y^1_{N,M}|^2] \leq C\sum_{N\leq |n_1|, |n_2| < M}\frac{n^2_1n_2^2}{\langle n_1\rangle^4\langle n_2\rangle^4}\lesssim  \frac{1}{N^2}.
\end{equation}
On the other hand, since 
$$|Y^2_{N,M}|^2=\sum_{N\leq |n_1|, |n_2|, |n_3|, |n_4| < M}n_1n_2 n_3n_4\frac{(G_{n_1}+G_{n_2})(G_{n_3}+G_{n_4})}{\langle n_1\rangle^2\langle n_2\rangle^2\langle n_3\rangle^2\langle n_4\rangle^2},$$
by symmetry it is enough to consider a single term of the form
$$
\sum_{N\leq |n_1|, |n_2|, |n_3|, |n_4| < M}n_1n_2 n_3n_4\frac{G_{n_j}G_{n_k}}{\langle n_1\rangle^2\langle n_2\rangle^2\langle n_3\rangle^2\langle n_4\rangle^2},$$
with $1\leq j\neq k\leq 4$, which we set without any loss of generality to be $j=1, \, k=3$. We then have 
$$\|Y^2_{N,M}\|_{L^2}^2=E[|Y^2_{N,M}|^2]\leq C\sum_{N\leq |n_1|, |n_2|, |n_4|\leq M}\frac{n^2_1n_2n_4}{\langle n_1\rangle^4\langle n_2\rangle^2\langle n_4\rangle^2}=0$$
by symmetry. From \eqref{l4xn1} and \eqref{y12} we obtain \eqref{l4xn} as desired.

To prove  \eqref{lqxn} we use \eqref{gg} to define 
\begin{equation}\label{snm}
S_{M,N}(v) := | X_M(v) - X_N(v) |^2 =\sum_{N\leq |n_1|, |n_2| < M}n_1n_2\frac{|g_{n_1}(\omega)|^2|g_{n_2}(\omega)|^2}{\langle n_1\rangle^2\langle n_2\rangle^2}
\end{equation}
which fits the framework of \eqref{form} in Proposition \ref{chaos} with $k=4$. Then it follows that for  any $p\geq 2$ 
\begin{equation}\label{pchaos}
\|S_{M,N}(v)\|_{L^p}\lesssim (p-1)^2\|S_{M,N}(v)\|_{L^2}=(p-1)^2\|X_M(v) - X_N(v\|_{L^4}^2\lesssim (p-1)^2\frac{1}{N}.
\end{equation}
 On the other if we set $q=2p$, then by \eqref{pchaos} we have that  
$$\| X_M(v) - X_N(v) \|_{L^q}=\|S_{M,N}(v)\|_{L^p}^\frac{1}{2}\lesssim (q-1)\frac{1}{N^\frac{1}{2}}, $$ hence \eqref{lqxn} for $q \ge 4$.  Finally, H\"older's inequality gives the \eqref{lqxn} 
for $2\le q \le 4$.

\end{proof}

\medskip

\begin{proof} [Proof of Lemma \ref{hard1}]
Let us denote $M_N(v):=\, \int_\T |v_N|^2\,dx$. Up to absolute constants we write
\begin{eqnarray} \label{zterms}
 \rho( | K_M(v) - K_N(v) | >  2 \gamma )  &\leq&  \rho( | X_M(v) M_M(v) - X_N(v) M_N(v) | >   \gamma ) \\
 &+&  \rho( | M_M(v)^3 - M_N(v)^3 | >   \gamma ). \notag
 \end{eqnarray}
   
Then 
\begin{eqnarray*} 
\rho( | X_M(v) M_M(v) - X_N(v) M_N(v) | >   \gamma ) & \leq & \rho(  | X_M(v)- X_N(v)| \,  M_M(v)   > \frac{\gamma}{2} ) \\
&+&  \rho(  | M_M(v)- M_N(v)| \, |X_N(v)| > \frac{\gamma}{2} ) = I_1 + I_2. 
\end{eqnarray*}
Let $\lambda >0$ to be determined. Then by \eqref{xn},  \eqref{Hld} and \eqref{lpbound2} with $p=2$, $\al= \frac{1}{2}-$,  we have that
\begin{eqnarray*} 
I_2 &\leq& \rho( |X_N(v)| > \lambda ) + \rho ( | M_M(v)- M_N(v)| > \frac{\gamma}{2} \lambda^{-1} ) \\
&\leq& e^{-c \lambda N^{-2 \eps}} \, + \, \rho( \| v^N- v^M \|_{L^2} > \frac{\gamma}{4B} \lambda^{-1} ) \\
&\leq &  e^{-c \lambda N^{-2 \eps}} \, +  e^{- c_{\gamma, B} N^{1-} \lambda^{-2}}. 
\end{eqnarray*}
By setting $\lambda= N^{\frac{1}{3}+ \frac{2\eps}{3}-}$ we have that $I_2 \lesssim e^{- c_{\gamma, B} N^{\frac{1}{3} - \frac{4\eps}{3}-} }$. 

To estimate $I_1$ we first note that 
\begin{equation} \label{massbound} M_M(v) \leq \|v\|^2_{L^2} \leq B^2.
\end{equation}   Then by \eqref{lqxn} and Tchebishev's inequality\footnote{C.f. Lemma 4.5 in \cite{T}.} we have that 

\begin{equation}\label{important}
I_1 \leq \rho(|X_M(v) - X_N(v) | > \frac{\gamma}{2B^2} ) \lesssim e^{-C_BN^\frac{1}{2}\gamma}.
\end{equation}
To estimate the second term of \eqref{zterms}, we  use \eqref{massbound} to obtain
$$\rho( | M_M(v)^3 - M_N(v)^3 | >  \gamma)\leq  \rho( | M_M(v) - M_N(v) | >   c_B\, \gamma) \leq e^{-C_{B} \gamma^2 N^{1-}} $$ 
by arguing as in the estimate for $I_2$ above.

\end{proof}

\begin{lemma} \label{LEM:conv in meas}  $R_N (v)$ 
converges in measure to $R(v)$. 
\end{lemma}

\begin{proof}
If $\| P_N v\|_{L^2} \leq B$ for all $N \in \mathbb{N}$, 
then we have $\|v\|_{L^2} \leq B$.
Then, by continuity from above, we have, for $\delta \in (0, 1)$,
\begin{align*}
\lim_{N\to \infty} 
& \rho\Big( \big\{v;  |\chi_{\{ \| v^N\|_{L^2} \leq B\}}-\chi_{\{ \| v\|_{L^2} \leq B\}}| > \delta \big\}\Big)\\
&  = \lim_{N\to \infty} \rho( \| v^N\|_{L^2} \leq B)
-\rho(\| v\|_{L^2} \leq B) \\
& =  \rho\Big( \bigcap_{N= 1}^\infty \big\{ \| v^N\|_{L^2} \leq B \big\} \Big)
-\rho(\| v\|_{L^2} \leq B) 
= 0.
\end{align*}

\noi
Thus, $\chi_{\{ \| v^N \|_{L^2} \leq B\}}$ converges 
to $\chi_{\{ \| v \|_{L^2} \leq B\}}$ in measure.
By Lemma \ref{LEM:Hbound} (a), 
$F_N$ converges to $F$ in measure and by Lemma \ref{hard1} $K_N$ converges to $K$ in measure. 

Lastly, we consider $G_N(v)$ and show it is Cauchy in measure provided $\| v\|_{L^2} \le B$.  Assume $N \leq M$ then, 
 \begin{align} 
4\pi G_N(v) - 4 \pi G_M(v)
&=    \biggl(\int_{\T} |v^M|^2 - |v^N|^2 \, dx \biggr)\biggl(\int_{\T} |v^M|^4 \, dx\biggr) + \biggl(\int_\T |v^N|^2 \, dx \biggr)\biggl(\int_{\T} |v^M|^4 - |v^N|^4 \, dx\biggr) \notag \\
&\leq c_B \, \| v^M - v^N\|_{L^2} \| v^M\|_{L^4}^4 +  \| v^N\|_{L^2}^2 \, \left| \| v^M\|_{L^4}^4 - \| v^N\|_{L^4}^4 \right| \notag\\
&\leq C_B  \biggl[ \| v^M - v^N\|_{L^2} \| v^M\|_{L^4}^4 +   3 ( \| v^M\|_{L^4}^3 + \| v^N\|_{L^4}^3)\,  \| v^M - v^N\|_{L^4} \biggr].  \label{HGweak1}
\end{align} 

Fix any $\gamma> 0$; then 

\begin{eqnarray*} \rho( |4\pi G_M(v) - 4 \pi G_N(v)| > \gamma) &\leq& \rho (  \| v^M - v^N\|_{L^2}^2 \| v^M\|_{L^4}^4  > \frac{\gamma}{2C_B} ) \\
&+& \rho \biggl( ( \| v^M\|_{L^4}^3 + \| v^N\|_{L^4}^3 )   \| v^M - v^N\|_{L^4}   > \frac{\gamma}{6C_B} \biggr). 
\end{eqnarray*}

To treat the first term we write 
\begin{equation*}  \rho (  \| v^M - v^N\|_{L^2} \| v^M\|_{L^4}^4 > \frac{\gamma}{2 C_B} ) \leq   \rho (  \| v^M - v^N\|_{L^2}  >  \lambda^{-1} \frac{\gamma}{2C_B} ) +  \rho (  \| v^M\|_{L^4}^4 >  \lambda ) 
\end{equation*}
for some $\lambda>0$ to be determined. We use \eqref{Hld2} with $\alpha =\frac{1}{2}-$  corresponding to $p=2$  and  \eqref{Hld1} to get that 
$$   \rho (  \| v^M - v^N\|_{L^2}  >  c_{B} \gamma\, \lambda^{-1}  )  \leq e^{- c'_{B} \gamma^2 N^{1-}  \lambda^{-2}} $$ and 
	$$   \rho (  \| v^M\|_{L^4} >  \lambda^{\frac{1}{4}} )  \leq e^{-c  \lambda^{\frac{1}{2}} }.$$ A decay of $e^{-C_B N^{\frac{1}{5}-} \gamma^{\frac{2}{5}}}$  follows by setting $\lambda = N^{\frac{2}{5}- } \gamma^{\frac{4}{5}}$.  

For the second term write 

\begin{eqnarray*}  & & \rho\biggl( \| v^M - v^N\|_{L^4}   ( \| v^M\|_{L^4}^3 + \| v^N\|_{L^4}^3)   > \frac{\gamma}{6 C_B} \biggr) \\
&\leq & \rho ( \| v^M - v^N\|_{L^4} >  c_{B} \gamma  \lambda^{-1}  ) + \rho( \| v^M\|_{L^4} >  c_1 \, \lambda^{\frac{1}{3}} ) + \rho( \| v^N\|_{L^4}   > c_2 \lambda^{\frac{1}{3}} )  \\
&\leq& e^{- c'_{ B} \gamma^2 N^{\frac{1}{2}-}  \lambda^{-2}}  +  2  e^{-c  \lambda^{\frac{2}{3}} },  \end{eqnarray*}
since $\al = \frac{1}{4}-$ when  $p= 4$ in  \eqref{Hld2}. A decay of   $e^{-C_B N^{\frac{1}{8}-} \gamma^{\frac{1}{2}}}$ follows by setting $\lambda= N^{\frac{3}{16}-} \gamma^{\frac{3}{4}}$.

\medskip

Thus, $G_N(v)$ converges to $G(v)$ in measure and  hence, by composition and multiplication of continuous functions, $R_N(v)$ converges to $R(v)$ in measure.
\end{proof}

The following proposition shows that  the weight $R(v)$ is indeed integrable with 
respect to the Wiener measure $\rho$.

\begin{proposition}\label{PROP:weak conv} 

\noi \textup{(a)}  For sufficiently small $B>0$, we have $R(v) \in L^2 (d\rho)$. In particular, the 
weighted Wiener measure $\mu$  is a probability measure,  absolutely continuous with 
respect to the Wiener measure $\rho$. 

\noi \textup{(b)} We have the following tail estimate. Let $2 \leq r <\infty$  and  $(s-1)r < -1$;  then there exists a constant $c$ such that 
\begin{equation} \label{H large deviation}
\mu ( \|v\|_{\mathcal{F} L^{s, r}} > K) \leq e^{-cK^2}
\end{equation}
for  sufficiently large $K > 0$.

\noi \textup{(c)} The finite dimensional weighted  Wiener measure $\mu_N$
in \eqref{HGibbs3} converges weakly to $\mu$.  
\end{proposition}

\begin{proof} 
(a)\,  By H\"older inequality, we have
\begin{eqnarray*}
\int R_N^2(v) d\rho(v) 
& \leq &C_B\, \bigg( \int \chi_{\{\|v^N\|_{L^2}\leq B\}}
e^{- 3{\rm Im} \int (v^N)^2 \cj{v^N} \cj{v^N_x}\,dx }d\rho(v)\bigg)^\frac{1}{3}\\
&\times &
\bigg( \int \chi_{\{\|v^N\|_{L^2}\leq B\}}
 e^{ \frac{3B^2}{2 \pi}(\int |v^N|^4\,dx )} d\rho(v)\bigg)^\frac{1}{3} \\
 &\times& \bigg( \int \chi_{\{\|v^N\|_{L^2}\leq B\}}e^{-\frac{6}{\pi}M_N(v) {\rm Im} \int (v^N  \cj{v^N_x})\, dx} d\rho(v)\bigg)^\frac{1}{3} .
\end{eqnarray*}

\noi
It follows from Lemma 3.10 in \cite{B1} (see also \cite{LRS}) that the second factor is finite 
for any $B> 0$,  whereas it was shown in \cite[Proposition 4.2]{TTzv}  that 
the first factor is finite for sufficiently small $B>0.$ For the third factor we proceed as in the proof of \cite{TTzv} Proposition 4.2. 
In what follows we always implicitly assume that $\|v_N\|_{L^2}\leq B$. If we define
$$A_{\gamma, N}=\{\chi_{\{\|v^N\|_{L^2}\leq B\}}e^{-\frac{6}{\pi}M_N(v) {\rm Im} \int (v^N  \cj{v^N_x})\, dx} >\gamma\},$$
then we need to show that 
\begin{equation}\label{infinity}
\int_0^\infty \gamma^2\rho(A_{\gamma, N})\, d\gamma,
\end{equation}
is convergent uniformly with respect to $N$ for $B>0$ small enough. Let $N_0=\ln \gamma$ and assume first that
$N\leq N_0\leq \frac{C}{B^2}\ln \gamma$, for $B$ small enough. 
We first observe that 
$$\left |M_N(v) {\rm Im} \int (v^N  \cj{v^N_x})\,dx \right |\leq CB^2\|\partial_x(v^N)^2 \|_{L^\infty(\T)}.$$
We also note that 
\begin{equation}\label{agamma}
\rho(A_{\gamma, N})\leq \rho\left (\left |M_N(v) {\rm Im} \int (v^N  \cj{v^N_x})\,dx \right |>C\ln\gamma \right )
\end{equation}
and combining \eqref{agamma} and \eqref{infinity} with Proposition 4.1 in \cite{TTzv}, we can continue with 
$$\rho(A_{\gamma, N})\leq \rho(\|\partial_x(v^N)^2 \|_{L^\infty(\T)}>CB^{-2}\ln \gamma)\lesssim e^{-\frac{C}{B^2}\ln \gamma}=\gamma^{-\frac{C}{B^2}},$$
and the convergence of \eqref{infinity} follows from taking $B$ small enough. 

Assume now that $N> N_0=\ln \gamma$ then we observe that $A_{\gamma, N}\subset B_{\gamma, N} \cup C_{\gamma, N}$ where 
$$B_{\gamma, N}:=\{|X_{N_0}(v)| >\frac{\pi}{12B^2}\ln \gamma\},$$
$$C_{\gamma, N}:=\{|X_{N}-X_{N_0}(v)| >\frac{\pi}{12B^2}\ln \gamma\}.$$
We first observe that from the argument above 
$$\rho(B_{\gamma, N})\leq \rho(\|\partial_x(v^{N_0})^2 \|_{L^\infty(\T)}>CB^{-2}\ln \gamma)\lesssim \gamma^{-\frac{C}{B^2}}.$$
On the other hand from \eqref{important}  and the fact that $N>\ln \gamma$ we have that 
$$\rho(C_{\gamma, N})\lesssim e^{-C_BN^\frac{1}{2}\ln \gamma}\leq e^{-C_B(\ln \gamma)^{(1+\frac{1}{2})}}\leq C_{B,L}\gamma^L, $$
for any $L\geq 1$ and an appropriate constant $C_{B,L}$ depending on $B$ and $L$. From here again the convergence of \eqref{infinity} follows.

Hence we have that $R_N(v) \in L^2 (d\rho)$ for sufficiently small $B>0$, 
independent of $N$.  Then, by Lemma \ref{LEM:conv in meas} and Fatou's lemma, 
we obtain $R(v) \in L^2 (d\rho)$.

(b) \, By Cauchy-Schwarz inequality, we have
\begin{align*}
\int \chi_{\{\|v\|_{\mathcal{F} L^{s, r}} > K\}} d\mu
\leq \| R(v) \|_{L^2(d\rho)} \big\{\rho( \|v\|_{\mathcal{F} L^{s, r}} > K) \big\}^\frac{1}{2}.
\end{align*}

\noi 
Then, 
\eqref{H large deviation} follows from \eqref{Hld}.

(c) \,   Let us define 
\begin{equation} \label{finitemodes}
\mathcal{H}\,:=\, \bigcup_{M} \left\{  F \,; F=G( \ft{v}_{-M}, \cdots, \ft{v}_M), \,\, 
G {\rm~bounded~and~continuous}\right\} \,.
\end{equation}  Note this is a dense set  in $L^1 (\FF L^{s,r}, \mu)$ with $2 \leq r <\infty$  and $(s-1) r < -1$. Fix 
$F \in \mathcal{H}$,  then $F$ depends on $M$ finitely many modes, for some $M$.  Fix $\eps > 0$. Then, for $N > M$, we have
\begin{align*}
 \bigg| \int F(v) & d \mu_N - \int F(v) d\mu \bigg|
 = \bigg|\int F(v) \big( R_N(v) - R(v) \big) d\rho \bigg|\\
& \leq \bigg|\int_{\{|R_N(v) - R(v)| < \eps\}}
F(v) \big( R_N(v) - R(v) \big) d\rho \bigg| \\
& \hphantom{XXXXX}
 + \bigg|\int_{\{|R_N(v) - R(v)| \geq \eps\}} F(v) \big( R_N(v) - R(v) \big) d\rho \bigg|\\
& \leq \eps \sup |F| + 
\sup|F|\, \|R_N(v) - R(v)\|_{L^2(d\rho)} \big\{\rho( |R_N(v) - R(v)| \geq \eps)\big\}^\frac{1}{2}.
\end{align*}
 
\noi 
 From the proof of Proposition \ref{PROP:weak conv}, 
 we have $\|R_N(v) - R(v)\|_{L^2(d\rho)}
 \leq  \|R_N(v)\|_{L^2(d\rho)}+ \|R(v)\|_{L^2(d\rho)} < C <\infty$ for all $N$.
 By Lemma \ref{LEM:conv in meas}, $\rho(|R_N(v) - R(v)| \geq \eps) \to 0$ as $n \to \infty$.

Now, let $F$ be a general bounded continuous function on $\mathcal{F}L^{s, r}$ with $2 \leq r <\infty$  and $(s-1) r < -1$. Let $F_M$ denote its restriction on $E_M$, 
i.e. $F_M(v) = F(v^M)$ where $v^M = P_M v$.
By Cauchy-Schwarz inequality, we have
\begin{align} \label{H dense}
\bigg|\int F(v) d\mu - \int F_M(v) d\mu \bigg|
= \bigg| \int \big(F (v) - F(v^M)\big) R(v) d \rho\bigg|\\
\leq \|R(v)\|_{L^2(d\rho)} \bigg( \int |F(v) - F(v^M)|^2 d \rho\bigg)^\frac{1}{2}. \notag 
\end{align}

\noi
By continuity of $F$, given $\eps > 0$, 
there exists $\dl > 0$ such that 
\[\| P_M^\perp v \|_{\mathcal{F}L^{s, r}} = \|v - v^M\|_{\mathcal{F}L^{s, r}} <\dl
\ \ \LRA \ \ |F(v) - F(v^M)| < \eps.\]

\noi
Then, the contribution in \eqref{H dense} 
from $\{ \,v\, ; \| P_M^\perp v \|_{\mathcal{F}L^{s, r}}< \dl\}$
is at most $\eps \, \|R(v)\|_{L^2(d\rho)}$.
Without loss of generality, assume $\dl \leq \eps^2$.
By the measurability of the $\mathcal{F}L^{s, r}$-norm (see Definition \ref{H measurable}), 
the contribution in \eqref{H dense} 
from $\{\, v\, ; \| P_M^\perp v \|_{\mathcal{F}L^{s, r}} \geq \dl\}$
is at most
\begin{align*}
2 \sup|F| \cdot &  \|R(v)\|_{L^2(d\rho)}
 \big\{\rho(\| P_M^\perp v  \|_{\mathcal{F}L^{s, r}} \geq \dl) \big\}^\frac{1}{2}\\
& \leq 2 \sup|F| \cdot \|R(v)\|_{L^2(d\rho)} \, \dl^\frac{1}{2}
 \leq 2 \sup|F| \cdot \|R(v)\|_{L^2(d\rho)}\,  \eps
\end{align*}

\noi
for sufficiently large $M$.
A similar argument can be used to show 
$|\int F(v) d\mu_N - \int F_M(v) d\mu_N | \leq C(f, R) \eps$,
independent of $N$.
Hence, $\mu_N$ converges weakly to $\mu$.  
\end{proof}
\medskip

\begin{remark}  \rm A tail estimate similar to \eqref{H large deviation}
holds for the finite dimensional weighted Wiener measure $\mu_N$; 
i.e. we have
\begin{equation} \label{H large deviation 2}
\mu_N ( \|v^N\|_{\mathcal{F} L^{s, r}} > K) \leq e^{-cK^2},
\end{equation}
\noi
where the constant is independent of $N$.
\end{remark}

\medskip

\begin{remark} \label{inverseR}  \rm  The measure $\rho_N$  is not absolutely 
continuous  with respect to $\mu_N$ but its restriction on $\{\|v^N\|_{L^2}\leq B\}$, i.e.,  
$\wt{\rho}_N = \ft{Z}_N^{-1}\chi_{\{\|v^N\|_{L^2}\leq B\}}\rho_N$ is absolutely continuous 
with respect to $\mu_N$  and  from \eqref{HGibbs3},   we  have that 
\begin{equation} \label{rhomu} \frac{d \wt{\rho}_N}{d \mu_N} :=  
\tilde R_N =  {\tilde Z_N}^{-1} \chi_{\{ \| v^N \|_{L^2} \leq B \}}  e^{\frac{1}{2} \mathcal N(v^N)} 
\end{equation} for suitable renormalization $ \tilde Z_N$. 
Since $\mathcal N(v^N)$ does not have a definite sign Lemma  \ref{LEM:Hbound}  
Lemma \ref{LEM:conv in meas} and part (a) of Proposition \ref{PROP:weak conv} hold for  
$\tilde R_N$ and its corresponding limit $\tilde R$.  In particular, for sufficiently small B, 
$ \tilde R_N \in L^2(d \rho)$ for all $N$ with bound independent of $N$. The latter fact 
will be used in the proof of Proposition \ref{aasgwp} in Section 6.
\end{remark}

\medskip 

\begin{remark} \rm
Given any $ p < \infty$, one can prove $R(v) \in L^p(d\rho)$ 
for sufficiently small $B \leq B(p)$.
However, $B(p) \to 0$ as $p \to \infty$.
i.e. there is no uniform lower bound on the size of the $L^2$-cutoff.
For our purpose, the integrability with $p = 2$ suffices.
\end{remark}

\medskip

\section{ Almost sure  well-posedness of  FGDNLS and invariance of the measure}

\medskip

In order to prove the global well-posedness of $\mu$-almost all solution of  (FGDNLS) \eqref{FGDNLS} 
we fix once again $s=\frac{2}{3}-$ and $r=3$ so that we have at our disposal the local well posedness result  in $\mathcal{F} L^{s, r}$, that the measure is supported on $\mathcal{F} L^{s, r}$ and also the energy growth estimates in Theorem \ref{Energy Growth Estimate} as explained in Remark \ref{fixchoice}.
   
We first use the almost invariance of the finite-dimensional measure $\mu_N$ under the flow 
of  the truncated equation \eqref{FGDNLS} to control the growth of solutions. 

\medskip

\begin{lemma} \label{Local1}   For any given  
$T>0$ and $\eps > 0$ there exists  an integer $N_0 = N_0(T, \eps)$ and 
sets $\wt{\Omega}_N=\wt{\Omega}_N(\eps, T) \subset \mathbb{R}^{4N+2}$ such that 
for $N >  N_0$ 

\smallskip \noindent
\textup{(a)} ${\displaystyle  \mu_N \left( \wt{\Omega}_N \right) \,\ge\, 1-  \eps \,.}$

\smallskip 
\noindent
\textup{(b)} For any initial condition $v^N_0 \in  \wt{\Omega}_N$,  (FGDNLS) \eqref{FGDNLS} is well-posed 
on $[-T, T]$ and its solution $v^N(t)$ satisfies the  bound 
$$ 
\sup_{|t|\le T} \| v^N(t) \|_{\mathcal{F} L^{\frac{2}{3}-, 3}} \lesssim \bigg( \log \frac{T}{\eps}\bigg)^\frac{1}{2}\,. 
$$
\end{lemma}

\begin{proof}  It is enough to consider $t\in [0,T]$, the argument is similar for $t \in [-T,0]$.  
We set 
$$
C_{N}(K,B) \,:=\, \left\{ w^N \in \mathbb{R}^{4N+2}\,:\,  \|w^N\|_{\mathcal{F} L^{\frac{2}{3}-, 3}} \leq K \,,\, \|w^N\|_{L^2} \le B \right\} \,.
$$  
If  the initial condition $v^N_0 \in C_{N}(K,B)$ then (FGDNLS) \eqref{FGDNLS} is locally  well-posed on the 
time interval of length   $\dl \sim K^{-\gamma}$ by Theorem \ref{lwp} ,  where $\gamma > 0$ 
is independent of $N$.  Furthermore, if $\mu_N$ is given by  \eqref{HGibbs3}, then for sufficiently  large $K$ we have that  $\mu_N(C_N(K,B)^c) \leq e^{-cK^2}$   
for some constant $c$ which is independent of $N$  by  
\eqref{H large deviation 2}.

Let $\Phi_N(t)$ the flow map of \eqref{FGDNLS}. We define $\wt{\Omega}_N$ by
\begin{equation*}
\wt{\Omega}_N \,:=\, \left\{ v_0^N \,:\,   \Phi_N(j \delta)(v_0^N) \in C_{N}(K,B) \,,\, j=0,1, 
\cdots  \left[\frac{T}{\dl}\right] \right\}. \end{equation*}
Note that  
$\wt{\Omega}_N^c \,=\, \bigcup_{k=0}^{\left[\frac{T}{\dl}\right]}  D_k$ \,,
where
\begin{eqnarray}
 D_k \,&=&\, \left\{ v_0^N \,;\, k \,=\,  \min \left\{j \,:\,  \Phi_N(j \dl)(v_0^N) \in C_{N}(K,B)^c \right\} \right\} \,, 
 \nonumber  \\
\,&=&\,   \left[  \bigcap_{j=0}^{k-1} \Phi_N(-j \dl) C_N(K,B)) \right]   \cap \Phi_N(-k \dl) ( C_N(B,K)^c ) \,.
\end{eqnarray} 
One verifies easily that the sets $D_k$ satisfy 
\begin{equation}\label{dkform}
D_0\,=\, C_{N}(K,B)^c \,, \quad D_k \,=\, C_{N}(K,B) \cap \Phi_N(-\delta) (D_{k-1}) \,.
\end{equation}

By Lemma \ref{invlebesgue}, the Lebesgue measure $d\mu^0_N \equiv \prod_{|n|\le N} da_n db_n$ 
is invariant under the flow $\Phi_N(t)$ (i.e. for any $f\in L^1(d\mu_N^0)$ we have 
$\int f \circ \Phi_N(t) d\mu_N^0 =  \int f d\mu_N^0$).  

Using the energy growth estimate\footnote{ Without loss of generality we assume 
$\max(K^6, K^8) = K^8$ in Theorem \ref{Energy Growth Estimate}.} in Theorem 
\ref{Energy Growth Estimate} and the invariance of the $L^2$ norm 
$m(v)=\frac{1}{2\pi}\|v\|_{L^2}$   under $\Phi_N(t)$ 
(i.e. $m \circ \Phi_N(t)= m$ for all $t$; see Remark \ref{alsoconserved})  
we have for any set  $A\subset \mathbb{R}^{4N+2}$
\begin{eqnarray} 
\mu_N ( C_{N}(K,B) \cap A) \,&=&\,  Z_N^{-1} \int  \chi_{ \{C_{N}(K,B) \cap A\}}  \chi_{\{ m \le 2\pi B^2\}} 
e^{ - \frac{1}{2}\EE -\pi m}  d\mu_N^0 \nonumber \\ 
\,&=&\, Z_N^{-1} \int  \chi_{ \{C_{N}(K,B) \cap A\}} \circ \Phi_N(-\delta) \chi_{\{ m \le 2\pi  B^2\}}   e^{ - \EE \circ \Phi_N(-\delta) - 
\pi m } d\mu_N^0 \nonumber \\
\,&=&\,  \int  \chi_{ \{ \Phi_N(\delta)(C_{N}(K,B) \cap A) \} }  e^{  -\frac{1}{2}( \EE \circ \Phi_N(-\delta)-\EE)} d\mu_N  \label{simplebound}  \\
\,&\le&\,   e^{c(\dl) N^{-\beta}  K^8} \mu_N\left( \Phi_N(\delta)(C_N(K,B)\cap A) \right) \nonumber \\
\,&\le&\,   e^{c (\dl) N^{-\beta}  K^8} \mu_N\left( \Phi_N(\delta)( A) \right). \nonumber
\end{eqnarray}

Applying \eqref{simplebound} to \eqref{dkform} with $A= \Phi_N(-\delta) (D_{k-1})$ and iterating in $k \in \{0, \dots [\frac{T}{\delta}]\}$,  we obtain 
$$\mu_N(D_K) \le e^{c(\dl) N^{-\beta}  K^8} \mu_N( D_{K-1}) \le e^{k \, c( \dl)  N^{-\beta}  K^8} e^{-cK^2}$$ 
and thus 
\begin{align} \label{Bsum}
\mu_N(\wt{\Omega}_N^c) \le \sum_{k = 0}^{[\frac{T}{\dl}]}
e^{k \, c(\dl) N^{-\beta}  K^8}e^{ -cK^2}
\lesssim \Big[\frac{T}{\dl}\Big]e^{ -cK^2} \sim T K^\gamma e^{ -cK^2} \,,
\end{align}
\noi
for $N \geq N_0(T, K)$. By choosing $K \sim \big( \log \frac{T}{\eps}\big)^\frac{1}{2}$,
we have $\mu_N(\wt{\Omega}_N^c) < \eps$ as desired.

Finally, by construction, we have
$\|v^N(j \dl) \|_{\mathcal{F} L^{\frac{2}{3}-, 3}} \leq K$
for $j = 0, \cdots,  [\frac{T}{\dl}]$ and by the local theory, we have 
\[ \sup_{0\le t \le T} \|v^N(t) \|_{\mathcal{F} L^{\frac{2}{3}-, 3}} \leq 2 K \sim 
\bigg( \log \frac{T}{\eps}\bigg)^\frac{1}{2}.\] \,
\end{proof}

Combining Lemma \ref{Local1} with the approximation Lemma \ref{approximation} we 
can now prove a similar result for the solution of the initial value problem  (GDNLS) \eqref{GDNLS}. 

\begin{proposition}\label{aasgwp}  For any given $T>0$ and $\eps > 0$ there exists 
a set $\Omega(\varepsilon,T)$ such that 

\smallskip \noindent
\textup{(a)} ${\displaystyle \mu \left( \Omega(\eps, T) \right) \,\ge\, 1- \eps \,.}$

\smallskip 
\noindent
\textup{(b)} For any initial condition $v_0 \in  \Omega(\eps, T)$  the initial value problem (GDNLS) \eqref{GDNLS} is well-posed 
on $[-T, T]$ with the  bound 
$$ 
\sup_{|t|\le T} \| v(t) \|_{\FF L^{\frac{2}{3}-,3}} \lesssim \bigg( \log \frac{T}{\eps}\bigg)^\frac{1}{2}\,. 
$$
\end{proposition}

\begin{proof}  Let $\wt{\Omega}_N=\wt{\Omega}_N(\eps, T)$ be 
the set given in Lemma \ref{Local1} for $N\ge N_0(\eps, T)$.  This set is defined  in terms of  
$K \sim \left(\log\frac{T}{\eps}\right)^{1/2}$ and for that same $K$ we define the set  
$$
\Omega_N := \Omega_N( \eps, T) \,:=\,\left\{ v_0 \in \FF L^{\frac{2}{3}-,3} \,:\,   \| v_0 \|_{\FF L^{\frac{2}{3}-,3}} \le K  \,,\,\,\,  
P_N v_0 \in \wt{\Omega}_N \right\}
$$
If $v_0 \in \Omega_N$ then by Lemma \ref{Local1} we have 
\begin{equation} \label{iteratebound}
\sup_{t \le T} \|  \Phi_N(t)( P_Nv_0) \|_{\FF L^{\frac{2}{3}-,3} } \le 2 K \,.
\end{equation}
On the other hand for $v_0 \in \Omega_N$ the local well posedness theorem in \cite{GH} gives a $\delta >0$ and a solution $v(t)$ of (GDNLS) \eqref{GDNLS} for $|t| \leq \delta$. 

By \eqref{loc} in the proof of the Lemma \ref{approximation}, with $K$ in place of $A$, we obtain that for every $s_1<\frac{2}{3}-$
$$
\|v(\delta)-v^N(\delta)\|_{\FF L^{s_1,3}}\lesssim KN^{s_1-\frac{2}{3}+}. 
$$
By choosing  a larger $N_0$ if necessary, so that $\left[\frac{T}{\delta}\right] \ KN^{s_1-\frac{2}{3}+}\ll 1$ for $N> N_0$ we can repeat this argument $\left[\frac{T}{\delta}\right]$ times over the intervals 
$[j\delta,(j+1)\delta], j=0,1,\dots, \left[\frac{T}{\delta}\right]-1$ and obtain 
\begin{equation}\label{localtilde}
\|v(j\delta)-v^N(j\delta)\|_{\FF L^{s_1,3}} < 1. 
\end{equation}
Then from \eqref{iteratebound} and \eqref{localtilde} we conclude
$$\|v(t)\|_{\FF L^{s_1,3}}\lesssim (2 K+1) \sim \left(\log \frac{T}{\varepsilon}\right)^{\frac{1}{2}},$$
and since the right hand side is independent of $s_1<\frac{2}{3}-$  we obtained the desired estimate.

To estimate $\mu(\Omega_N)$ note  first that 
\begin{equation}\label{omegac}
\Omega_N^c \,\subset \,  \left\{   v_0 \in {\FF L^{\frac{2}{3}-,3}} \,: \|v_0\|_{\FF L^{\frac{2}{3}-,3}} \ge K \right\} \cup 
\left \{  v_0 \in {\FF L^{\frac{2}{3}-,3}}  \,:\,  P_Nv_0    \in  {\tilde \Omega}_N^c  \right\}
\end{equation}
The first set on the right hand side of \eqref{omegac} has $\mu$ measure less than $\eps$ by the tail bound 
in Proposition \ref{PROP:weak conv}.  
The set $F_N \,\equiv\, \left \{  v_0 \in {\FF L^{\frac{2}{3}-,3} } \,;  P_N v_0    \in  {\tilde \Omega}_N^c  \right\}$ 
is  a cylinder set  and we have $F_N \cap E_N = \wt{\Omega}^c_N$  (recall 
$E_N ={\rm span} \{e^{inx}\}_{|n|\le N}$).  Thus $\rho(F_N)= \rho_N(F_N) =
\rho_N(\wt{\Omega}^c_N)$.  On the other hand, recall that $\mu \ll \rho$ and that, 
$\wt{\rho}_N$  the  restriction of  $\rho_N$ to the ball $\{\|v^N\|_{L^2} \le B\}$ is 
absolutely continuous with respect to $\mu_N$ (see Remark \ref{inverseR}).  Then using  
Cauchy-Schwarz repeatedly we obtain
\begin{eqnarray}
\mu( F_N) & \le & \left( \int R^2 d\rho \right)^{\frac{1}{2}}  \left(\int_{\wt{\Omega}^c_N} \,  \chi_{\{\|v^N\|_{L^2} \le B\}} \,  d\rho_N  \right)^\frac{1}{2} \notag \\
&\le &   \left( \int R^2 d\rho \right)^{\frac{1}{2}}  \left( \int \tilde R_N^{2}  d\mu_N \right)^{\frac{1}{4} } \mu_N(\wt{\Omega}_N^c)^{\frac{1}{4}} \notag\\
& \le& \,   \left( \int R^2 d\rho \right)^{\frac{1}{2}}  \left( \int \tilde R_N d\rho_N \right)^{\frac{1}{4}}  \mu_N(\wt{\Omega}_N^c)^{\frac{1}{4}} \label{lastbound}
\end{eqnarray} where $\tilde R_N$ is as defined in Remark \ref{inverseR} and where  in the last inequality we have used that by definition $\tilde R_N^2 R_N = \tilde R_N$. 

By relying on Lemma \ref{LEM:conv in meas}, Proposition \ref{PROP:weak conv} and Remark \ref{inverseR}  we can bound the first two terms in \eqref{lastbound} by a constant independent of $N$. This combined with Lemma \ref{Local1} allows us to conclude that there exist a constant $d>0$
and $N_1(\eps, T)$ such that $\mu( F_N) \le d \, \varepsilon$ for $N \ge N_1$.   So for $N \geq \max{(N_0, N_1)}$, any set $\Omega(\eps,T) := \Omega_N(\eps, T)$ satisfies the 
desired hypothesis. 
\end{proof}

\begin{theorem}[Almost sure global well-posedness]\label{asgwp}
There exists a subset $\Omega$ of the space $\mathcal FL^{\frac{2}{3}-, 3}$ with 
$\mu(\Omega^{c})=0$ such that for every $v_0\in \Omega$ the initial value problem 
(GDNLS) \eqref{GDNLS} with initial data $v_0$ is globally well-posed.
\end{theorem}

\begin{proof}  Fix an arbitrary $T$ and let $\eps = 2^{-i}$.  Using the sets given in 
Proposition  \ref{aasgwp}  we set 
\begin{equation}\label{fullset1}
\Omega(T) := \bigcup_i \Omega(2^{-i}, T).
\end{equation}
If $v_0 \in \Omega(T)$ then the initial value problem (GDNLS) \eqref{GDNLS} is 
well-posed up to time $T$.  Since $\mu(\Omega(T)) \ge 1 - 2^{-i}$ for any $i \in \mathbb N$, 
the set $\Omega(T)$ has full measure. 
 
Finally by taking $T:=  2^j$  the set 
\begin{equation}\label{fullset2}
\Omega = \bigcap_j \Omega(2^j)
\end{equation}
has also full measure and if $v_0 \in \Omega$ then the initial value problem (GDNLS) \eqref{GDNLS} is globally well-posed. 
\end{proof}

\begin{remark} \rm
We note that by slightly modifying the proof of Theorem \ref{asgwp} above we could also derive a logarithmic bound in time on solutions similar to the one in \cite{B1} and \cite{BT1}.

\end{remark}

Now that we have a well-defined flow on the measure space $(\FF L^{\frac{2}{3}-,3} , \mu)$ we  show 
that $\mu$ is invariant under the flow $\Phi(t)$,  following the argument in \cite{RT}. 

\begin{theorem}\label{gauged invariance measure} The measure $\mu$ is invariant under the flow $\Phi(t)$. 
\end{theorem}

\begin{proof}  Let us consider the measure space 
$(\FF L^{\frac{2}{3}-,3}, \mu)$. 
We need to show that for any measurable $A$ we have $\mu(A)= \mu(\Phi(-t)(A))$ for 
all $t\in \R$. Note that  by the group property of the flow without loss of generality we can assume that 
$|t|\le \delta$.  An equivalent characterization of invariance is that for all $F \in L^1 (\FF L^{\frac{2}{3}-,3}, \mu)$ 
we have 
\begin{equation}\label{inv0}
\int F( \Phi(t) (v) ) d\mu \,=\, \int F(v) d\mu \,.
\end{equation}
By an elementary approximation argument it is enough to show \eqref{inv0} for $F$ in a dense set 
in $L^1 (\FF L^{\frac{2}{3}-,3}, \mu)$ which we choose as in \eqref{finitemodes} to be 
$$
\mathcal{H}\,:= \, \bigcup_{M} \left\{  F \,: F=G( \ft{v}_{-M}, \cdots, \ft{v}_M), \,\, 
G {\rm~bounded~and~continuous}\right\} \,.
$$ 
For $F \in \mathcal{H}$ let us choose an arbitrary $\epsilon>0$ and assume $N \ge M$.    
By Proposition \ref{PROP:weak conv}   $\mu_N$ converges weakly to $\mu$ and thus
\begin{equation}\label{inv01}
\left| \int F   d \mu - \int F  d\mu_N \right|  +  \left| \int F \circ \Phi(t)  d \mu - \int F \circ \Phi(t) d\mu_N \right|  \le  \epsilon \,.
\end{equation}

Let  $\Phi_N(t)$ be the flow map for FGDNLS \eqref{FGDNLS}. For  $s_1<  \frac{2}{3}-$,  
by the Lemma \ref{approximation}, we have that $\| \Phi(t)(v) -\Phi_N(t)(v)\|_{\FF L^{s_1,3}}$   converges to $0$ uniformly  on $\{ v \,;\, \ \|v\|_{\FF L^{\frac{2}{3}-,3}} \le K\}$. 
Using the tail estimate $\mu_N( \|v^N\|_{\FF L^{\frac{2}{3}-,3}} > K)  \le e^{-cK^2}$  (uniformly in $N$)  and the continuity 
of $F$ in ${\FF L^{s_1,3}}$ we obtain
\begin{equation}\label{inv02}
\left| \int F \circ \Phi(t)  d \mu_N - \int F \circ \Phi_N(t) d\mu_N \right|  \le  2 \|F\|_\infty e^{-cK^2}  + \epsilon
\le 3 \epsilon 
\end{equation} 
for large enough $K$ and $N$.  

Finally using again the tail estimate for $\mu_N$, the invariance of Lebesgue measure under $\Phi_N(t)$  
and the energy estimate given in Theorem \ref{Energy Growth Estimate} we obtain
\begin{eqnarray}
&& \left| \int F \circ \Phi_N(t)  \,d\mu_N - \int F \, d\mu_N \right| \nonumber  \\
&& \,\le\, 2 \|F\|_{L^\infty} e^{-c K^2} +  \left| \int_{\{ \|v\|_{\FF L^{\frac{2}{3}-,3}}  \le K\}}  F \left[e^{- \frac{1}{2}\left( \EE \circ\Phi_N(-t) - \EE\right)} - 1 \right]  d\mu_N \right| \nonumber \\
&& \,\le\,  2 \epsilon+ \|F\|_{L^\infty}  \left( e^{ c(\delta)  N^{-\beta} K^8} -1\right) \le 3\epsilon\,, 
\label{inv03}
\end{eqnarray}
for sufficiently large $N$.  By combining \eqref{inv01}, \eqref{inv02}, and \eqref{inv03} we obtain 
invariance. 
\end{proof}

\section{The ungauged DNLS equation}

Recall that if $u(t,x)$ is a solution of  DNLS \eqref{DNLS} then  $w(t,x)=  G(u(t,x))$ where 
$G(f)(x) = \exp(-i J(f)) \, f(x) \,$ (see \eqref{expgauge})   
is a solution of 
\begin{equation}\label{GDNLS+} w_t - i w_{xx} - 2 m(w) w_x = 
- w^2 {\cj w}_x + \frac{i}{2} |w|^4 w  - i  \psi(w) w - i m(w) |w|^2 w 
\end{equation} 
with initial data $w(0)=  G(u(0))$.  Furthermore 
$v(t,x) =  w(t, x - 2 tm(w))$ is a solution of  \eqref{GDNLS}
with initial condition $v(0)=w(0)$.   If $\Phi(t)$ denotes the flow map
for GDNLS \eqref{GDNLS}, let   $\widetilde{ \Phi} (t)$ denote the flow map
of \eqref{GDNLS+} and let $\Psi(t)$ denote the flow map
of \eqref{DNLS}.

Clearly we have the relation 
\begin{equation}\label{conjugation}
\Psi(t)  =   G^{-1} \circ \widetilde{\Phi}(t) \circ G \,.
\end{equation}
To elucidate the relation between $\Phi(t)$ and $\widetilde{\Phi}(t)$ 
let $\tau_\alpha(s)$ denote the action of the group of  
spatial  translations on functions, i.e.,  $(\tau_\alpha(s) w)(x) \,: =\,  
w( x - \alpha s) \,. $    We define a state dependent translation 
$$
(\Gamma(s)w) (x) : = ( \tau_{2 m(w)}(s) w)(x) = w(x - 2\,s\, m(w) ) \,.
$$
Note the $H^s$, $L^p$ and $\mathcal F L^{s,r}$ norms are all invariant 
under this transformation.  Furthermore we have 
$$
v(t,x) \,:=\, (\Gamma(t)  w)(t,x)  \,.
$$
Since $m$ is preserved under $G$,  $\Gamma(s)$ and both flows 
$\Psi(t)$ and $\tilde{\Phi}(t)$ we have the relation
\begin{equation}\label{commute}
\Phi(t) \,=\, \Gamma (t) \widetilde{\Phi}(t) = \widetilde{\Phi}(t) \Gamma(t) \,,
\end{equation}
in particular $\widetilde{\Phi}(t)$ and $\Gamma(t)$ commute.

Finally  if $\mu$ is a measure on $\Omega$ as in Theorem \ref{asgwp} and $\varphi: \Omega \to \Omega$ is a measurable map then 
we define the measure $\nu = \mu \circ \varphi^{-1}$ by 
$$
\nu(A) \,:= \, \mu( \varphi^{-1}(A) ) \,=\, \mu( \{ x \,;\, \varphi(x) \in A\}) \,.
$$
for all measurable sets $A$ or equivalently by 
$$
\int F  d\nu = \int  F \circ \varphi \, d\mu \,
$$
for integrable $F$. 

Consider the measure defined by 
\begin{equation}\label{numeasure}
\nu := \mu \circ G.
\end{equation}
Since the measure $\mu$ constructed in Proposition \ref{PROP:weak conv} 
is invariant under the flow $\Phi(t)$ we show that the flow $\Psi(t)$ for  DNLS is well defined $\nu$ almost surely
and that $\nu$ is invariant under the flow $\Psi(t)$.

\begin{theorem}[Almost sure global well-posedness for DNLS]\label{finalasgwp}
There exists a subset $\Sigma$ of the space $\mathcal F L^{\frac{2}{3}-, 3}$ with $\nu(\Sigma^{c})=0$ such that for every $u_0\in \Sigma$ the IVP (DNLS)
\eqref{DNLS} with initial data $u_0$ is globally well-posed.
\end{theorem}

\begin{proof}  

Let $\Omega$ be the set of full $\mu$ measure given in Theorem \ref{asgwp} and let $\Sigma= G^{-1}(\Omega)$. 
Note that $\Sigma$ is a set of full $\nu$-measure by \eqref{numeasure}.  For $v_0 \in \Omega$ the IVP (GDNLS)
\eqref{GDNLS} with initial data $v_0$ is globally well-posed.  Hence since the map 
${\mathcal{G}} :  C([-T, T]; \FF L^{s,r}) \to   C([-T, T]; \FF L^{s,r})$ is a homeomorphism if 
$s > \frac{1}{2} -\frac{1}{r}$ when $2 < r <\infty$  the IVP (DNLS)
\eqref{DNLS} with initial data $u_0=G^{-1}(v_0)$ is also globally well-posed. 

\end{proof}

Finally we show that the measure $\nu$ is invariant under the flow map of DNLS \eqref{DNLS}.

\begin{theorem} 
The measure $\nu= \mu \circ G$  is invariant under the flow $\Psi(t)$. 
\end{theorem}

\begin{proof}   First we note that the measure $\mu$ is invariant under $\Gamma(t)$. The density of $\mu$
 with respect  to $\rho$ is $R(v)$, see \eqref{Hweight}, and it is verified easily that $R \circ \Gamma(t) = R$.  Furthermore one also 
 verifies easily that the finite-dimensional measures $\rho_N$ are also invariant under $\Gamma(t)$. 
 As a consequence since $\mu$ is invariant under $\Phi(t)$  by Theorem  
 \ref{gauged invariance measure} then $\mu$ is also invariant under $\tilde{\Phi}(t)$ because
 of  \eqref{commute}.  Finally $\nu$ is invariant under $\Psi(t)$ since by \eqref{conjugation}
 $$
 \int  F \circ \Psi(t) \, d \nu \,=\,  \int  F \circ G^{-1} \circ \tilde{\Phi}(t) \circ G  \, d \mu \circ G \,=\,  \int  F \circ G^{-1} \circ  \tilde{\Phi}(t)   d \mu   \,=\,  \int  F \circ G^{-1}  d \mu  \,=\, \int F d \nu.
 $$
\end{proof}

\end{document}